\documentclass[11pt]{article}

\usepackage{amssymb}
\usepackage{amsmath,amsfonts,mathtools}
\usepackage{amsthm}
\usepackage{latexsym}
\usepackage{mathrsfs}
\usepackage{color}
\usepackage{float}
\usepackage{enumitem}
\usepackage{algorithm,algorithmic}
\usepackage{graphicx}
\usepackage{subcaption}
 \allowdisplaybreaks[3]
 
\usepackage{bbm}
\usepackage{microtype}
\usepackage[a4paper]{geometry}
\usepackage{pifont}
\usepackage{latexsym}
\usepackage{url}
\usepackage[table]{xcolor}

\usepackage[a4paper]{geometry}
\newcommand{\beq}{\begin{equation}}
\newcommand{\eeq}{\end{equation}}
\newcommand{\beqa}{\begin{eqnarray}}
\newcommand{\eeqa}{\end{eqnarray}}
\newcommand{\beqas}{\begin{eqnarray*}}
\newcommand{\eeqas}{\end{eqnarray*}}
\newcommand{\ba}{\begin{array}}
\newcommand{\ea}{\end{array}}
\newcommand{\bi}{\begin{itemize}}
\newcommand{\ei}{\end{itemize}}
\newcommand{\nn}{\nonumber}

\newcommand{\cX}{{\mathcal X}}
\newcommand{\cL}{{\mathcal L}}

\newcommand{\prox}{\mathrm{prox}}

\newcommand{\argmin}{\arg\min}

\newtheorem{lemma}{Lemma}
\newtheorem{thm}{Theorem}

\newtheorem{assumption}{Assumption}

\newtheorem{rem}{Remark}

\newcounter{spb}
\def\cO{{\cal O}}

\def\tf{{\tilde f}}

\def\tx{{\tilde x}}

\def\tC{{C}}

\def\tr0{{\tilde r_0}}

\def\rr{{\mathbb{R}}}

\def\bbE{{\mathbb{E}}}

\def\barf{{\bar f}}

\def\mcO{{\mathcal O}}

\title{First-order methods for stochastic and finite-sum convex optimization with deterministic constraints}

\author{
Zhaosong Lu
\thanks{
Department of Industrial and Systems Engineering, University of Minnesota, USA (email: {\tt zhaosong@umn.edu}, {\tt xiao0414@umn.edu}). This work was partially supported by the Office of Naval Research under Award N00014-24-1-2702,  the Air Force Office of Scientific Research under Award FA9550-24-1-0343, and the National Science Foundation under Awards 2211491 and 2435911.
}
\and
Yifeng Xiao
\footnotemark[1]
}
\date{June 25, 2025 (Revised: October 29, 2026)}

\begin{document}
\maketitle

\begin{abstract}
In this paper, we study a class of stochastic and finite-sum convex optimization problems with deterministic constraints. Existing methods typically aim to find an $\epsilon$-\textit{expected feasible stochastic optimal} solution, in which both the expected constraint violation and the expected optimality gap are within a prescribed tolerance $\epsilon$. However, in many practical applications, constraints must be nearly satisfied with certainty, rendering such solutions potentially unsuitable due to the risk of substantial violations. To address this issue, we propose stochastic first-order methods for finding an $\epsilon$-\textit{surely feasible stochastic optimal} ($\epsilon$-SFSO) solution, where the constraint violation is deterministically bounded by $\epsilon$ and the expected optimality gap is at most $\epsilon$. Our methods apply an accelerated stochastic gradient (ASG) scheme or a modified variance-reduced ASG scheme \emph{only once} to a sequence of quadratic penalty subproblems with appropriately chosen penalty parameters. We establish first-order oracle complexity bounds for the proposed methods in computing an $\epsilon$-SFSO solution.  Preliminary numerical results are presented to illustrate the performance of the proposed methods.
\end{abstract}


\medskip


\section{Introduction} \label{introduction}

In this paper, we consider constrained stochastic convex optimization problems of the form 
\beq\label{prob1}
\begin{aligned}
F^*=\min_x & \ \, \{F(x):=\bbE[\tilde f(x,\xi)]+\psi(x)\} \\ 
\mbox{s.t.}& \ \, c(x)\leq0,
\end{aligned}
\eeq
where $\xi$ is a random variable with sample space $\Xi$, $\tilde{f}(\cdot, \xi)$ is a continuously differentiable convex function on $\rr^n$ for each $\xi \in \Xi$, $\psi: \rr^n \rightarrow \rr \cup \{+\infty\}$ is a proper closed convex function with an exactly evaluable proximal operator and domain $\cX$, and $c = (c_1, \ldots, c_m)$ is a Lipschitz smooth \textit{deterministic} mapping with each component $c_i$ being 
convex.\footnote{For simplicity, we focus on problem \eqref{prob1} with convex inequality constraints only. However, our algorithms and results can be directly extended to problems that include both affine equality constraints and convex inequality constraints.}

Problem \eqref{prob1} arises in a variety of important areas, including energy systems \cite{shabazbegian2020stochastic}, healthcare \cite{shehadeh2022stochastic}, image processing \cite{luke2020proximal}, machine learning \cite{cuomo2022scientific,karniadakis2021physics},  network optimization \cite{bertsekasnetwork}, optimal control \cite{betts2010practical}, PDE-constrained optimization \cite{rees2010optimal}, resource allocation \cite{fontaine2020adaptive}, and transportation \cite{maggioni2009stochastic}. More applications can be found, for example, in \cite{bottou2018optimization,combettes2011proximal,jiang2018consensus,liu2015sparse} and references therein.

Numerous stochastic gradient methods have been developed to solve special cases of problem \eqref{prob1} with $c = 0$,  Lipschitz smooth $\mathbb{E}[\tilde f(\cdot,\xi)]$, and compact $\mathcal{X}$. Notably, when $\Xi$ is an infinite set, stochastic gradient methods (e.g., \cite{lan2012optimal,nemirovski2009robust}) achieve a first-order oracle (FO) complexity of $\mathcal{O}(\epsilon^{-2})$---measured by the number of gradient evaluations of $\tilde f(\cdot,\xi)$---for finding an $\epsilon$-\textit{stochastic optimal solution}, i.e., a stochastic point $x$ satisfying $\mathbb{E}[F(x) - F^*] \leq \epsilon$. On the other hand, when $\Xi$ is a finite set, several stochastic gradient methods (e.g., \cite{defazio2014saga,schmidt2017minimizing,shalev2013stochastic,xiao2014proximal}) can be applied to compute an $\epsilon$-stochastic optimal solution with improved FO complexity of $\widetilde{\mathcal{O}}(\epsilon^{-1})$,\footnote{$\widetilde{\mathcal{O}}(\cdot)$ represents $\mathcal{O}(\cdot)$ with logarithmic factors hidden.} where variance reduction techniques are utilized. More recently, optimal incremental gradient methods (e.g., \cite{lan2018optimal}) and accelerated variance-reduced stochastic methods (e.g., \cite{allen2018katyusha,lan2019unified}) have been developed to achieve (nearly) optimal complexity of $\widetilde{\mathcal{O}}(\epsilon^{-1/2})$ or $\mathcal{O}(\epsilon^{-1/2})$ for computing an $\epsilon$-stochastic optimal solution.

Stochastic gradient methods have also been developed for solving special instances of problem \eqref{prob1} with $c \neq 0$ and Lipschitz smooth $\mathbb{E}[\tilde f(\cdot,\xi)]$. Specifically, when $\Xi$ is an infinite set, a stochastic gradient method is proposed in \cite[Section 3.2]{lan2012optimal} for solving \eqref{prob1} where $\psi$ is the indicator function of a simple convex compact set and the constraints are linear equalities of the form $Ax - b = 0$. This method achieves an FO complexity of $\mathcal{O}(\epsilon^{-2})$ for finding a stochastic solution $x$ satisfying $\|Ax - b\| \leq \epsilon$ and $F(x) - F^* \leq \epsilon$ with \textit{high probability}. In addition, the constraint extrapolation method \cite{boob2023stochastic} can be applied to \eqref{prob1} to compute a stochastic solution $x$ satisfying 
\beq \label{weak-eps-exp-opt-soln} 
\mathbb{E}[\|[c(x)]_+\|] \leq \epsilon, \quad \mathbb{E}[F(x) - F^*] \leq \epsilon, 
\eeq 
with an operation complexity of $\mathcal{O}(\epsilon^{-2})$. This complexity is measured by the number of convex subproblems solved, each involving simple quadratic constraints and an objective that is the sum of $\psi$ and a simple quadratic function (where ``simple'' means that the Hessian is a multiple of the identity matrix). Furthermore, when $\Xi$ is a finite set and $\psi$ is the indicator function of a simple convex compact set, the level-set method \cite{lin2018level} can also be applied to problem \eqref{prob1} to compute a solution $x$ satisfying \eqref{weak-eps-exp-opt-soln}, achieving an FO complexity of $\widetilde{\mathcal{O}}(\epsilon^{-2})$.

In many applications such as energy systems \cite{shabazbegian2020stochastic}, machine learning \cite{cuomo2022scientific,karniadakis2021physics}, resource allocation \cite{fontaine2020adaptive}, and transportation \cite{maggioni2009stochastic}, all or some of the constraints in problem \eqref{prob1} are hard constraints that represent imperative requirements. Consequently, any desirable approximate solution must (nearly) satisfy these constraints. In this paper, we propose stochastic first-order methods with complexity guarantees for computing an $\epsilon$-\textit{surely feasible stochastic optimal} ($\epsilon$-SFSO) solution to problem \eqref{prob1}, that is, a stochastic point $x$ satisfying
\beq \label{eps-opt-soln} 
\|[c(x)]_+\| \leq \epsilon, \quad |\bbE[F(x)-F^*]| \leq \epsilon, 
\eeq 
 by considering two separate cases for the sample space $\Xi$: infinite or finite. 

Specifically, when  $\Xi$ is an infinite sample space, we propose a \textit{single-loop} stochastic first-order method (Algorithm \ref{alg1}) to solve problem \eqref{prob1} by applying an accelerated stochastic gradient scheme \cite{lan2012optimal} \textit{only once} to a sequence of quadratic penalty problems
\[
\min_x\left\{\bbE[\tf(x,\xi)]+\psi(x)+\frac{\rho_k}{2}\|[c(x)]_+\|^2\right\}
\]
for appropriately chosen penalty parameters $\{\rho_k\} \subset (0,\infty)$. Under suitable assumptions, we show that this method achieves a (nearly) optimal FO complexity of $\mcO(\epsilon^{-2})$ or $\widetilde\mcO(\epsilon^{-2})$ for finding an $\epsilon$-\emph{SFSO solution} $x$ satisfying \eqref{eps-opt-soln}, depending on whether constant or dynamic penalty parameters $\{\rho_k\}$ are used.

When $\Xi$ is a finite sample space with $\Xi = \{1, 2, \ldots, s\}$, we assume $\xi$ follows a uniform distribution for simplicity. In this case, problem \eqref{prob1} reduces to a constrained finite-sum optimization problem: \beq\label{prob2}
\begin{aligned}
F^*=\min_{x} &\  \, \Big\{F(x):=\frac{1}{s}\sum_{i=1}^s f_i(x) +\psi(x)\Big\}\\
\mbox{s.t.}&\ \, c(x)\leq0,
\end{aligned}
\eeq
where $f_i(x) := \tilde f(x, i)$. To solve problem \eqref{prob2}, we propose a stochastic first-order method 
(Algorithm \ref{alg2}) that applies a modified variance-reduced accelerated gradient scheme \emph{only once} to a sequence of associated quadratic penalty subproblems:
\beq \label{penalty-finitesum}
\min_{x}\Big\{\frac{1}{s}\sum_{i=1}^sf_i(x)+\psi(x)+\frac{\rho_k}{2}\|[c(x)]_+\|^2\Big\}
\eeq
for a suitably chosen sequence of penalty parameters $\{\rho_k\} \subset [0, \infty)$. Our modified scheme resembles a single loop of \cite[Algorithm 1]{lan2019unified}, with a key distinction: the smooth components of \eqref{penalty-finitesum}, namely $s^{-1} \sum_{i=1}^s f_i(x)$ and $\rho_k \|[c(x)]_+\|^2 / 2$, are treated separately. Specifically, the gradient of $s^{-1} \sum_{i=1}^s f_i(x)$ is approximated using a variance-reduced stochastic estimator, while the gradient of $\rho_k\|[c(x)]_+\|^2 / 2$ is computed exactly (see Section~\ref{sec:finite-sum} for details). Under suitable assumptions and $m=\mathcal{O}(1)$,  we show that  the proposed method achieves an FO complexity of $\mathcal{O}(s\log s+\sqrt{s}\epsilon^{-3/2})$ for computing an $\epsilon$-\emph{SFSO solution} $x$ satisfying \eqref{eps-opt-soln}. Moreover, it achieves an FO complexity of $\mathcal{O}(s\log s+\sqrt{s}\epsilon^{-1})$ or $\widetilde{\mathcal{O}}(s\log s+\sqrt{s}\epsilon^{-1})$ for computing an $\epsilon$-\textit{expected feasible stochastic optimal} ($\epsilon$-EFSO) solution, that is, a stochastic point $x$ satisfying
\beq \label{eps-exp-opt-soln}  
\bbE[\|[c(x)]_+\|] \leq \epsilon, \quad |\bbE[F(x)-F^*]|\leq\epsilon,
\eeq 
depending on whether constant or dynamic penalty parameter schemes  are used.  Comparing \eqref{eps-opt-soln} with \eqref{eps-exp-opt-soln}, it is clear that an $\epsilon$-SFSO solution is generally stronger than an $\epsilon$-EFSO solution.

For ease of reference, we summarize in Table~\ref{tab:complexity} the aforementioned FO complexity bounds for the proposed methods (Algorithms~\ref{alg1} and~\ref{alg2}) under both constant and dynamic penalty schemes for computing $\epsilon$-EFSO and $\epsilon$-SFSO solutions.

\begin{table}[H]
\centering
\caption{{Summary of FO complexity bounds for problem~\eqref{prob1}. }}
\label{tab:complexity}
\renewcommand{\arraystretch}{1.15}
\setlength{\tabcolsep}{5pt}
\begin{tabular}{|l|l|l|l|l|}
\hline
{Sample space} & {Method} & {Penalty scheme} & {$\epsilon$-EFSO complexity} & {$\epsilon$-SFSO complexity} \\
\hline
{Infinite} & {Alg.~\ref{alg1}} & {constant ($\rho=K^{3/2}$)} & {$\cO(\epsilon^{-2})$} & {$\cO(\epsilon^{-2})$} \\
{Infinite} & {Alg.~\ref{alg1}} & {dynamic ($\rho_k=(k+4)^{3/2}$)} & {$\widetilde\cO(\epsilon^{-2})$} & {$\widetilde\cO(\epsilon^{-2})$} \\
{Finite} & {Alg.~\ref{alg2}} & {constant ($\rho=s^{2/3}K^{4/3}$)} & {$\cO(s\log s+\sqrt{s}\,\epsilon^{-3/2})$} & {$\cO(s\log s+\sqrt{s}\,\epsilon^{-3/2})$} \\
{Finite} & {Alg.~\ref{alg2}} & {constant ($\rho=\sqrt{s}\,K$)} & {$\cO(s\log s+\sqrt{s}\,\epsilon^{-1})$} & {$\cO(s\log s+\sqrt{s}\,\epsilon^{-2})$} \\
{Finite} & {Alg.~\ref{alg2}} & {dynamic, eq.~\eqref{rho1}} & $\cO(s\log s+\sqrt{s}\,\epsilon^{-3/2})$ & $\cO(s\log s+\sqrt{s}\,\epsilon^{-3/2})$ \\
{Finite} & {Alg.~\ref{alg2}} & {dynamic, eq.~\eqref{rho2}} & {$\widetilde\cO(s\log s+\sqrt{s}\,\epsilon^{-1})$} & $\cO(s\log s+\sqrt{s}\,\epsilon^{-2})$ \\
\hline
\end{tabular}
\end{table}

The main contributions of this paper are summarized as follows.

\begin{itemize}
 \item We propose a \textit{single-loop} stochastic first-order method (Algorithm~\ref{alg1}) for solving problem \eqref{prob1} with an infinite sample space. This method finds an $\epsilon$-SFSO solution---generally stronger than the commonly studied $\epsilon$-EFSO solution---with (nearly) optimal FO complexity of $\mathcal{O}(\epsilon^{-2})$ or $\widetilde{\mathcal{O}}(\epsilon^{-2})$.

\item We propose a stochastic first-order method (Algorithm~\ref{alg2}) for solving problem \eqref{prob1} with a finite sample space. It achieves an FO complexity of $\mathcal{O}(s\log s+\sqrt{s}\epsilon^{-3/2})$ for finding an $\epsilon$-SFSO solution, and $\mathcal{O}(s\log s+\sqrt{s}\epsilon^{-1})$ or $\widetilde{\mathcal{O}}(s\log s+\sqrt{s}\epsilon^{-1})$ for finding an $\epsilon$-EFSO solution, assuming $m = \mathcal{O}(1)$.
\end{itemize}

To the best of our knowledge, this work is the first to study stochastic first-order methods for computing an $\epsilon$-SFSO solution of problem~\eqref{prob1} with provable complexity guarantees. Moreover, in the special case where $s=1$, Algorithm~\ref{alg2} reduces to a \textit{single-loop first-order penalty method} for solving deterministic constrained convex optimization problems, and achieves (nearly) optimal FO complexity of $\mathcal{O}(\epsilon^{-1})$ or $\widetilde{\mathcal{O}}(\epsilon^{-1})$, depending on whether constant or dynamic penalty parameters are used. 

The rest of the paper is organized as follows. Subsection~\ref{sec:notation} introduces the notation and assumptions used throughout the paper. Sections~\ref{sec:composite} and~\ref{sec:finite-sum} present our proposed stochastic first-order methods for solving problem~\eqref{prob1} under infinite and finite sample spaces, respectively, along with their convergence guarantees. Numerical results are presented in Section~\ref{sec:results}, and the proofs of the main results are provided in Section~\ref{sec:proof}. Finally, concluding remarks are given in Section~\ref{sec:conclude}.

\subsection{Notation and assumptions}\label{sec:notation}
The following notation will be used throughout this paper. Let $\rr^n$ denote the Euclidean space of dimension $n$. We use $\langle\cdot, \cdot \rangle$ to denote the standard inner product, and $\|\cdot\|$ to denote Euclidean norm. The notation $[\cdot]_+$ denotes the projection operator onto the nonnegative orthant.

A mapping $\phi$ is said to be \emph{$L_{\phi}$-Lipschitz continuous} on a set $\Omega$ if $\|\phi(x)-\phi(x')\| \leq L_{\phi} \|x-x'\|$ for all $x,x'\in \Omega$. In addition, it is said to be \emph{$L_{\nabla\phi}$-smooth} on $\Omega$ if $\|\nabla\phi(x)-\nabla\phi(x')\| \leq L_{\nabla\phi} \|x-x'\|$ for all $x,x'\in \Omega$, where $\nabla \phi$ denotes the transpose of the Jacobian of $\phi$.

Throughout this paper, we make the following assumptions for problem \eqref{prob1}.
\begin{assumption} \label{a1}
\begin{enumerate}[label=(\roman*)]
\item The set $\cX$ (i.e., the domain of $\psi$) is bounded and the function $F$ has a bounded variation on $\cX$, i.e.,  there exist $D_\cX>0$ and $D_F>0$ such that $ \max_{x, y\in \cX} \|x-y\|\leq D_\cX$ and $\max_{x\in\cX}F(x)-\min_{x\in\cX}F(x)\leq D_F$.
\item Each constraint function $c_i$ is $L_{c_i}$-Lipschitz continuous and $L_{\nabla c_i}$-smooth on $\cX$, and satisfies $|c_i(x)| \leq C_i$ for all $x \in \cX$ and $i = 1, \ldots, m$.
\item The set of optimal Lagrange multipliers of problem \eqref{prob1}, denoted by $\boldsymbol{\Lambda^*}$, is nonempty. Consequently, $\Lambda:=\min_{\lambda\in \boldsymbol{\Lambda^*}} \|\lambda\|$ is finite.
\end{enumerate}
\end{assumption}
It follows from Assumption \ref{a1}(ii) and the convexity of $c_i$'s that $\|[c(x)]_+\|^2/2$ is $L_{\nabla c^2}$-smooth and convex, where
\begin{equation}\label{L_nablac}
    L_{\nabla c^2}:=\sum_{i=1}^m \left(L_{c_i}^2+C_iL_{\nabla c_i}\right).
\end{equation}

\section{A stochastic first-order method for problem \eqref{prob1}
with an infinite sample space}\label{sec:composite}

In this section, we consider problem \eqref{prob1} in which $\xi$ is a random variable with an infinite sample space $\Xi$. Specifically, we propose a single-loop stochastic first-order method to solve this problem by applying an accelerated stochastic gradient (ASG) scheme \cite{lan2012optimal} \textit{only once} to a sequence of quadratic penalty problems
\beq\label{def-Frho1}
F_{\rho_k}^*=\min_x\left\{F_{\rho_k}(x):=\bbE[\tf(x,\xi)]+\psi(x)+\frac{\rho_k}{2}\|[c(x)]_+\|^2\right\}
\eeq
for suitably chosen penalty parameters $\{\rho_k\} \subset (0,\infty)$.  For each $k \geq 1$, the ASG scheme generates three points, $y_k$, $z_{k+1}$, and $x_{k+1}$, as follows: it first forms a convex combination of $x_k$ and $z_k$ to obtain $y_k$; then it computes a stochastic gradient of $f_{\rho_k}(\cdot):=\bbE[\tf(\cdot,\xi)]+\rho_k\|[c(\cdot)]_+\|^2/2$ at $y_k$ and uses it in a proximal step associated with $\psi$ to obtain $z_{k+1}$; finally, it forms another convex combination of $x_k$ and $z_{k+1}$ to obtain $x_{k+1}$. Since $c$ is a smooth deterministic mapping, a stochastic gradient of $f_{\rho_k}$ at $y_k$ can be computed as $\nabla \tf(y_k,\xi_k)+\rho_k\nabla c(y_k)[c(y_k)]_+$, where $\xi_k$ is randomly sampled from $\Xi$. The resulting method for problem \eqref{prob1} with an infinite sample space is presented in Algorithm~\ref{alg1}.

\begin{algorithm}[H]
\caption{A stochastic first-order method for problem \eqref{prob1}
with an infinite sample space}\label{alg1}
\begin{algorithmic}[1]
\REQUIRE $x_1\in\cX$, $\{\beta_k\}\subset[1,\infty)$, $\{\gamma_k\}\subset (0,\infty)$, and $\{\rho_k\}\subset(0,\infty)$.
\STATE Set $z_1=x_1$.
\FOR{$k=1,2,\dots$}
\STATE $y_k=(1-\beta_k^{-1})x_k+\beta_k^{-1}z_k$. 
\STATE Sample $\xi_k\in \Xi$ and compute $g_k=\nabla \tilde f(y_k,\xi_k)+\rho_k\nabla c(y_k)[c(y_k)]_+$.
\STATE $z_{k+1}=\prox_{\gamma_k\psi}(z_k-\gamma_kg_k)$.
\STATE $x_{k+1}=(1-\beta_k^{-1})x_k+\beta_k^{-1}z_{k+1}$.
\ENDFOR \\
\end{algorithmic}
\end{algorithm}

Before presenting convergence results for Algorithm \ref{alg1}, we introduce the following additional assumption.

\begin{assumption}\label{a2}
The function $f(\cdot):=\bbE[\tilde f(\cdot,\xi)]$ is $L_{\nabla f}$-smooth on $\cX$, and $\tf$ satisfies the following conditions:
\begin{equation*}
\bbE[\nabla \tilde f(x,\xi)]=\nabla f(x),\quad\bbE[\|\nabla \tilde f(x,\xi)-\nabla f(x)\|^2]\leq\sigma^2,\quad \max_{\xi\in\Xi}\|\nabla \tilde f(x,\xi)\|\leq G \qquad \forall x\in\cX
\end{equation*}
for some constants $\sigma>0$ and $G>0$ independent of $x\in\cX$.\footnote{The bounded variance condition, namely the second part of Assumption~2, is in fact redundant, as it is implied by the last part of Assumption~2. Moreover, the variance bound $\sigma^2$ can be replaced by an expression involving $G$ and other problem-dependent quantities. For notational convenience, however, we retain the second part of Assumption~2 and continue to use $\sigma^2$, consistent with the standard bounded variance setting.}
\end{assumption}

The next theorem presents convergence results for Algorithm \ref{alg1} with constant penalty parameters $\{\rho_k\}$, whose proof is deferred to Subsection \ref{sec:proof-composite}. 

\begin{thm}[{\bf constant penalty parameters}]\label{thm:composite_fixed}
Suppose that Assumptions \ref{a1} and \ref{a2} hold. Let $L_{\nabla c^2}$, $D_{\mathcal{X}}$, $\Lambda$, $L_{\nabla f}$, $\sigma$, and $G$ be given in \eqref{L_nablac} and Assumptions \ref{a1} and \ref{a2}, respectively. Assume that $x_K$ is generated by Algorithm \ref{alg1} for some $K \geq 2$, with the parameters chosen as
\begin{equation}\label{def-para2}
 \rho_k =\rho, \quad \beta_k=\frac{k+1}{2}, \quad\gamma_k=\frac{k+1}{4(L_{\nabla f}+\rho_k L_{\nabla c^2})}
\end{equation}
 for all $1\leq k<K$, where $\rho=K^{3/2}$. Then we have
 \begin{align}
     &\|[c(x_K)]_+\|\leq (6GD_\cX)^{1/2}K^{-3/4}+(12L_{\nabla c^2})^{1/2} D_\cX K^{-1}+2\Lambda K^{-3/2}+(12L_{\nabla f})^{1/2}D_\cX K^{-7/4},\label{thm2:ineq1}\\
     &\bbE[F(x_K)-F^*]
   \leq \big(6L_{\nabla c^2} D_\cX^2 +3\sigma^2L_{\nabla c^2}^{-1}\big)K^{-1/2}+ 6L_{\nabla f} D_\cX^2 K^{-2}, \label{thm2:ineq2} \\
     &\bbE[F(x_K)-F^*]\geq -\Lambda \Big((12L_{\nabla c^2} D_\cX^2+6\sigma^2L_{\nabla c^2}^{-1})^{1/2} K^{-1}+2\Lambda K^{-3/2}+(12L_{\nabla f})^{1/2} D_\cX K^{-7/4} \Big). \label{thm2:ineq3}
 \end{align}
\end{thm}

\begin{rem} \label{remark1}
(i) It follows from Theorem~\ref{thm:composite_fixed} that Algorithm~\ref{alg1}, with constant penalty parameters $\rho_k \equiv K^{3/2}$, achieves an FO complexity of $\cO(\epsilon^{-2})$ for finding an $\epsilon$-SFSO solution of problem~\eqref{prob1} as defined in~\eqref{eps-opt-soln}. This matches the optimal FO complexity of stochastic gradient methods for solving the unconstrained counterpart of~\eqref{prob1} (i.e., when $c = 0$) under the bounded-variance setting. Consequently, the presence of constraints does not affect the order of dependence of the complexity on $\epsilon$.

(ii)  For the implementation of Algorithm~\ref{alg1} with $\rho_k \equiv K^{3/2}$, one may choose $K = \lceil \epsilon^{-2} \rceil$, where $\epsilon$ is a prescribed tolerance. As discussed above, this  choice ensures that the output $x_K$ of Algorithm~\ref{alg1} is an $\cO(\epsilon)$-SFSO solution to problem~\eqref{prob1}.
\end{rem}

While Algorithm~\ref{alg1} with a constant penalty parameter achieves the optimal complexity, the pre-specified penalty parameter depends on $K$---the total number of iterations. Since $K$ is typically not known a priori, it may need to be tuned in practice. To address this issue, we next consider Algorithm~\ref{alg1} with dynamic penalty parameters $\{\rho_k\}$ and present its convergence results, with the proof deferred to Subsection~\ref{sec:proof-composite}.

\begin{thm}[{\bf dynamic penalty parameters}]\label{thm:composite}
Suppose that Assumptions \ref{a1} and \ref{a2} hold. Let $L_{\nabla c^2}$, $D_{\mathcal{X}}$, $\Lambda$, $L_{\nabla f}$, $\sigma$, and $G$ be given in \eqref{L_nablac} and Assumptions \ref{a1} and \ref{a2}, respectively. Assume that $\{x_k\}$ is generated by Algorithm \ref{alg1} with the parameters chosen as
\begin{equation}\label{def-para1}
  \begin{aligned}
 \rho_k=(k+4)^{\frac{3}{2}}, \quad \beta_k=\frac{k+4}{5}, \quad\gamma_k=\frac{k+4}{10(L_{\nabla f}+\rho_k L_{\nabla c^2})}.
\end{aligned}  
\end{equation}
Then for all $k\geq2$, we have
\begin{align}
    \|[c(x_k)]_+\|
     &\leq (2\tC_2)^{1/2}k^{-3/4}+2(\tC_1L_{\nabla c^2})^{1/2}k^{-1}+2\Lambda k^{-3/2}+2(\tC_1L_{\nabla f})^{1/2}k^{-7/4}, \label{thm1-ineq1}\\
    \bbE[F(x_k)-F^*]&\leq 3\sqrt{3}\sigma^2L_{\nabla c^2}^{-1}k^{-1/2}\log k+2\sqrt{3}\tC_1L_{\nabla c^2}k ^{-1/2}+3\sigma^2L_{\nabla f}L_{\nabla c^2}^{-2}k^{-2}\log k+2\tC_1L_{\nabla f}k^{-2}, \label{thm1-ineq2}\\
    \bbE[F(x_k)-F^*]&\geq-\Lambda\Big(\sqrt{6}\sigma L_{\nabla c^2}^{-1/2} k^{-1}(\log k)^{1/2} +2(\tC_1L_{\nabla c^2})^{1/2}k^{-1}+2\Lambda k^{-3/2}\nn \\ 
  &\quad\quad\quad+\sigma L_{\nabla c^2}^{-1} (6L_{\nabla f})^{1/2} k^{-7/4}(\log k)^{1/2}+2(\tC_1L_{\nabla f})^{1/2}k^{-7/4}\,\Big),  \label{thm1-ineq3}
\end{align}
where 
 \begin{equation}\label{c12}
    \tC_1=25D_\cX^2+\frac{10\Lambda^2}{L_{\nabla c^2}},
    \quad \tC_2=\frac{20GD_\cX L_{\nabla f}}{\sqrt{6}L_{\nabla c^2}}+{120GD_\cX}.
 \end{equation}
\end{thm}

\begin{rem} \label{remark2}
{(i) The parameters $\beta_k$ and $\gamma_k$ in Theorem~\ref{thm:composite} are nearly proportional to those in Theorem~\ref{thm:composite_fixed}. In addition, the parameter $\rho_k$ in Theorem~\ref{thm:composite} increases gradually to $(K+4)^{3/2}$ as $k$ approaches $K$, attaining the same order as $\rho_k \equiv K^{3/2}$ in Theorem~\ref{thm:composite_fixed}, where $K$ is a prescribed iteration number. Consequently, the convergence rates in~\eqref{thm1-ineq1}--\eqref{thm1-ineq3} with $k = K$ achieved by Algorithm~\ref{alg1} with dynamic penalty parameters are of similar form to those in~\eqref{thm2:ineq1}--\eqref{thm2:ineq3} achieved with constant penalty parameters, except for additional logarithmic factors.}

{(ii) Setting $k = K$ in~\eqref{thm1-ineq1}--\eqref{thm1-ineq3}, we observe that Algorithm~\ref{alg1} with dynamic penalty parameters achieves a convergence rate of $\mathcal{O}(K^{-3/4})$ in $\|[c(x_K)]_+\|$, and $\mathcal{O}(K^{-1/2}\log K)$ in $|\bbE[F(x_K) - F^*]|$. As a result, Algorithm~\ref{alg1} with dynamic penalty parameters attains a nearly optimal FO complexity of
$\widetilde{\mathcal{O}}(\epsilon^{-2})$ for finding an $\epsilon$-SFSO solution of problem~\eqref{prob1} as defined in \eqref{eps-opt-soln}. Although this complexity is at most worse by a logarithmic factor than that achieved with constant penalty parameters, Algorithm~\ref{alg1} with dynamic penalty parameters generally performs better in practice, as observed in our numerical experiments.}
\end{rem}

\section{A stochastic first-order method for problem \eqref{prob1}
with a finite sample space}\label{sec:finite-sum}

In this section, we consider problem \eqref{prob1} in which $\xi$ is a uniformly distributed random variable with a finite sample space $\Xi=\{1,2,\ldots,s\}$. In this case, problem~\eqref{prob1} reduces to a constrained finite-sum optimization problem given in \eqref{prob2}.

When $c=0$, problem \eqref{prob2} reduces to an unconstrained finite-sum optimization problem  
\beq \label{fun-f}
\min_{x} \big\{\barf(x) +\psi(x)\big\}, \ \ \text{where} \ \  \barf(x):=\frac{1}{s}\sum_{i=1}^sf_i(x).
\eeq 
This problem has been extensively studied in the literature (e.g., see \cite{hendrikx2021optimal,lan2019unified,schmidt2017minimizing,song2020variance,zhou2019lower}).
Notably, \cite[Algorithm 1]{lan2019unified} proposes a variance-reduced accelerated gradient method for solving~\eqref{fun-f}. This method can be viewed as a restarted accelerated stochastic gradient scheme consisting of two nested loops: the outer loop updates a reference point $\tilde x$ and  the starting points $x_0$, $z_0$ for the inner loop, while the inner loop performs a sequence of accelerated stochastic gradient steps. At each inner iteration, a stochastic variance-reduced  gradient (SVRG) of $\barf$ is computed by using the full gradient $\nabla \barf(\tilde x)$ at the reference point $\tilde x$, along with the sampled gradients $\nabla f_i(\tilde x)$ and $\nabla f_i(y)$, where $y$ is updated at each inner iteration, and  $i\in \{1, 2, \dots, s\}$ is drawn according to a specified probability distribution.

We propose a stochastic first-order method for solving problem \eqref{prob2} by applying a \textit{modified} variance-reduced accelerated gradient scheme \textit{only once} to a sequence of quadratic penalty problems: 
\beq\label{def-Frho2}
F_{\rho_k}^*=\min_{x}\Big\{F_{\rho_k}(x):=\barf(x)+\psi(x)+\frac{\rho_k}{2}\|[c(x)]_+\|^2\Big\}
\eeq
for suitably chosen penalty parameters $\{\rho_k\} \subset [0,\infty)$. Our modified scheme closely follows one loop of \cite[Algorithm 1]{lan2019unified}, but with a key distinction: the two objective components $\barf(x)$  and $\rho_k\|[c(x)]_+\|^2/2$ are handled separately. Specifically, we approximate the gradient of $\barf(x)$  using a SVRG, while computing the gradient of $\rho_k\|[c(x)]_+\|^2/2$ exactly. The sum of these two components serves as the SVRG for $\barf(x)+\rho_k\|[c(x)]_+\|^2/2$ in our scheme.  The resulting method for problem \eqref{prob2} is presented in Algorithm \ref{alg2}.

\begin{algorithm}[H]
\caption{A variance reduced stochastic first-order method for problem \eqref{prob2}}\label{alg2}
\begin{algorithmic}[1]
\REQUIRE $x^0\in\cX$, $\{\alpha_k\}\subset(0, 1]$, $\{p_k\}\subset(0,1]$, $\{\gamma_k\}\subset (0,\infty)$, $\{T_k\}$, {$\{\theta_t\}$}, a probability distribution $\mathcal{Q}=\{q_1, q_2,\dots, q_s\}$, and $\{\rho_k\}\subset (0,\infty)$.
\STATE Set $\tilde x_1=x^0$ and $\tilde z_1=x^0$.
\FOR{$k=1,2,\dots$}
\STATE Set $x_0=\tilde x_{k}$, $z_0=\tilde z_{k}$,  and $\tilde g_{k} = \nabla \barf(\tilde x_{k})$. 
\FOR{$t=1, 2, \dots, T_k$}
\STATE Sample $i_t\in \{1, 2, \dots, s\}$ according to $\mathcal{Q}$.
\STATE $y_t=(1-\alpha_k-p_k)x_{t-1}+\alpha_kz_{t-1}+p_k\tilde x_{k}$.
\STATE $g_t=(\nabla f_{i_t}(y_t)-\nabla f_{i_t}(\tilde x_{k}))/(q_{i_t}s)+\tilde g_{k} +\rho_k  \nabla c(y_t)[c(y_t)]_+$.
\STATE $z_t=\argmin_{x\in\cX}\left\{\gamma_k(\langle g_t,x \rangle + \psi(x))+\|x-z_{t-1}\|^2/2\right\}$.
\STATE $x_t=(1-\alpha_k-p_k)x_{t-1}+\alpha_kz_t+p_k\tilde x_k$.
\ENDFOR
\STATE Set $\tilde z_{k+1}=z_{T_k}$ and $\tilde x_{k+1}=\sum_{t=1}^{T_k}(\theta_tx_t)/\sum_{t=1}^{T_k}\theta_t$.
\ENDFOR
\end{algorithmic}
\end{algorithm}

Before presenting convergence results for Algorithm \ref{alg2}, we introduce the following additional assumption for problem \eqref{prob2}.

\begin{assumption}\label{a3}
$f_i$ is $L_i$-smooth on $\cX$ for all $i=1,2,\dots,s$.
\end{assumption}

In view of this assumption and the definition of $\barf$ in \eqref{fun-f}, one can observe that $\barf$ is $L_{\nabla \barf}$-smooth, where
 \begin{equation}\label{Lam}
L_{\nabla \barf}=\sum_{i=1}^sL_i/s.
 \end{equation}

The next theorem presents convergence results for Algorithm \ref{alg2} with constant penalty parameters $\{\rho_k\}$, whose proof is deferred to Subsection \ref{sec:proof-finitesum}. 

\begin{thm}[{\bf constant penalty parameters}]\label{thm:finite-sum_fixed}
 Suppose that Assumptions \ref{a1} and \ref{a3} hold. Let $k_0=\lfloor \log_2 s \rfloor+1$, and $L_{\nabla c^2}$, $L_{\nabla \barf}$,  $D_{\mathcal{X}}$, $D_F$, $\Lambda$, and $L_i$'s  be  given  in \eqref{L_nablac},  \eqref{Lam}, and Assumptions \ref{a1} and \ref{a3}, respectively. Assume that $\tilde x_K$ is generated by Algorithm \ref{alg2} for some $K \geq \max\{k_0+1,2(k_0-3)\}$, with the parameters chosen as 
\begin{equation}\label{def-para4}
\begin{aligned}
    &  \alpha_k=\left\{
    \begin{array}{ll}
      \frac{1}{2}   & \text{if} \  \ 1\leq k\leq k_0, \\
       \frac{2}{k-k_0+4}  & \text{if} \  \ k>k_0,
    \end{array}\right.  \quad  p_k=\frac{1}{2}, \quad \rho_k = \rho, \quad \gamma_k=\frac{1}{3(L_{\nabla \barf}+\rho_k L_{\nabla c^2})\alpha_k},  \\
&       q_i=\frac{L_i}{\sum_{j=1}^sL_j}, \quad T_k=\left\{\begin{array}{ll}
       2^{k-1}  & \text{if} \  \ 1\leq k\leq k_0, \\
       2^{k_0-1}  & \text{if} \  \ k>k_0,
       \end{array}\right. 
    \quad \theta_t=\left\{
    \begin{array}{ll}
      \frac{\gamma_k}{\alpha_k}(\alpha_k+p_k)   & \text{if} \  \ 1\leq t\leq T_k-1, \\
       \frac{\gamma_k}{\alpha_k}  & \text{if} \  \ t=T_k
    \end{array}\right.  
\end{aligned}  
\end{equation}
for some $\rho>0$ specified below. 
Then the following statements hold.
\begin{itemize}
\item [(i)] 
Assume that $\rho=s^{2/3}K^{4/3}$.
Then we have
\begin{align}
&\|[c(\tilde x_K)]_+\|
\leq 8L_{\nabla \barf}^{1/2}D_\cX s^{-1/3}K^{-2/3}+16\tC_4^{1/2}s^{-1/2}K^{-1}+2\Lambda s^{-2/3}K^{-4/3}\nn\\
&\qquad\qquad\quad\ \ +16\tC_3^{1/2}s^{-5/6}K^{-5/3}, \label{thm4-ineq1} \\
&\bbE[\|[c(\tilde x_K)]_+\|]
\leq  16\tC_4^{1/2}s^{-1/2}K^{-1}+2\Lambda s^{-2/3}K^{-4/3}+16\tC_3^{1/2}s^{-5/6}K^{-5/3}, \label{thm4-ineq1-exp} \\
&\bbE [F(\tilde x_K)-F^*]
\leq 64\tC_4s^{-1/3}K^{-2/3}+64\tC_3s^{-1} K^{-2},  \label{thm4-ineq2} \\
&\bbE [F(\tilde x_K)-F^*]
\geq -16 \Lambda\tC_4^{1/2}s^{-1/2}K^{-1}-2\Lambda^2 s^{-2/3}K^{-4/3}-16 \Lambda\tC_3^{1/2}s^{-5/6}K^{-5/3},  \label{thm4-ineq3}
\end{align}
where
\begin{equation}\label{c1011}
    \tC_3=2D_F+3L_{\nabla \barf}D_\cX^2/2,\quad
     \tC_4=\|[c(x^0)]_+\|^2+3L_{\nabla c^2}D^2_\cX/2.
\end{equation}

\item[(ii)] Assume that $\rho= \sqrt{s}K$.
Then we have
\begin{align}
&\|[c(\tilde x_K)]_+\|
\leq 8L_{\nabla \barf}^{1/2}D_\cX s^{-1/4} K^{-1/2}+\big(2\Lambda +16\tC_4^{1/2}\big)s^{-1/2}K^{-1}+16\tC_3^{1/2}s^{-3/4}K^{-3/2},  \label{thm4-ineq4-1}  \\
&\bbE[\|[c(\tilde x_K)]_+\|]
\leq \big(2\Lambda +16\tC_4^{1/2}\big)s^{-1/2}K^{-1}+16\tC_3^{1/2}s^{-3/4}K^{-3/2}, \label{thm4-ineq4} \\
&\bbE [F(\tilde x_K)-F^*]
\leq 64\tC_4 s^{-1/2}K^{-1}+64\tC_3 s^{-1}K^{-2},  \label{thm4-ineq5} \\
&\bbE [F(\tilde x_K)-F^*]
\geq -\Lambda \big(2\Lambda +16\tC_4^{1/2}\big)s^{-1/2} K^{-1}-16\Lambda\tC_3^{1/2}s^{-3/4} K^{-3/2},   \label{thm4-ineq6}
\end{align}
where $\tC_3$ and $\tC_4$ are given in \eqref{c1011}.
\end{itemize}
\end{thm}

\begin{rem}\label{remark3} 
(i) The choice of $\rho = s^{2/3} K^{4/3}$ is obtained by minimizing an upper bound on
\[
\max\left\{\|[c(\tilde x_K)]_+\|, \big|\bbE[F(\tilde x_K)-F^*]\big|\right\},
\]
which measures both feasibility violation and the expected optimality gap, over all possible choices of $\rho$.    
 Indeed, as observed from Lemma~\ref{l-exactbound3} and the proof of Theorem~\ref{thm:finite-sum_fixed} (see \eqref{l9-expbound} and \eqref{Thm3-detercons}), we have
\[
\|[c(\tilde x_K)]_+\| = \cO(\rho^{-1/2}) \qquad 
\big|\bbE[F(\tilde x_K)-F^*]\big| = \cO(\rho s^{-1} K^{-2}),
\]
which imply
\[
\max\big\{\|[c(\tilde x_K)]_+\|, \big|\bbE[F(\tilde x_K)-F^*]\big|\big\}
= \cO\big(\max\{\rho^{-1/2}, \rho s^{-1} K^{-2}\}\big).
\]
Consequently, to obtain an $\epsilon$-SFSO solution using Algorithm~\ref{alg2} with low complexity, it is natural to choose $\rho$ to minimize 
$\max\{\rho^{-1/2},  \rho s^{-1} K^{-2}\}$, which leads to $\rho = s^{2/3} K^{4/3}$.

Moreover, for this choice of $\rho$, we have
\[
\max\big\{\|[c(\tilde x_K)]_+\|, \big|\bbE[F(\tilde x_K)-F^*]\big|\big\}
= \cO\big(s^{-1/3} K^{-2/3}\big).
\]
Therefore, for the implementation of Algorithm~\ref{alg2} with $\rho_k \equiv s^{2/3} K^{4/3}$, one may choose
$K = \max\{k_0+1,2(k_0-3),\lceil s^{-1/2}\epsilon^{-3/2}\rceil\}$, where $\epsilon$ is a prescribed tolerance. With this choice, the output $\tilde x_K$ of Algorithm~\ref{alg2} is guaranteed to be an $\cO(\epsilon)$-SFSO solution to problem~\eqref{prob2}.

(ii) The choice of $\rho=\sqrt{s}K$ is obtained by minimizing an upper bound on
\[
\max\left\{\bbE[\|[c(\tilde x_K)]_+\|], \big|\bbE[F(\tilde x_K)-F^*]\big|\right\},
\]
which measures both the expected feasibility violation and the expected optimality gap, over all possible choices of $\rho$.   Indeed, as observed from Lemma~\ref{l-exactbound3} and the proof of Theorem~\ref{thm:finite-sum_fixed} (see \eqref{l9-expbound} and \eqref{Thm3-expcons}), we have
\[
\bbE[\|[c(\tilde x_K)]_+\|]=\mathcal{O}(\rho^{-1}), \qquad |\bbE[F(\tilde x_K)-F^*]|=\mathcal{O}(\rho s^{-1}K^{-2}),
\]
which imply
\[
\max\big\{\bbE[\|[c(\tilde x_K)]_+\|], \big|\bbE[F(\tilde x_K)-F^*]\big|\big\}
= \cO\big(\max\{\rho^{-1}, \rho s^{-1} K^{-2}\}\big).
\]
Consequently, to obtain an $\epsilon$-EFSO solution using Algorithm~\ref{alg2} with low complexity, it is natural to choose $\rho$ to minimize 
$\max\{\rho^{-1},  \rho s^{-1} K^{-2}\}$, which leads to $\rho = \sqrt{s}K$.

Moreover, for this choice of $\rho$, we have
\[
\max\big\{\bbE[\|[c(\tilde x_K)]_+\|], \big|\bbE[F(\tilde x_K)-F^*]\big|\big\}
= \cO\big(s^{-1/2} K^{-1}\big).
\]
Therefore, for the implementation of Algorithm~\ref{alg2} with $\rho_k \equiv \sqrt{s}K$, one may choose
$K = \max\{k_0+1,2(k_0-3),\lceil s^{-1/2}\epsilon^{-1}\rceil\}$, where $\epsilon$ is a prescribed tolerance. With this choice, the output $\tilde x_K$ of Algorithm~\ref{alg2} is guaranteed to be an $\cO(\epsilon)$-EFSO solution to problem~\eqref{prob2}.

(iii) Assume that $m = \mathcal{O}(1)$. Notice from \eqref{def-para4} that $T_k=\cO(s)$, and thus the number of gradient evaluations per outer iteration of Algorithm \ref{alg2} with the parameters specified in Theorem \ref{thm:finite-sum_fixed} is $\cO(s)$. Given $k_0=\lfloor \log_2 s \rfloor+1$ and $K \geq \max\{k_0+1,2(k_0-3)\}$, it follows from Theorem \ref{thm:finite-sum_fixed} that Algorithm \ref{alg2} with constant penalty parameter $\rho_k \equiv s^{2/3}K^{4/3}$ achieves an FO complexity of $\cO(s\log s+\sqrt{s}\epsilon^{-3/2})$ for computing an $\epsilon$-SFSO solution and an $\epsilon$-EFSO solution of \eqref{prob2}, which significantly improves upon the complexity of Algorithm \ref{alg1} in terms of the dependence of $\epsilon$. On the other hand, Algorithm \ref{alg2} with constant penalty parameter $\rho_k \equiv\sqrt{s} K$ achieves an FO complexity of $\cO(s\log s+\sqrt{s}\epsilon^{-2})$ for computing an $\epsilon$-SFSO solution, and an FO complexity of $\cO(s\log s+\sqrt{s}\epsilon^{-1})$ for finding an $\epsilon$-EFSO solution of \eqref{prob2}. Therefore, the choice $\rho_k \equiv s^{2/3}K^{4/3}$ leads to better complexity for obtaining an $\epsilon$-SFSO solution, while $\rho_k \equiv \sqrt{s}K$ yields better complexity for computing an $\epsilon$-EFSO solution, which is a weaker notion of approximate optimality than an $\epsilon$-SFSO solution. In addition, the complexity of $\mathcal{O}(s \log s + \sqrt{s} \epsilon^{-1})$ achieved by Algorithm~\ref{alg2} with $\rho_k \equiv \sqrt{s}K$ for finding an $\epsilon$-EFSO solution is significantly lower than the complexity of $\mathcal{O}(s \epsilon^{-1})$ or $\widetilde{\mathcal{O}}(s \epsilon^{-1})$ achieved by (nearly) optimal first-order methods \cite{kovalev2022first, lu2023iteration, xu2021iteration} for finding a deterministic $\epsilon$-optimal solution. 

(iv) In the special case $s=1$, Algorithm~\ref{alg2} with $\rho_k \equiv K$ reduces to a \textit{single-loop first-order penalty method} with a constant  penalty parameter for solving deterministic constrained convex optimization problems, while achieving an optimal FO complexity of $\mathcal{O}(\epsilon^{-1})$.
\end{rem}

While Algorithm \ref{alg2} with constant penalty parameters achieves strong complexity guarantees, the pre-specified penalty parameter depends on $K$---the total number of iterations. Since $K$ is typically not known a priori, it may need to be tuned in practice. To address this issue, we next consider Algorithm \ref{alg2} with dynamic penalty parameters $\{\rho_k\}$ and present its convergence results, with the proofs deferred to Subsection \ref{sec:proof-finitesum}.

\begin{thm}[{\bf dynamic penalty parameters}]\label{thm:finite-sum}
 Suppose that Assumptions \ref{a1} and \ref{a3} hold. Let $L_{\nabla c^2}$, 
$L_{\nabla \barf}$,  $D_{\mathcal{X}}$, $D_F$, $\Lambda$, and $L_i$'s  be  given  in \eqref{L_nablac},  \eqref{Lam}, and Assumptions \ref{a1} and \ref{a3}, respectively. Assume that $\{\tilde x_k\}$ is generated by Algorithm \ref{alg2} with the parameters chosen as
\begin{equation}\label{def-para3}
\begin{aligned}
    &   \alpha_k=\left\{
    \begin{array}{ll}
      \frac{6}{7}   & \text{if} \  \ 1\leq k\leq k_0, \\
       \frac{6}{k-k_0+7}  & \text{if} \  \ k>k_0,
    \end{array}\right.  \qquad  p_k=\frac{1}{7}, \quad \gamma_k=\frac{1}{8(L_{\nabla \barf}+\rho_k L_{\nabla c^2})\alpha_k}, \\ 
& q_i=\frac{L_i}{\sum_{j=1}^sL_j}, \qquad 
 \theta_t=\left\{
    \begin{array}{ll}
      \frac{\gamma_k}{\alpha_k}(\alpha_k+p_k)   & \text{if} \  \ 1\leq t\leq T_k-1, \\
       \frac{\gamma_k}{\alpha_k}  & \text{if} \  \ t=T_k,
    \end{array}\right.     
\end{aligned} 
\end{equation}
and $k_0$, $T_k$, $\rho_k$ specified below. Then the following statements hold.
\begin{itemize}
\item [(i)] 
Assume that 
\begin{equation}\label{rho1}
\begin{aligned}
k_0&=\lfloor (4/3)\log_2 s \rfloor+1,\\
    T_k&=\left\{\begin{array}{ll}
       \lceil2^{3(k-1)/4}\rceil  & \text{if} \  \ 1\leq k\leq k_0, \\ [4pt]
        \lceil2^{3(k_0-1)/4}\rceil  & \text{if} \  \ k>k_0,
       \end{array}\right. 
    \quad
    \rho_k=\left\{\begin{array}{ll}
       2^{k/2}  & \text{if} \  \ 1\leq k\leq k_0, \\
       3s^{2/3}(k-k_0+7)^{4/3}/32  & \text{if} \  \ k>k_0.
       \end{array}\right.
\end{aligned}
\end{equation}
 Then for all $k\geq \max\{k_0+1,2(k_0-7)\}$, we have
\begin{align}
&\|[c(\tilde x_k)]_+\|\leq 117L_{\nabla \barf}^{1/2}D_{\cX} s^{-1/3}k^{-2/3}+10\big(\tC_6^{1/2}+ \tC_8^{1/2}\big)s^{-1/3}k^{-1}\nn\\
&\qquad\qquad\quad\,\,+\big(54\Lambda+12\tC_9^{1/2} \big) s^{-2/3}k^{-4/3}+15\big(\tC_5^{1/2}+\tC_7^{1/2}\big)s^{-2/3} k^{-5/3}, \label{thm3-ineq1}\\
&\bbE [\|[c(\tilde x_k)]_+\|]\leq 10\tC_6^{1/2}s^{-1/2} k^{-1}+54\Lambda s^{-2/3} k^{-4/3}+ 15\tC_5^{1/2} s^{-5/6}k^{-5/3}, \label{thm3-ineq1-exp} \\
&\bbE [F(\tilde x_k)-F^*]\leq 2\tC_6 s^{-1/3}k^{-2/3}+4\tC_5 s^{-1}k^{-2}, \label{thm3-ineq2} \\
&\bbE [F(\tilde x_k)-F^*]\geq -\Lambda\big(10\tC_6^{1/2}s^{-1/2} k^{-1}+54\Lambda s^{-2/3} k^{-4/3}+ 15\tC_5^{1/2} s^{-5/6}k^{-5/3}\big), \label{thm3-ineq3}
\end{align}
where
\begin{equation}\label{c3-5}
\begin{aligned}
\tC_5&={4032L_{\nabla \barf}}\Big(\frac{D_F+\sqrt{2}\|[c(x^0)]_+\|^2/2}{40(L_{\nabla \barf}+\sqrt{2}L_{\nabla c^2})}+\frac{D_\cX^2}{2}+\frac{5\Lambda^2}{L_{\nabla c^2}}\Big),\\
\tC_6&=3L_{\nabla \barf}^{-1} L_{\nabla c^2}  \tC_5/32, \qquad \tC_7={8871L_{\nabla c^2}^{-1}L_{\nabla \barf}^2D_\cX^2},\\
\tC_8&=832L_{\nabla \barf}D_\cX^2, \qquad 
\tC_9=2688L_{\nabla c^2}^{-1}L_{\nabla \barf}^2D_\cX^2.
\end{aligned}
\end{equation}

\item[(ii)] Assume that
\begin{equation}\label{rho2}
\begin{aligned}
k_0=\lfloor \log_2 s \rfloor+1,\quad
    T_k=\left\{\begin{array}{ll}
       2^{k-1}  & \text{if} \  \ 1\leq k\leq k_0, \\
       2^{k_0-1}  & \text{if} \  \ k>k_0,
       \end{array}\right. 
    \quad
    \rho_k=\left\{\begin{array}{ll}
       2^{k/2}  & \text{if} \  \ 1\leq k\leq k_0, \\
       3\sqrt{s}(k-k_0+7)/16  & \text{if} \  \ k>k_0.
       \end{array}\right. 
\end{aligned}
\end{equation}
Then  for all $k\geq \max\{k_0+1,2(k_0-7)\}$, we have
\begin{align}
&\|[c(\tilde x_k)]_+\|\leq 60L_{\nabla \barf}^{1/2}D_{\cX}s^{-1/4}k^{-1/2}+ 150\Lambda s^{-1/2} k^{-1}\log k+ 7\big( 4\Lambda +\tC_{11}^{1/2}+\tC_{14}^{1/2}+\tC_{15}^{1/2}\big)s^{-1/4} k^{-1}\nn\\
&\qquad\qquad\quad\ \ +10\tC_{12}^{1/2}s^{-3/4}k^{-3/2}\log k+10\big(\tC_{10}^{1/2}+\tC_{13}^{1/2}\big)s^{-1/2}k^{-3/2}, \label{thm3-ineq4-1}\\
&\bbE[\|[c(\tilde x_k)]_+\|]
\leq 150\Lambda s^{-1/2} k^{-1}\log k+ \big(22\Lambda +7\tC_{11}^{1/2}\big)s^{-1/2} k^{-1}\nn\\
&\qquad\qquad\qquad\ \ +10\tC_{12}^{1/2}s^{-3/4}k^{-3/2}\log k+10\tC_{10}^{1/2}s^{-3/4}k^{-3/2},\label{thm3-ineq4} \\
&\bbE [F(\tilde x_k)-F^*]\leq 1046\Lambda^2s^{-1/2}k^{-1}\log k+ 2\tC_{11}s^{-1/2} k^{-1}+4\tC_{12}s^{-1}k^{-2}\log k+4\tC_{10}s^{-1} k^{-2},\label{thm3-ineq5}  \\
& \bbE [F(\tilde x_k)-F^*] \geq -\Lambda\Big(150\Lambda s^{-1/2} k^{-1}\log k+ \big(22\Lambda +7\tC_{11}^{1/2}\big)s^{-1/2} k^{-1}\nn\\
&\qquad\qquad\qquad\qquad\quad \ +10\tC_{12}^{1/2}s^{-3/4}k^{-3/2}\log k+10\tC_{10}^{1/2}s^{-3/4}k^{-3/2}\Big), \label{thm3-ineq6}
\end{align}
where
\begin{equation}\label{c6-9}
\begin{aligned}
\tC_{10}&={4032L_{\nabla \barf}}\Big(\frac{D_F+\sqrt{2}\|[c(x^0)]_+\|^2/2}{40(L_{\nabla \barf}+\sqrt{2}L_{\nabla c^2})}+\frac{D_\cX^2}{2}+\frac{5\Lambda^2k_0}{24L_{\nabla c^2}}\Big),\\
\tC_{11}&=3L_{\nabla \barf}^{-1} L_{\nabla c^2}  \tC_{10}/16, \qquad 
\tC_{12}=2823L_{\nabla c^2}^{-1}L_{\nabla \barf}\Lambda^2, \qquad \tC_{13}=2840L_{\nabla c^2}^{-1}L_{\nabla \barf}^2D_\cX^2, 
\\\qquad \tC_{14}&=533L_{\nabla \barf}D_\cX^2,\qquad
\tC_{15}= 896L_{\nabla c^2}^{-1}L_{\nabla \barf}^2D_\cX^2.
\end{aligned}
\end{equation}
\end{itemize}
\end{thm}

\begin{rem}  \label{remark4}
{(i) For similar reasons as discussed in Remark~\ref{remark3}(i) and (ii), the choices of $\rho_k$ in \eqref{rho1} and \eqref{rho2} are obtained by minimizing upper bounds on
\[
\max\left\{\|[c(\tilde x_k)]_+\|,  \big|\bbE[F(\tilde x_k)-F^*]\big|\right\}
\quad \text{and} \quad
\max\left\{\bbE[\|[c(\tilde x_k)]_+\|],  \big|\bbE[F(\tilde x_k)-F^*]\big|\right\},
\]
respectively, over all possible choices of the sequence $\{\rho_k\}$. These choices ensure that Algorithm~\ref{alg2} computes an $\epsilon$-SFSO or $\epsilon$-EFSO solution with low complexity.}

(ii) Assume that $m = \mathcal{O}(1)$. Using similar arguments as in Remark~\ref{remark3}(iii), it follows from Theorem~\ref{thm:finite-sum} that Algorithm~\ref{alg2} with $(k_0, T_k, \rho_k)$ chosen as in~\eqref{rho1} achieves an FO complexity of $\cO(s\log s+\sqrt{s}\epsilon^{-3/2})$ for computing both an $\epsilon$-SFSO solution and an $\epsilon$-EFSO solution of problem~\eqref{prob2}. In contrast, choosing $(k_0, T_k, \rho_k)$ as in~\eqref{rho2} results in an FO complexity of $\cO(s\log s+\sqrt{s}\epsilon^{-2})$ for finding an $\epsilon$-SFSO solution, and $\widetilde\cO(s\log s+\sqrt{s}\epsilon^{-1})$ for finding an $\epsilon$-EFSO solution.

(iii) In the special case $s=1$, Algorithm~\ref{alg2} with dynamic penalty parameters chosen in \eqref{rho2} becomes  \textit{the first single-loop first-order penalty method} with dynamic penalty parameters for solving deterministic constrained convex optimization problems, while achieving a nearly optimal FO complexity of $\widetilde{\mathcal{O}}(\epsilon^{-1})$.

{(iv) The parameters $k_0$, $\alpha_k$, $p_k$, $\gamma_k$, $q_i$, $\theta_t$, and $T_k$ in Theorem~\ref{thm:finite-sum} are either identical to or nearly proportional to those in Theorem~\ref{thm:finite-sum_fixed}. In addition, the parameter $\rho_k$ in Theorem~\ref{thm:finite-sum} increases gradually to $3s^{2/3}(K-k_0+7)^{4/3}/32$ or $3\sqrt{s}(K-k_0+7)/16$ as $k$
approaches $K$, attaining the same order as $\rho_k \equiv s^{2/3}K^{4/3}$ or $\rho_k \equiv \sqrt{s}K$
in Theorem~\ref{thm:finite-sum_fixed}, where $K>k_0$ is a prescribed iteration number. Consequently, the convergence rates in \eqref{thm3-ineq1}--\eqref{thm3-ineq3} and \eqref{thm3-ineq4-1}--\eqref{thm3-ineq6} with $k = K$ achieved by Algorithm~\ref{alg2} with dynamic penalty parameters are of similar form to those in \eqref{thm4-ineq1}--\eqref{thm4-ineq3} and \eqref{thm4-ineq4-1}--\eqref{thm4-ineq6} achieved with constant penalty parameters, respectively.

Setting $k = K$ in~\eqref{thm3-ineq1}--\eqref{thm3-ineq3}, we observe that Algorithm~\ref{alg2} with dynamic penalty parameters in~\eqref{rho1} achieves convergence rates of $\cO(s^{-1/3}K^{-2/3})$ for $\|[c(\tilde x_K)]_+\|$, $\cO(s^{-1/2}K^{-1})$ for $\bbE[\|[c(\tilde x_K)]_+\|]$, and $\cO(s^{-1/3}K^{-2/3})$ for $\big|\bbE[F(\tilde x_K) - F^*]\big|$, which match those achieved by Algorithm~\ref{alg2} with constant penalty parameters $\rho_k \equiv s^{2/3}K^{4/3}$.

Similarly, setting $k = K$ in~\eqref{thm3-ineq4-1}--\eqref{thm3-ineq6}, we observe that Algorithm~\ref{alg2} with dynamic penalty parameters in~\eqref{rho2} achieves convergence rates of $\cO(s^{-1/4}K^{-1/2})$ for $\|[c(\tilde x_K)]_+\|$, $\cO(s^{-1/2}K^{-1}\log K)$ for $\bbE[\|[c(\tilde x_K)]_+\|]$, and $\cO(s^{-1/2}K^{-1}\log K)$ for $\big|\bbE[F(\tilde x_K) - F^*]\big|$, which are at most worse by a logarithmic factor than those obtained by its constant-penalty counterpart.

(v) As discussed above, the convergence rates achieved by Algorithm~\ref{alg2} with dynamic penalty parameters are of the same order as, or at most worse by a logarithmic factor than, those obtained with constant penalty parameters. Nevertheless, Algorithm~\ref{alg2} with dynamic penalty parameters generally performs better in practice than its constant-penalty counterpart, as observed in our numerical experiments.
}
\end{rem}

\section{Numerical results} \label{sec:results}

In this section, we conduct preliminary experiments to test the performance of our proposed methods (namely, Algorithms \ref{alg1} and \ref{alg2}). Both algorithms are coded in Matlab and all the computations are performed on a desktop with a 2 GHz Intel i9-14900F 24-core processor and 32 GB of RAM.

\subsection{Constrained stochastic logistic regression model}\label{experiment-sto}
In this subsection, we consider the following constrained stochastic logistic regression model:
 \begin{equation}\label{logistic}
    \begin{aligned}
        \min_{w,b} \ & \bbE[\ell(w,b;\, X,Y)]+\lambda\|w\|_1 + \delta_{[-1,1]^{n+1}}(w,b)\\
     \mbox{s.t.} \ & -\tilde y_i(w^{\top}\tilde x_i+b) \leq 0, \quad  i=1,2,\dots,m 
    \end{aligned}
 \end{equation}
for some $\lambda>0$,  where $\delta_{[-1,1]^{n+1}}(\cdot)$ is the indicator function of the box set $[-1,1]^{n+1}$, $X\in \mathbb{R}^n$ is a standard normal random vector,  $Y\in\{-1,1\}$ is a binary random variable, and $\{(\tilde x_i,\tilde y_i)\}_{i=1}^m$ are $m$ core samples defined below,  and $\ell$ is the logistic loss function given by
\beq \label{loss}
\ell(w,b; x, y) = \log\big(1 + e^{-y (w^Tx+b)}\big).
\eeq

We set $\lambda = 0.1$, $n = 100$, and $m = 50$ for  \eqref{logistic}.  To generate the core samples $\{(\tilde x_i,\tilde y_i)\}_{i=1}^m$, we first generate a ground-truth parameter pair $(\hat w, \hat b)$, with each entry independently sampled from the uniform distribution on $[-1, 1]$. Given a realization $X = x$, the conditional distribution of $Y$ follows the logistic model:
\[
P(Y=1|X=x)=\frac{1}{1+e^{-(\hat w^Tx+\hat b)}}.
\]
Using this model, we generate 10{,}000 independent samples of $(X,Y)$, where $X$ is drawn from the standard normal distribution. Among these samples, we select the $m$ points closest to the decision boundary $\hat w^T x + \hat b = 0$ and use them as the core samples. To ensure feasibility of problem~\eqref{logistic}, we set $\tilde y_i=\operatorname{sgn}(\hat w^T \tilde x_i + \hat b)$, where $\operatorname{sgn}(\cdot)$ denotes the sign function. 

We apply Algorithm~\ref{alg1} to solve \eqref{logistic} under both constant and dynamic penalty parameter settings, and compare its performance with the accelerated stochastic approximation (AC-SA) method  \cite{lan2012optimal} and the constraint extrapolation (ConEx) method  \cite{boob2023stochastic}. The parameters of Algorithm~\ref{alg1} are chosen according to~\eqref{def-para2} with $K = 5{,}000$ for the constant penalty setting, and according to~\eqref{def-para1} for the dynamic penalty setting. The parameters of AC-SA and ConEx are set following \cite[Corollary~1]{lan2012optimal} and \cite[Theorem~2.3]{boob2023stochastic}, respectively. All methods are initialized at $(w_0, b_0)$, where each component is independently drawn from the uniform distribution on $[-1,1]$.

\begin{figure}[h]  
    \centering
    \begin{subfigure}{0.45\textwidth}
        \centering
        \includegraphics[width=\linewidth]{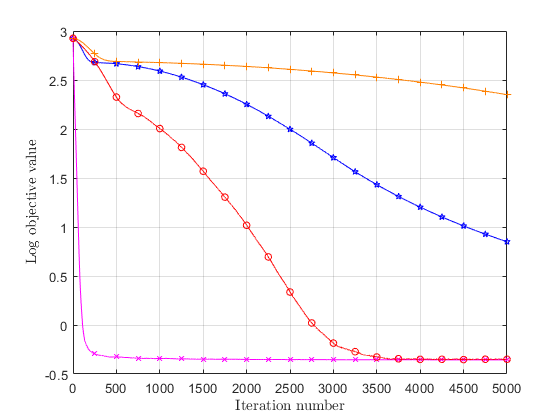} 
    \end{subfigure}
    \hfill
    \begin{subfigure}{0.45\textwidth}
        \centering
        \includegraphics[width=\linewidth]{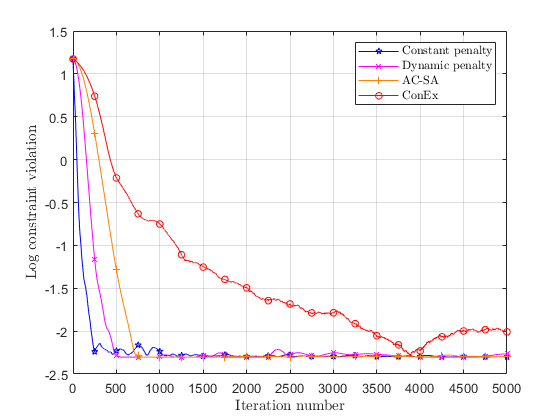} 
    \end{subfigure}
    \caption{Performance of Algorithm \ref{alg1} with constant and dynamic penalty parameters. }
     \label{fig:1}
\end{figure}

To compare the performance of Algorithm~\ref{alg1}, AC-SA, and ConEx for solving \eqref{logistic}, we evaluate the objective value and constraint violation at each iteration. The objective value at a point $(w,b)$ is approximated by averaging $\ell(w,b; X,Y)$ over the 10{,}000 previously generated samples of $(X,Y)$. The results are reported in Figure~\ref{fig:1}. Algorithm~\ref{alg1} with dynamic penalty parameters converges significantly faster than AC-SA, ConEx, and its constant-penalty counterpart in terms of objective value. In terms of constraint violation, the dynamic-penalty variant converges slightly more slowly than the constant-penalty variant, but still outperforms both AC-SA and ConEx.

\subsection{Constrained finite-sum logistic regression model}

In this subsection, we consider the following constrained finite-sum logistic regression model:
 \begin{equation}\label{logistic-finitesum}
    \begin{aligned}
        \min_{w,b } \; & \frac{1}{s}\sum_{i=1}^{s} \ell(w,b;\, x_i, y_i)+\lambda\|w\|_1+ \delta_{[-1,1]^{n+1}}(w,b) \\
     \mbox{s.t.}  \ & -\tilde y_i(w^{T}\tilde x_i+b) \leq 0, \quad i=1,2,\dots,m
    \end{aligned}
 \end{equation}
for some $\lambda>0$, where $\delta_{[-1,1]^{n+1}}(\cdot)$ is defined in \eqref{logistic} , $x_i \in \rr^n$ and $y_i \in \{-1,1\}$, $i=1,\ldots, s$, are the sample data, $\tilde x_i \in \rr^n$ and $\tilde y_i \in \{-1,1\}$, $i=1,\ldots, m$, are defined below, and $\ell$ is the logistic loss function given in \eqref{loss}.

We set $\lambda=0.03$ and $m=50$ for \eqref{logistic-finitesum}. The dataset $\{(x_i,y_i)\}_{i=1}^s$ consists of $s = 32,561$ samples with $n = 123$ features from the  \texttt{a1a} dataset in LIBSVM \cite{chang2011libsvm}. In addition, the set $\{(\tilde x_i, \tilde y_i)\}_{i=1}^m$ contains $m$ core samples generated as follows. We first solve the unconstrained counterpart of \eqref{logistic-finitesum} (with $\lambda = 0$ and $m = 0$) to obtain a reference solution $(\hat w, \hat b)$. We then select the $m$ samples closest to the decision boundary $\hat w^T x + \hat b = 0$ and use them as the core samples. To ensure feasibility of problem~\eqref{logistic-finitesum}, we set $\tilde y_i = \operatorname{sgn}(\hat w^T \tilde x_i + \hat b)$ for $i=1,\ldots, m$. 

We now apply Algorithm~\ref{alg2} to solve  problem \eqref{logistic-finitesum} under both constant and dynamic penalty parameter settings. The algorithm is initialized at $(\hat w_0, \hat b_0)$, where each component is independently drawn from the uniform distribution on $[-1, 1]$. The parameters for Algorithm~\ref{alg2} are set according to Theorem~\ref{thm:finite-sum_fixed} with $K = 500$ for the constant penalty setting and Theorem~\ref{thm:finite-sum} for the dynamic penalty setting, respectively.

\begin{figure}[htbp] 
    \centering
    \begin{subfigure}{0.45\textwidth}
        \centering
        \includegraphics[width=\linewidth]{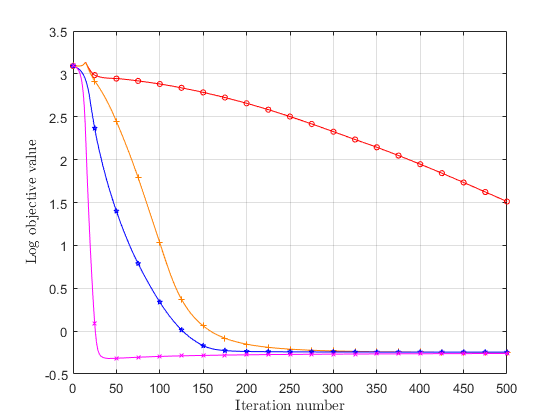} 
    \end{subfigure}
    \hfill
    \begin{subfigure}{0.45\textwidth}
        \centering
        \includegraphics[width=\linewidth]{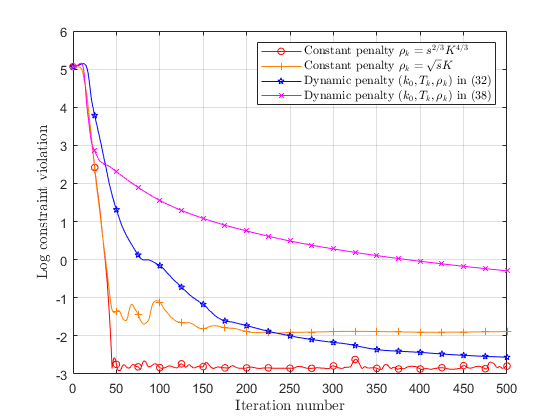} 
    \end{subfigure}
    \caption{Performance of Algorithm \ref{alg2} with constant and dynamic penalty parameters. }
    \label{fig:2}
\end{figure}

We run Algorithm~\ref{alg2} for $500$ iterations and evaluate the objective value and constraint violation at each iterate. The results are presented in Figure~\ref{fig:2}, which illustrates the convergence behavior of the algorithm under both constant and dynamic penalty settings. As observed, the dynamic penalty schemes generally lead to faster convergence in objective value compared to their constant counterparts, while incurring only a slightly slower convergence in constraint violation. In addition, both the constant penalty settings $\rho_k \equiv s^{2/3}K^{4/3}$ and its dynamic counterpart given in \eqref{rho1} outperform the constant penalty $\rho_k \equiv \sqrt{s}K$ and its corresponding dynamic version given in \eqref{rho2} in terms of constraint violation convergence, although they exhibit slower convergence in objective value.

We also compare Algorithm~\ref{alg2} with an optimal deterministic first-order method---the extra anchored gradient (EAG) method---for solving problem~\eqref{logistic-finitesum}. EAG, proposed in~\cite[Algorithm 3]{kovalev2022first}, is an optimal first-order method designed for monotone inclusion problems. Since problem~\eqref{logistic-finitesum} can be equivalently reformulated as a monotone inclusion problem, EAG is well-suited for solving it.  The parameters of Algorithm~\ref{alg2} are set according to Theorem~\ref{thm:finite-sum_fixed}(ii) with $K=500$ for the constant penalty setting, and Theorem~\ref{thm:finite-sum}(ii) for the dynamic penalty setting. The parameters for EAG are chosen as specified in~\cite[Theorem 2]{kovalev2022first}. Both Algorithm~\ref{alg2} and EAG are run for $500$ iterations. The results are presented in Figure~\ref{fig:3}, which illustrates the convergence behavior of the methods as the number of gradient evaluations increases. It can be observed that Algorithm~\ref{alg2}, under both constant and dynamic penalty settings, significantly outperforms EAG in terms of both objective value and constraint violation.

\begin{figure}[htbp] 
    \centering
    \begin{subfigure}{0.45\textwidth}
        \centering
        \includegraphics[width=\linewidth]{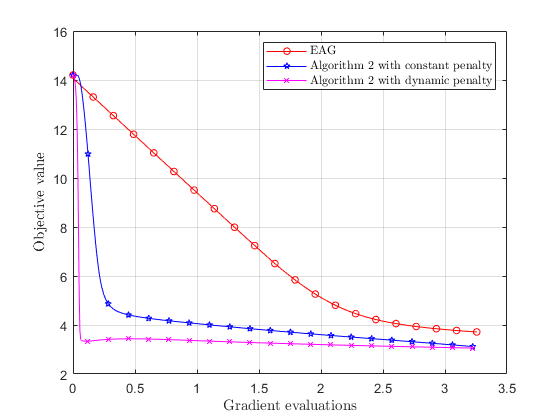} 
    \end{subfigure}
    \hfill
    \begin{subfigure}{0.45\textwidth}
        \centering
        \includegraphics[width=\linewidth]{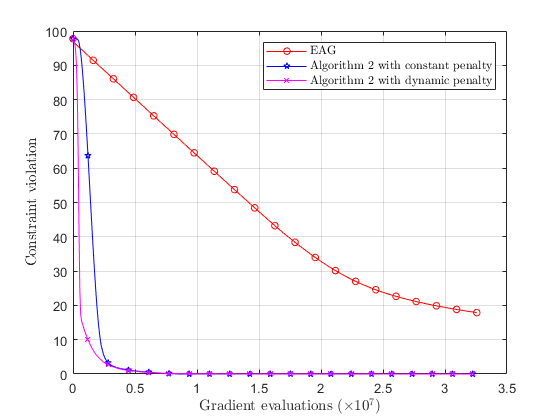} 
    \end{subfigure}
    \caption{Comparison of Algorithm \ref{alg2} with EAG}
    \label{fig:3}
\end{figure}

\section{Proof of the main results}\label{sec:proof}
In this section, we provide a proof of our main results presented in Sections \ref{sec:composite} and \ref{sec:finite-sum}, which are particularly Theorems  \ref{thm:composite_fixed}, \ref{thm:composite}, \ref{thm:finite-sum_fixed}, and  \ref{thm:finite-sum}. 

To proceed, we establish a lemma characterizing the relationship between problems~\eqref{prob1} and~\eqref{def-Frho1}. In particular, we show that, at any given point, the feasibility violation and optimality gap of~\eqref{prob1} can be bounded by the optimality gap of~\eqref{def-Frho1} with $\rho_k$ replaced by $\rho$. These results will be used to derive complexity guarantees for Algorithms~\ref{alg1} and~\ref{alg2} for computing $\epsilon$-EFSO and $\epsilon$-SFSO solutions, by appropriately bounding $F_{\rho_k}(x_k)-F_{\rho_k}^*$ and its expectation for Algorithm~\ref{alg1}, and $F_{\rho_k}(\tilde{x}_k)-F_{\rho_k}^*$ and its expectation for Algorithm~\ref{alg2}.

\begin{lemma}\label{l-penalty}
Consider the penalty problem associated with \eqref{prob1}:
\beq \label{penalty-prob}
F_\rho^* = \min_x\Big\{F_\rho(x) := F(x)+\frac{\rho}{2}\|[c(x)]_+\|^2\Big\},
\eeq
where $F$ is defined in  \eqref{prob1}, and $\rho>0$. Suppose that Assumption \ref{a1}(iii) holds, and let $\Lambda$ be given in Assumption \ref{a1}(iii).  Then for any $x\in\cX$ and $\rho>0$, one has
\begin{align}
 &-\Lambda\|[c(x)]_+\| \leq F(x)-F^*\leq F_{\rho}(x)-F_{\rho}^*, \label{opt-gap}\\
  & \|[c(x)]_+\| \leq \frac{2\Lambda}{\rho}+\sqrt{\frac{2(F_{\rho}(x)-F_{\rho}^*)}{\rho}}. \label{feasibility}
\end{align}

\end{lemma}
\begin{proof}
By \eqref{prob1} and \eqref{def-Frho1}, one has $F_{\rho}(x)\geq F(x)$ and $F_{\rho}^*\leq F^*$, which imply that  
\begin{equation}\label{l1-0}
    F(x)-F^*\leq F_{\rho}(x)-F_{\rho}^*.
\end{equation}
By the definition of $\Lambda$ in Assumption \ref{a1}(iii), there exists an optimal Lagrange multiplier $\lambda^*$ for problem  \eqref{prob1} satisfying $\Lambda=\|\lambda^*\|$. Hence, one has that $\lambda^* \geq 0$ and
\begin{equation*}
    F^*=\min_z \{F(z)+\langle\lambda^*,c(z)\rangle\} \leq F(x)+\langle\lambda^*, c(x)\rangle.
\end{equation*}
Using these {facts} and $\Lambda=\|\lambda^*\|$, we obtain that
\[
F(x)-F^* \geq -\langle\lambda^*, c( x)\rangle \geq -\langle\lambda^*,[c( x)]_+\rangle 
\geq -\Lambda\|[c( x)]_+\|,
\]
which together with \eqref{l1-0} leads to \eqref{opt-gap}. Also, by this relation, $F_\rho(x)=F(x)+\rho\|[c(x)]_+\|^2/2$, and $F_{\rho}^*\leq F^*$, one has
\[
 -\Lambda\|[c(x)]_+\| \leq F(x)-F^* = F_{\rho}(x)-\frac{\rho}{2}\|[c(x)]_+\|^2-F^*\leq F_{\rho}(x) -\frac{\rho}{2}\|[c(x)]_+\|^2- F_{\rho}^*.
\]
{Hence, we have} 
\[
    \frac{\rho}{2}\|[c(x)]_+\|^2-\Lambda\|[c(x)]_+\|+F_{\rho}^*-F_{\rho}(x) \leq 0.
\]
It follows from this {inequality} and $\sqrt{a^2+b^2}\leq |a|+|b|$ that
\[
\|[c(x)]_+\| \leq \frac{\Lambda}{\rho}+\frac{\sqrt{\Lambda^2+2\rho(F_{\rho}(x)-F_{\rho}^*)}}{\rho} 
\leq \frac{2\Lambda}{\rho}+\sqrt{\frac{2(F_{\rho}(x)-F_{\rho}^*)}{\rho}}.
\]
\end{proof}

\subsection{Proof for the main results in Section \ref{sec:composite}}\label{sec:proof-composite}

In this subsection, we first establish several technical lemmas and then use them to prove Theorems \ref{thm:composite_fixed} and  \ref{thm:composite}.

For notational convenience, let 
\beq \label{L-rho}
f_{\rho}(x) := \bbE[\tilde f(x,\xi)]+\frac{\rho}{2}\|[c(x)]_+\|^2, \qquad L_{\rho} := L_{\nabla f}+\rho L_{\nabla c^2}
\eeq
for all $\rho>0$. It follows from these  {expressions}, the $L_{\nabla f}$-smoothness of $\bbE[\tilde f(\cdot,\xi)]$, and the $L_{\nabla c^2}$-smoothness of $\|[c(\cdot)]_+\|^2/2$ that $f_{\rho}$ is $L_{\rho}$-smooth.  One can also observe that $F_{\rho}(\cdot)=f_{\rho}(\cdot)+\psi(\cdot)$, where $F_\rho$ is defined in \eqref{penalty-prob}. These facts will be used in our subsequent analysis.  

The following lemma establishes upper bounds on
\[
\beta_k \gamma_k \big(F_{\rho_k}(x_{k+1}) - F_{\rho_k}(x)\big)
+ \frac{1}{2}\|x - z_{k+1}\|^2
\]
and its expectation for all $x \in \cX$ and $k \geq 1$. These bounds
will be used to derive upper bounds on $F_{\rho_k}(x_k) - F_{\rho_k}^*$ and
its expectation.

\begin{lemma}\label{l-iter}
Suppose that Assumptions \ref{a1} and \ref{a2} hold. Let $\{x_k\}$, $\{y_k\}$ and $\{z_k\}$ be generated in Algorithm \ref{alg1} with $\beta_k$, $ \gamma_k$, and $\rho_k$ satisfying $\beta_k\geq1$ and $L_{\rho_k}\gamma_k-\beta_k<0$ for all $k\geq1$, where $L_{\rho_k}$ is  defined in \eqref{L-rho}. Then for all $x\in\cX$ and $k\geq 1$, we have
\begin{align}
 & \beta_k\gamma_k\big(F_{\rho_k}(x_{k+1})-F_{\rho_k}(x)\big)+\frac{1}{2}\|x-z_{k+1}\|^2\leq(\beta_k-1)\gamma_k\big(F_{\rho_k}(x_k)-F_{\rho_k}(x)\big)+\frac{1}{2}\|x-z_k\|^2+\Delta_k(x), \label{determ-ineq}\\
 &  \bbE\Big[\beta_k\gamma_k\big(F_{\rho_k}(x_{k+1})-F_{\rho_k}(x)\big)+\frac{1}{2}\|x-z_{k+1}\|^2\Big]\leq \bbE\Big[(\beta_k-1)\gamma_k\big(F_{\rho_k}(x_k)-F_{\rho_k}(x)\big)+\frac{1}{2}\|x-z_k\|^2\Big] \nn \\
    &\qquad\qquad\qquad\qquad\qquad\qquad\qquad\qquad\qquad\qquad \quad+\frac{\beta_k\gamma_k^2}{2(\beta_k-L_{\rho_k}\gamma_k)}\bbE[\|\delta_k\|^2], \label{stoch-ineq}
\end{align}
where $F_{\rho_k}$ and $f_{\rho_k}$  are  defined in \eqref{penalty-prob} and \eqref{L-rho}, respectively, and 
\begin{equation}\label{delta}
    \delta_k=g_k-\nabla f_{\rho_k}(y_k), \quad \Delta_k(x)=\gamma_k\langle \delta_k, x-z_{k+1}\rangle-\frac{\beta_k-L_{\rho_k}\gamma_k}{2\beta_k}\|z_{k+1}-z_k\|^2.
\end{equation}
\end{lemma}

\begin{proof}
Observe from the definitions of $x_{k+1}$ and $y_k$ in Algorithm \ref{alg1} that
\begin{equation}\label{l1-1}
    x_{k+1}-y_k=\beta_k^{-1}(z_{k+1}-z_k).
\end{equation}
{Using this relation, together with the definition of $x_{k+1}$ and the convexity and $L_{\rho_k}$-smoothness of $f_{\rho_k}$, we obtain that}
\begin{align}
    &f_{\rho_k}(x_{k+1})\leq f_{\rho_k}(y_k)+\langle\nabla f_{\rho_k}(y_k), x_{k+1}-y_k\rangle+\frac{L_{\rho_k}}{2}\|x_{k+1}-y_k\|^2 \nn\\
    &= f_{\rho_k}(y_k)+\beta_k^{-1}\langle\nabla f_{\rho_k}(y_k), z_{k+1}-y_k\rangle+(1-\beta_k^{-1})\langle\nabla f_{\rho_k}(y_k), x_k-y_k\rangle+\frac{L_{\rho_k}}{2}\|x_{k+1}-y_k\|^2 \nn\\
    &= \beta_k^{-1}\left(f_{\rho_k}(y_k)+\langle\nabla f_{\rho_k}(y_k), z_{k+1}-y_k\rangle\right)+(1-\beta_k^{-1})\left(f_{\rho_k}(y_k)+\langle\nabla f_{\rho_k}(y_k), x_k-y_k\rangle\right)+\frac{L_{\rho_k}}{2}\|x_{k+1}-y_k\|^2 \nn\\
    &\leq \beta_k^{-1} \left(f_{\rho_k}(y_k)+\langle\nabla f_{\rho_k}(y_k), z_{k+1}-y_k\rangle\right) + (1-\beta_k^{-1})f_{\rho_k}(x_k) +\frac{L_{\rho_k}}{2}\|x_{k+1}-y_k\|^2\nn\\
    &=\beta_k^{-1} \left(f_{\rho_k}(y_k)+\langle g_k, z_{k+1}-y_k\rangle\right) +\beta_k^{-1}\langle\delta_k, y_k-z_{k+1}\rangle+(1-\beta_k^{-1})f_{\rho_k}(x_k) +\frac{L_{\rho_k}\beta_k^{-2}}{2}\|z_{k+1}-z_k\|^2, \label{l1-2}
\end{align}
where the first inequality is due to the $L_{\rho_k}$-smoothness of $f_{\rho_k}$,  the first equality follows from the definition of $x_{k+1}$ in Algorithm \ref{alg1}, and the second inequality is due to $\beta_k\geq 1$ and the convexity of $f_{\rho_k}$, while the last equality follows from \eqref{delta} and \eqref{l1-1}.

In addition, by the definition of $z_{k+1}$ in Algorithm \ref{alg1}, one has
\begin{equation*}
    z_{k+1}=\argmin_x\left\{\frac{1}{2\gamma_k}\|x-(z_k-\gamma_k g_k)\|^2+\psi(x)\right\},
\end{equation*}
which, together with the strong convexity of its objective function, implies that for all $x\in \cX$,
\begin{equation*}
    \frac{1}{2\gamma_k}\|z_{k+1}-z_k+\gamma_k g_k\|^2+\psi(z_{k+1})\leq \frac{1}{2\gamma_k}\|x-z_k+\gamma_k g_k\|^2+\psi(x) -\frac{1}{2\gamma_k}\|x-z_{k+1}\|^2.
\end{equation*}
It then follows that
\begin{equation*}
\langle g_k, z_{k+1}-y_k\rangle \leq \langle g_k, x-y_k\rangle +\psi(x)- \frac{1}{2\gamma_k}\|z_{k+1}-z_k\|^2+\frac{1}{2\gamma_k}\|x-z_k\|^2 -\frac{1}{2\gamma_k}\|x-z_{k+1}\|^2-\psi(z_{k+1}).
\end{equation*}
Using this {relation}, \eqref{delta}, and the convexity of $f_{\rho_k}$, we have
\begin{align*}
    f_{\rho_k}(y_k)+\langle g_k, z_{k+1}-y_k\rangle&\leq f_{\rho_k}(y_k) +\langle \nabla f_{\rho_k}(y_k), x-y_k\rangle+\langle \delta_k, x-y_k\rangle+\psi(x)-\frac{1}{2\gamma_k}\|z_{k+1}-z_k\|^2 \\
    &\quad +\frac{1}{2\gamma_k}\|x-z_k\|^2-\frac{1}{2\gamma_k}\|x-z_{k+1}\|^2 -\psi(z_{k+1})\\
    & {\leq F_{\rho_k}(x) +\langle \delta_k, x-y_k\rangle -\frac{1}{2\gamma_k}\|z_{k+1}-z_k\|^2  +\frac{1}{2\gamma_k}\|x-z_k\|^2 -\frac{1}{2\gamma_k}\|x-z_{k+1}\|^2} \\ 
&\quad {-\psi(z_{k+1}).}
\end{align*}
This together with \eqref{l1-2} implies that
\begin{align*}
    f_{\rho_k}(x_{k+1})&\leq \beta_k^{-1} \Big(F_{\rho_k}(x) +\langle \delta_k, x-y_k\rangle -\frac{1}{2\gamma_k}\|z_{k+1}-z_k\|^2+\frac{1}{2\gamma_k}\|x-z_k\|^2-\frac{1}{2\gamma_k}\|x-z_{k+1}\|^2 -\psi(z_{k+1})\Big) \\
    &\quad +(1-\beta_k^{-1})f_{\rho_k}(x_k)+\beta_k^{-1}\langle\delta_k, y_k-z_{k+1}\rangle +\frac{L_{\rho_k}\beta_k^{-2}}{2}\|z_{k+1}-z_k\|^2 \\
    &= \beta_k^{-1}F_{\rho_k}(x)+(1-\beta_k^{-1})f_{\rho_k}(x_k) +\beta_k^{-1}\langle \delta_k, x-z_{k+1}\rangle +\frac{1}{2}\Big(L_{\rho_k}\beta_k^{-2}-\frac{\beta_k^{-1}}{\gamma_k}\Big)\|z_{k+1}-z_k\|^2\\
    &\quad +\frac{\beta_k^{-1}}{2\gamma_k}\|x-z_k\|^2-\frac{\beta_k^{-1}}{2\gamma_k}\|x-z_{k+1}\|^2 - \beta_k^{-1}\psi(z_{k+1}).
\end{align*}
Also, by $\beta_k \geq 1$, the definition of $x_{k+1}$, and the convexity of $\psi$, one has
\begin{equation*}
    \psi(x_{k+1})=\psi\big(\beta_k^{-1}z_{k+1}+(1-\beta_k^{-1}) x_k\big)\leq \beta_k^{-1}\psi(z_{k+1})+(1-\beta_k^{-1})\psi(x_k).
\end{equation*}
Summing up the above two inequalities gives 
\begin{align*}
   F_{\rho_k}(x_{k+1})&\leq \beta_k^{-1}F_{\rho_k}(x)+(1-\beta_k^{-1})F_{\rho_k}(x_k)+\beta_k^{-1}\langle \delta_k, x-z_{k+1}\rangle +\frac{1}{2}\Big(L_{\rho_k}\beta_k^{-2}-\frac{\beta_k^{-1}}{\gamma_k}\Big)\|z_{k+1}-z_k\|^2 \\
    &\quad + \frac{\beta_k^{-1}}{2\gamma_k}\|x-z_k\|^2-\frac{\beta_k^{-1}}{2\gamma_k}\|x-z_{k+1}\|^2.
\end{align*}
Multiplying this by $\beta_k\gamma_k$ and rearranging the terms lead to the conclusion \eqref{determ-ineq}. 

We next show that \eqref{stoch-ineq} holds. To this end, one can first observe that 
\begin{align}
   \Delta_k(x)  &\overset{\eqref{delta}}= \gamma_k\langle \delta_k, x-z_k\rangle +\gamma_k\langle \delta_k, z_{k+1}-z_k\rangle-\frac{\beta_k-L_{\rho_k}\gamma_k}{2\beta_k}\|z_{k+1}-z_k\|^2 \nn \\
    &\leq \gamma_k\langle \delta_k, x-z_k\rangle + \frac{\beta_k\gamma_k^2\|\delta_k\|^2}{2(\beta_k-L_{\rho_k}\gamma_k)}, \label{Delta-ineq}
\end{align}
where the inequality follows from $\beta_k-L_{\rho_k}\gamma_k>0$ and Young's inequality. 
Let $\Xi_k=\{\xi_1,\xi_2,\dots,\xi_{k-1}\}$ denote the collection of samples drawn up to iteration $k-1$ in Algorithm 
\ref{alg1}.  It then follows from \eqref{delta}, $g_k=\nabla \tilde f (y_k,\xi_k)+{\rho_k}  \nabla c(y_k)[c(y_k)]_+$, $y_k\in\cX$, and Assumption \ref{a2} that for all $x\in\cX$, 
\begin{equation*}
    \bbE_{{\xi_k}}[\langle\delta_k, x-z_k\rangle|\Xi_k]=\bbE_{{\xi_k}}[\langle\nabla \tilde f(y_k,\xi_k)-\nabla f(y_k), x-z_k\rangle|\Xi_k]=\big\langle \bbE_{{\xi_k}}[\nabla \tilde f(y_k,\xi_k)|\Xi_k]-\nabla f(y_k), x-z_k\big\rangle=0.
\end{equation*}
This and \eqref{Delta-ineq} immediately imply that
\[
\bbE_{{\xi_k}}[\Delta_k(x)|\Xi_k] \leq  \frac{\beta_k\gamma_k^2\bbE_{{\xi_k}}[\|\delta_k\|^2|\Xi_k]}{2(\beta_k-L_{\rho_k}\gamma_k)}.
\]
Using this relation and taking the conditional expectation with respect to $\xi_k$, conditioned on $\Xi_k$, on both sides of~\eqref{determ-ineq}, we obtain
\begin{align*}
    \bbE_{{\xi_k}}\Big[\beta_k\gamma_k(F_{\rho_k}(x_{k+1})-F_{\rho_k}(x))+\frac{1}{2}\|x-z_{k+1}\|^2\Big|\Xi_k\Big]&\leq \bbE_{{\xi_k}}\Big[(\beta_k-1)\gamma_k(F_{\rho_k}(x_k)-F_{\rho_k}(x))+\frac{1}{2}\|x-z_k\|^2\Big|\Xi_k\Big]\\
        &\quad +\frac{\beta_k\gamma_k^2\bbE_{{\xi_k}}[\|\delta_k\|^2|\Xi_k]}{2(\beta_k-L_{\rho_k}\gamma_k)}.
\end{align*}
Taking expectation with respect to $\Xi_k$ on both sides of this inequality yields~\eqref{stoch-ineq}.
\end{proof}

In view of Lemma~\ref{l-iter}, we next derive upper bounds on
$F_\rho(x_k) - F_\rho^*$ and its expectation under the setting where
Algorithm~\ref{alg1} uses a constant penalty parameter $\rho_k \equiv \rho$
for some $\rho > 0$. This result, together with Lemma~\ref{l-penalty},
will be used to prove Theorem~\ref{thm:composite_fixed}.

\begin{lemma}\label{l-sum-fixed}
Suppose that Assumptions \ref{a1} and \ref{a2} hold.  Let $\{x_k\}$ be generated in Algorithm \ref{alg1} with $\rho_k \equiv \rho$,  and $\beta_k\in [1,\infty)$ and $\gamma_k>0$ satisfying $\beta_1=1$, $0<(\beta_{k+1}-1)\gamma_{k+1}\leq \beta_k\gamma_k$, and $2L_\rho\gamma_k\leq\beta_k$ for some $\rho>0$ and all $k\geq 1$,  where $L_{\rho}$ is defined in \eqref{L-rho}.  Then for all $k\geq 2$, we have
\begin{align*}
 & F_\rho(x_k)-F^*_\rho \leq\frac{1}{(\beta_k-1)\gamma_k}\Big(\frac{1}{2}D_\cX^2+2GD_\cX\sum_{i=1}^{k-1}\gamma_i\Big),  \\
 &  \bbE[ F_\rho(x_k)-F^*_\rho] \leq \frac{1}{(\beta_k-1)\gamma_k}\Big(\frac{1}{2}D_\cX^2+\sigma^2\sum_{i=1}^{k-1}\gamma_i^2\Big),
\end{align*}
where $F_\rho$ and $F^*_\rho$ are defined in \eqref{penalty-prob}, and $D_\cX$, $\sigma$ and $G$ are given in Assumptions \ref{a1} and \ref{a2}, respectively.
\end{lemma}

\begin{proof}
Since $2L_\rho\gamma_k\leq\beta_k$ and $\beta_k\geq 1$ for all $k$, the assumption of Lemma \ref{l-iter} holds.  Let $x^*_\rho\in\argmin\limits_{x} F_\rho(x)$. It follows from this and \eqref{penalty-prob} that $F^*_\rho=F_\rho(x^*_\rho)$ and $F_\rho(x_k) \geq F^*_\rho$ for all $k$. {Using these facts}, \eqref{determ-ineq} with $x=x^*_\rho$, \eqref{delta}, and the assumption that $\rho_k= \rho$ and $0<(\beta_{k+1}-1)\gamma_{k+1}\leq\beta_k\gamma_k$ for all $k$,  {we obtain that} 
\begin{align}
    &(\beta_{i+1}-1)\gamma_{i+1}(F_\rho(x_{i+1})-F^*_\rho)+\frac{1}{2}\|x^*_\rho-z_{i+1}\|^2 \leq \beta_i\gamma_i(F_\rho(x_{i+1})-F^*_\rho)+\frac{1}{2}\|x^*_\rho-z_{i+1}\|^2\nn\\
    &\leq (\beta_i-1)\gamma_i(F_\rho(x_i)-F^*_\rho)+\frac{1}{2}\|x^*_\rho-z_i\|^2+\gamma_i\langle \delta_i, x^*_\rho-z_{i+1}\rangle-\frac{\beta_i-L_\rho\gamma_i}{2\beta_i}\|z_{i+1}-z_i\|^2\nn\\
    &\leq (\beta_i-1)\gamma_i(F_\rho(x_i)-F^*_\rho)+\frac{1}{2}\|x^*_\rho-z_i\|^2 + \gamma_i\|\delta_i\|\|x^*_\rho-z_{i+1}\| \qquad \forall i \geq 1.\label{rec}
\end{align}
In addition, by \eqref{delta} and Assumption \ref{a2}, one can observe that 
\begin{equation}\label{deltabound1}
    \|\delta_i\|\overset{\eqref{delta}}{=}\|\nabla \tilde f(y_i, \xi_i)-\bbE[\nabla \tf(y_i,\xi)]\|\leq \|\nabla \tilde f(y_i, \xi_i)\|+\bbE[\|\nabla \tf(y_i,\xi)\|]\leq 2G.
\end{equation}
Using this {relation}, $\beta_1=1$, $\|x^*_\rho-z_1\|\leq D_\cX$, and summing up \eqref{rec} from $i=1$ to $k-1$, we obtain that
\begin{align*}
    (\beta_k-1)\gamma_k(F_\rho(x_k)-F^*_\rho)&\leq (\beta_1-1)\gamma_1(F_\rho(x_1)-F^*_\rho)+\frac{1}{2}\|x^*_\rho-z_1\|^2+\sum_{i=1}^{k-1}\gamma_i\|\delta_i\|\|x^*_\rho-z_{i+1}\| \\
    &\leq \frac{D_\cX^2}{2}+2GD_\cX\sum_{i=1}^{k-1}\gamma_i,
\end{align*}
which implies that the first conclusion of this lemma holds.

We next show that the second conclusion of this lemma holds.  Indeed, by $2L_\rho\gamma_k\leq\beta_k$ and $\beta_k\geq1$, one has $\beta_k/(\beta_k-L_\rho\gamma_k)\leq 2$. Using this {relation},  \eqref{stoch-ineq}, $0<(\beta_{k+1}-1)\gamma_{k+1}\leq\beta_k\gamma_k$, and $F_\rho(x) \geq F^*_\rho$, we obtain that
\begin{align*}
&  \bbE\Big[(\beta_{i+1}-1)\gamma_{i+1}(F_\rho(x_{i+1})-F^*_\rho)+\frac{1}{2}\|x^*_\rho-z_{i+1}\|^2\Big] \leq  \bbE\Big[\beta_i\gamma_i (F_\rho(x_{i+1})-F^*_\rho) +\frac{1}{2}\|x^*_\rho-z_{i+1}\|^2\Big] \\
    &\leq  \bbE\Big[(\beta_i-1)\gamma_i(F_\rho(x_i)-F^*_\rho)
     +\frac{1}{2}\|x^*_\rho-z_{i}\|^2 \Big]+\gamma_i^2\bbE[\|\delta_i\|^2],
\end{align*}
where the second inequality follows from \eqref{stoch-ineq} and $\beta_i/(\beta_i-L_\rho\gamma_i)\leq 2$. 
Using $\beta_1=1$, $\bbE[\|\delta_i\|^2]\leq\sigma^2$, $\|x^*_\rho-z_1\|\leq D_\cX$, and summing up the above inequalities from $i=1$ to $k-1$, we have
\begin{equation*}
    (\beta_k-1)\gamma_k\bbE[F_\rho(x_k)-F^*_\rho]\leq \frac{1}{2}\|x^*_\rho-z_{1}\|^2+\sum_{i=1}^{k-1}\gamma_i^2\bbE[\|\delta_i\|^2] \leq \frac{1}{2}D_\cX^2+\sigma^2\sum_{i=1}^{k-1}\gamma_i^2,
\end{equation*}
which implies that the second conclusion of this lemma holds. 
\end{proof}

We are now ready to prove Theorem \ref{thm:composite_fixed}.

\begin{proof}[\textbf{Proof of Theorem \ref{thm:composite_fixed}}]
Notice from \eqref{def-para2} that $\rho_k = \rho$, $\beta_k=(k+1)/2$, and $\gamma_k=(k+1)/(4L_\rho)$, where $L_\rho=L_{\nabla f}+\rho L_{\nabla c^2}$ and $\rho=K^{3/2}$. Observe that $0<\big((k+2)/2-1\big)(k+2) \leq (k+1)^2/2$ for all $k\geq 1$, which implies $0<(\beta_{k+1}-1)\gamma_{k+1}\leq \beta_k\gamma_k$. In addition, $2L_\rho\gamma_k\leq\beta_k$ clearly holds. Hence, the assumption of Lemma \ref{l-sum-fixed} holds for such $\rho_k$, $\beta_k$, and $\gamma_k$.  It then follows from $K\geq 2$ and Lemma \ref{l-sum-fixed} with such $\rho_k$, $\beta_k$, and $\gamma_k$ that 
\begin{align}
    F_{\rho}(x_K)-F_{\rho}^*&\leq\frac{1}{(\beta_K-1)\gamma_K}\Big(\frac{1}{2}D_\cX^2+2GD_\cX\sum_{i=1}^{K-1}\gamma_i\Big) \nn \\
    &=\frac{4L_\rho}{((K+1)/2-1) (K+1)}\Big(\frac{1}{2}D_\cX^2+2GD_\cX\sum_{i=1}^{K-1}\frac{i+1}{4L_\rho}\Big)\nn\\
    &= \frac{4L_\rho D_\cX^2}{K^2-1}+\frac{2GD_\cX(K+2)}{K+1}\, \leq \, \frac{4L_\rho D_\cX^2}{K^2-1}+\frac{8GD_\cX}{3},\label{t2_1}
\end{align}
where  the last inequality is due to $K\geq 2$. By \eqref{feasibility}, \eqref{t2_1}, {and the fact that $\sqrt{a+b}\leq\sqrt{a}+\sqrt{b}$ for all $a,b \geq 0$,} one has that
\begin{align*}
   & \|[c(x_K)]_+\|
     \overset{\eqref{feasibility}}{\leq} \frac{2\Lambda}{\rho}+\sqrt{\frac{2(F_{\rho}(x_K)-F_{\rho}^*)}{\rho}} \overset{\eqref{t2_1}}{\leq} \frac{2\Lambda}{\rho}+\sqrt{\frac{2}{\rho}\left(\frac{4L_\rho D_\cX^2}{K^2-1}+\frac{8GD_\cX}{3}\right)} \\
&\leq \frac{2\Lambda}{\rho}+\sqrt{\frac{8L_\rho D_\cX^2}{\rho(K^2-1)}}+\sqrt{\frac{16GD_\cX}{3\rho}}
  =\frac{2\Lambda}{K^{3/2}}+\sqrt{\frac{8(L_{\nabla f}+L_{\nabla c^2}K^{3/2}) D_\cX^2}{K^{3/2}(K^2-1)}}+\sqrt{\frac{16GD_\cX}{3K^{3/2}}} \\
& \leq\frac{2\Lambda}{K^{3/2}}+\sqrt{\frac{8L_{\nabla f} D_\cX^2}{K^{3/2}(K^2-1)}}+\sqrt{\frac{8L_{\nabla c^2} D_\cX^2}{K^2-1}}+\sqrt{\frac{16GD_\cX}{3K^{3/2}}}, 
\end{align*}
which, along with $K^2-1 \geq 3K^2/4$ due to $K\geq 2$, implies that \eqref{thm2:ineq1} holds.

In addition, using $\rho=K^{3/2}$, $L_\rho=L_{\nabla f}+\rho L_{\nabla c^2}$, $K\geq 2$, and Lemma \ref{l-sum-fixed},  we have that
\begin{align}
    \bbE[F_{\rho}(x_K)-F_{\rho}^*] &\overset{{\rm Lemma}\, \ref{l-sum-fixed}}{\leq} \frac{1}{(\beta_K-1)\gamma_K}\Big(\frac{1}{2}D_\cX^2+\sigma^2\sum_{i=1}^{K-1}\gamma_i^2\Big) \nn \\
    &=\frac{4L_\rho}{((K+1)/2-1) (K+1)}\Big(\frac{1}{2}D_\cX^2+\sigma^2\sum_{i=1}^{K-1}\left(\frac{i+1}{4L_\rho}\right)^2\Big) \nn \\ 
& { \leq \frac{4L_\rho D_\cX^2}{K^2-1}+\frac{\sigma^2(K+1)^3}{6L_\rho (K^2-1)}} \nn \\ 
& =\frac{4L_{\nabla f} D_\cX^2}{K^2-1}+\frac{4L_{\nabla c^2} D_\cX^2K^{3/2}}{K^2-1}+\frac{\sigma^2(K+1)^2}{6(L_{\nabla f}+L_{\nabla c^2}K^{3/2})(K-1)}\nn \\
& \leq \frac{16L_{\nabla f} D_\cX^2}{3K^2}+\frac{16L_{\nabla c^2} D_\cX^2}{3K^{1/2}}+\frac{9\sigma^2}{4L_{\nabla c^2}K^{1/2}} \nn \\ 
&{< \frac{6L_{\nabla f} D_\cX^2}{K^2}+\frac{6L_{\nabla c^2} D_\cX^2+3\sigma^2L_{\nabla c^2}^{-1}}{K^{1/2}},} \label{t2_2}
\end{align}
 where the second inequality follows from the fact that $\sum_{i=1}^{K-1}(i+1)^2\leq\int_{1}^K(\tau+1)^2d\tau\leq(K+1)^3/3$,  and the third inequality follows from the fact that $K^2-1 \geq 3K^2/4$, $K+1 \leq 3K/2$, and $K-1\geq K/2$ due to $K\geq 2$. The relation \eqref{thm2:ineq2} then follows from  \eqref{t2_2} and $\bbE[F(x_K)-F^*]\leq\bbE[F_{\rho}(x_K)-F_{\rho}^*]$ due to \eqref{opt-gap}.

Lastly, {using \eqref{feasibility},  \eqref{t2_2}, and $\rho=K^{3/2}$,  we obtain that}
\begin{align}
    \bbE[\|[c(x_K)]_+\|] & \overset{\eqref{feasibility}}
    {\leq} \frac{2\Lambda}{\rho}+\bbE\left[\sqrt{\frac{2(F_{\rho}(x_K)-F_{\rho}^*)}{\rho}}\,\right] \leq\frac{2\Lambda}{\rho}+\sqrt{\frac{2\bbE\left[F_{\rho}(x_K)-F_{\rho}^*\right]}{\rho}} \nn \\
   & \leq \frac{2\Lambda}{K^{3/2}}+\sqrt{\frac{12L_{\nabla f} D_\cX^2}{K^{7/2}}}+\sqrt{12L_{\nabla c^2} D_\cX^2+6\sigma^2L_{\nabla c^2}^{-1}}\  \frac{1}{K}, \label{expect-feas}
\end{align}
{where the second inequality uses Jensen’s inequality, and the last inequality follows from   \eqref{t2_2},  $\rho=K^{3/2}$, and the fact that $\sqrt{a+b}\leq\sqrt{a}+\sqrt{b}$ for all $a,b \geq 0$. The relation \eqref{thm2:ineq3} then follows from  \eqref{opt-gap} and \eqref{expect-feas}.}
\end{proof}

In the remainder of this subsection, we present the proof of
Theorem~\ref{thm:composite}. {To this end, in view of Lemma~\ref{l-iter},
we first derive upper bounds on $F_{\rho_k}(x_k) - F_{\rho_k}^*$ and its
expectation in the setting where Algorithm~\ref{alg1} employs dynamic
penalty parameters $\{\rho_k\}$. These bounds, together with
Lemma~\ref{l-penalty}, will be used to prove
Theorem~\ref{thm:composite}.}

\begin{lemma}\label{l-sum}
Suppose that Assumptions \ref{a1} and \ref{a2} hold. Let $\{x_k\}$ be generated in Algorithm \ref{alg1} with a non-decreasing positive $\{\rho_k\}$, $\beta_k\in [1,\infty)$ and $\gamma_k>0$ satisfying $\beta_1=1$, $0<(\beta_{k+1}-1)\gamma_{k+1}\leq \beta_k\gamma_k(1-2(\rho_{k+1}-\rho_k)/\rho_{k+1})$, and $2L_{\rho_k}\gamma_k\leq\beta_k$ for all $k\geq 1$,  where $L_{\rho_k}$ is defined in \eqref{L-rho}. Then for all $k\geq 2$, we have
\begin{align*}
 & F_{\rho_k}(x_k)-F^*_{\rho_k} \leq\frac{1}{(\beta_k-1)\gamma_k}\Big(\frac{1}{2}D_\cX^2+2GD_\cX\sum_{i=1}^{k-1}\gamma_i+\frac{\Lambda^2}{2}\sum_{i=1}^{k-1}\frac{\gamma_i}{\rho_i}+4\Lambda^2\sum_{i=1}^{k-1}\frac{\beta_i\gamma_i(\rho_{i+1}-\rho_i)}{\rho_{i+1}^2}\Big),  \\
 &  \bbE\big[F_{\rho_k}(x_k)-F^*_{\rho_k}\big] \leq \frac{1}{(\beta_k-1)\gamma_k}\Big(\frac{1}{2}D_\cX^2+\sigma^2\sum_{i=1}^{k-1}\gamma_i^2+\frac{\Lambda^2}{2}\sum_{i=1}^{k-1}\frac{\gamma_i}{\rho_i}+4\Lambda^2\sum_{i=1}^{k-1}\frac{\beta_i\gamma_i(\rho_{i+1}-\rho_i)}{\rho_{i+1}^2}\Big),
\end{align*}
where $F_{\rho_k}$ and $F^*_{\rho_k}$ are defined in \eqref{penalty-prob},  and $D_\cX$, $\Lambda$, $\sigma$ and $G$ are given in Assumptions \ref{a1} and \ref{a2}, respectively. 
\end{lemma}

\begin{proof}
Since $2L_{\rho_k}\gamma_k\leq\beta_k$ and $\beta_k\geq 1$ for all $k\geq 1$, the assumption of Lemma \ref{l-iter} holds.  Let $x^*$ and $x^*_k$ be an arbitrary optimal solution of problems \eqref{prob1} and \eqref{def-Frho1}, respectively. It then follows that $F^*_{\rho_k}=F_{\rho_k}(x_k^*)$, $F_{\rho_k}(x_k) \geq F^*_{\rho_k}$, $F_{\rho_{k+1}}^* \geq F^*_{\rho_k}$, and $F_{\rho_k}(x^*)=F^*$ for all $k\geq 1$, where $F^*$ is defined in \eqref{prob1}.  By these, \eqref{opt-gap}, \eqref{feasibility}, and \eqref{determ-ineq} with $x=x^*$, one has that for all $k\geq 1$, 
\begin{align}
&\frac{1}{2}\|x^*-z_k\|^2-\frac{1}{2}\|x^*-z_{k+1}\|^2+\Delta_k(x) \overset{\eqref{determ-ineq}}{\geq}\beta_k\gamma_k\big(F_{\rho_k}(x_{k+1})-F^*\big)-(\beta_k-1)\gamma_k\big(F_{\rho_k}(x_k)-F^*\big) \nn\\
&=\beta_k\gamma_k\big(F_{\rho_k}(x_{k+1})-F_{\rho_{k+1}}^*\big)-(\beta_k-1)\gamma_k\big(F_{\rho_k}(x_k)-F_{\rho_k}^*\big)+\beta_k\gamma_k\big(F_{\rho_{k+1}}^*-F^*\big)- (\beta_k-1)\gamma_k\big(F_{\rho_k}^*-F^*\big)\nn\\
&\geq\beta_k\gamma_k\big(F_{\rho_k}(x_{k+1})-F_{\rho_{k+1}}^*\big)-(\beta_k-1)\gamma_k\big(F_{\rho_k}(x_k)-F_{\rho_k}^*\big)+\gamma_k\big(F_{\rho_k}^*-F^*\big)\nn\\
&=\beta_k\gamma_k\big(F_{\rho_k}(x_{k+1})-F_{\rho_{k+1}}^*\big)-(\beta_k-1)\gamma_k\big(F_{\rho_k}(x_k)-F_{\rho_k}^*\big)+\gamma_k\big(F(x_k^*)-F^*\big)+\frac{\gamma_k\rho_k}{2}||[c(x_k^*)]_+\|^2\nn\\
&\overset{\eqref{opt-gap}}{\geq}\beta_k\gamma_k\big(F_{\rho_k}(x_{k+1})-F_{\rho_{k+1}}^*\big)-(\beta_k-1)\gamma_k\big(F_{\rho_k}(x_k)-F_{\rho_k}^*\big)- \gamma_k\Lambda\|[c(x_k^*)]_+\|+\frac{\gamma_k\rho_k}{2}\|[c(x_k^*)]_+\|^2\nn\\
&\geq\beta_k\gamma_k\big(F_{\rho_k}(x_{k+1})-F_{\rho_{k+1}}^*\big)-(\beta_k-1)\gamma_k\big(F_{\rho_k}(x_k)-F_{\rho_k}^*\big)- \frac{\gamma_k\Lambda^2}{2\rho_k}\nn\\
&=\beta_k\gamma_k\big(F_{\rho_{k+1}}(x_{k+1})-F_{\rho_{k+1}}^*\big)-(\beta_k-1)\gamma_k\big(F_{\rho_k}(x_k)-F_{\rho_k}^*\big)- \frac{\gamma_k\Lambda^2}{2\rho_k}-\frac{1}{2}\beta_k\gamma_k(\rho_{k+1}-\rho_k)\|[c(x_{k+1})]_+\|^2\nn\\
&\overset{\eqref{feasibility}}{\geq}\beta_k\gamma_k\big(F_{\rho_{k+1}}(x_{k+1})-F_{\rho_{k+1}}^*\big)-(\beta_k-1)\gamma_k\big(F_{\rho_k}(x_k)-F_{\rho_k}^*\big)- \frac{\gamma_k\Lambda^2}{2\rho_k}\nn\\
&\quad\,\,-\frac{1}{2}\beta_k\gamma_k(\rho_{k+1}-\rho_k)\left(\frac{2\Lambda}{\rho_{k+1}}+\sqrt{\frac{2(F_{\rho_{k+1}}(x_{k+1})-F_{\rho_{k+1}}^*)}{\rho_{k+1}}}\,\right)^2\nn\\
&\geq \beta_k\gamma_k\Big(1-\frac{2(\rho_{k+1}-\rho_k)}{\rho_{k+1}}\Big)\big(F_{\rho_{k+1}}(x_{k+1})-F_{\rho_{k+1}}^*\big)-(\beta_k-1)\gamma_k\big(F_{\rho_k}(x_k)-F_{\rho_k}^*\big)- \frac{\gamma_k\Lambda^2}{2\rho_k}\nn\\
&\quad-\frac{4\beta_k\gamma_k(\rho_{k+1}-\rho_k)\Lambda^2}{\rho_{k+1}^2}, \label{l3-0}
\end{align}
where the second inequality is due to $F_{\rho_{k+1}}^* \geq F^*_{\rho_k}$, the second equality follows from \eqref{def-Frho1} and the definition of $F$ in \eqref{prob1}, and the fourth inequality is due to Young's inequality, while the last inequality follows from the fact that $(a+b)^2 \leq 2(a^2+b^2)$ for all $a,b\in\rr$.

Recall that $0<(\beta_{k+1}-1)\gamma_{k+1}\leq\beta_k\gamma_k\big(1-2(\rho_{k+1}-\rho_k)/\rho_{k+1}\big)$ and $L_{\rho_k}\gamma_k-\beta_k<0$ for all $k\geq 1$. These, together with \eqref{delta}, \eqref{l3-0}, and $F_{\rho_{i}}(x_{i}) \geq F^*_{\rho_{i}}$ for all $i$, imply that for all $i\geq1$, 
\begin{align}
    &(\beta_{i+1}-1)\gamma_{i+1}(F_{\rho_{i+1}}(x_{i+1})-F^*_{\rho_{i+1}})+\frac{1}{2}\|x^*-z_{i+1}\|^2 \nn\\
    &\leq \beta_i\gamma_i\Big(1-\frac{2(\rho_{i+1}-\rho_i)}{\rho_{i+1}}\Big)(F_{\rho_{i+1}}(x_{i+1})-F^*_{\rho_{i+1}})+\frac{1}{2}\|x^*-z_{i+1}\|^2\nn\\
    &\overset{\eqref{delta}\eqref{l3-0}}{\leq} (\beta_i-1)\gamma_i(F_{\rho_i}(x_i)-F^*_{\rho_i})+\frac{1}{2}\|x^*-z_i\|^2+\gamma_i\langle \delta_i, x^*-z_{i+1}\rangle-\frac{\beta_i-L_{\rho_i}\gamma_i}{2\beta_i}\|z_{i+1}-z_i\|^2\nn\\
    &\qquad\quad+\frac{\gamma_i\Lambda^2}{2\rho_i}+\frac{4\beta_i\gamma_i(\rho_{i+1}-\rho_i)\Lambda^2}{\rho_{i+1}^2}\nn\\
    &\leq (\beta_i-1)\gamma_i(F_{\rho_i}(x_i)-F^*_{\rho_i})+\frac{1}{2}\|x^*-z_i\|^2 + \gamma_i\|\delta_i\|\|x^*-z_{i+1}\|+\frac{\gamma_i\Lambda^2}{2\rho_i}+\frac{4\beta_i\gamma_i(\rho_{i+1}-\rho_i)\Lambda^2}{\rho_{i+1}^2}.\label{l3-1}
\end{align}

Using \eqref{deltabound1}, $\beta_1=1$, $\|x^*-z_1\|\leq D_\cX$, and summing up \eqref{l3-1} from $i=1$ to $k-1$, we obtain that
\begin{align*}
    &(\beta_k-1)\gamma_k(F_{\rho_k}(x_k)-F^*_{\rho_k})\\
    &\leq (\beta_1-1)\gamma_1(F_{\rho_1}(x_1)-F^*_{\rho_k})+\frac{1}{2}\|x^*-z_1\|^2+\sum_{i=1}^{k-1}\gamma_i\|\delta_i\|\|x^*-z_{i+1}\|+\frac{\Lambda^2}{2}\sum_{i=1}^{k-1}\frac{\gamma_i}{\rho_i}+4\Lambda^2\sum_{i=1}^{k-1}\frac{\beta_i\gamma_i(\rho_{i+1}-\rho_i)}{\rho_{i+1}^2} \\
    &\leq \frac{D_\cX^2}{2}+2GD_\cX\sum_{i=1}^{k-1}\gamma_i+\frac{\Lambda^2}{2}\sum_{i=1}^{k-1}\frac{\gamma_i}{\rho_i}+4\Lambda^2\sum_{i=1}^{k-1}\frac{\beta_i\gamma_i(\rho_{i+1}-\rho_i)}{\rho_{i+1}^2},
\end{align*}
which, along with the assumption that $(\beta_k-1)\gamma_k>0$ for all $k>1$, implies that the first conclusion of this lemma holds.

We next prove that the second conclusion of this lemma holds. Using \eqref{stoch-ineq} and applying similar arguments as those used to derive \eqref{l3-0}, we can show that for all $k\geq 1$, 
\begin{align}
&\frac{1}{2}\bbE[\|x^*-z_k\|^2]-\frac{1}{2}\bbE[\|x^*-z_{k+1}\|^2]+\frac{\beta_k\gamma_k^2}{2(\beta_k-L_{\rho_k}\gamma_k)}\bbE[\|\delta_k\|^2]\nn \\ 
&\geq \beta_k\gamma_k\Big(1-\frac{2(\rho_{k+1}-\rho_k)}{\rho_{k+1}}\Big)\bbE[F_{\rho_{k+1}}(x_{k+1})-F_{\rho_{k+1}}^*]
-(\beta_k-1)\gamma_k\bbE[F_{\rho_k}(x_k)-F_{\rho_k}^*]- \frac{\gamma_k\Lambda^2}{2\rho_k}\nn \\
&\quad-\frac{4\beta_k\gamma_k(\rho_{k+1}-\rho_k)\Lambda^2}{\rho_{k+1}^2}.\label{l3-2}
\end{align}
Also, by the assumption $2L_{\rho_k}\gamma_k\leq\beta_k$, one has $\beta_k/(\beta_k-L_{\rho_k}\gamma_k)\leq 2$. Using this {relation}, \eqref{l3-2}, and the assumption that $0<(\beta_{k+1}-1)\gamma_{k+1}\leq\beta_k\gamma_k(1-2(\rho_{k+1}-\rho_k)/\rho_{k+1})$ for all $k$,  we obtain that for all $i\geq 1$, 
\begin{align*}
&  (\beta_{i+1}-1)\gamma_{i+1}\bbE[F_{\rho_{i+1}}(x_{i+1})-F^*_{\rho_{i+1}}]+\frac{1}{2}\bbE[\|x^*-z_{i+1}\|^2] \\
&\leq \beta_i\gamma_i\Big(1-\frac{2(\rho_{i+1}-\rho_i)}{\rho_{i+1}}\Big) \bbE[F_{\rho_{i+1}}(x_{i+1})-F^*_{\rho_{i+1}}] +\frac{1}{2}\bbE[\|x^*-z_{i+1}\|^2] \\
    &\leq (\beta_i-1)\gamma_i\bbE[F_{\rho_i}(x_i)-F^*_{\rho_i}]
     +\frac{1}{2}\bbE[\|x^*-z_{i}\|^2]+\gamma_i^2\bbE[\|\delta_i\|^2]+\frac{\gamma_i\Lambda^2}{2\rho_i}+\frac{4\beta_i\gamma_i(\rho_{i+1}-\rho_i)\Lambda^2}{\rho_{i+1}^2}.
\end{align*}
Using $\beta_1=1$, $\bbE[\|\delta_i\|^2]\leq\sigma^2$, $\|x^*-z_1\|\leq D_\cX$, and summing up the above inequalities from $i=1$ to $k-1$, we have
\begin{equation*}
    (\beta_k-1)\gamma_k\bbE[F_{\rho_k}(x_k)-F^*_{\rho_k}]\leq\frac{1}{2}D_\cX^2+\sigma^2\sum_{i=1}^{k-1}\gamma_i^2+\frac{\Lambda^2}{2}\sum_{i=1}^{k-1}\frac{\gamma_i}{\rho_i}+4\Lambda^2\sum_{i=1}^{k-1}\frac{\beta_i\gamma_i(\rho_{i+1}-\rho_i)}{\rho_{i+1}^2},
\end{equation*}
which, along with the assumption that $(\beta_k-1)\gamma_k>0$ for all $k>1$,  implies that the second conclusion of this lemma holds. 
\end{proof}

We are now ready to prove Theorem \ref{thm:composite}.

\begin{proof}[\textbf{Proof of Theorem \ref{thm:composite}}]
Notice that $\rho_k=(k+4)^{3/2}$, $\beta_k=(k+4)/5$, $\gamma_k=(k+4)/(10L_{\rho_k})$, and $L_{\rho_k}=L_{\nabla f}+(k+4)^{3/2}L_{\nabla c^2}$. Also, one can verify that  $3(k+4)^2 \leq (k+5)\big((k+4)^2-k(k+5)\big)$, and  
\beq \label{aux-ineq1}
(k+5)^{3/2}-(k+4)^{3/2} \leq \frac{3}{2} (k+5)^{1/2} \ \ (\text{due to the convexity of}\  \tau^{3/2}). 
\eeq
Using these  {relations} and $L_{\rho_{k+1}} \geq L_{\rho_k}$, we have
\begin{align*}
\frac{2\beta_k\gamma_k(\rho_{k+1}-\rho_k)}{\rho_{k+1}}&=\frac{2(k+4)^2((k+5)^{3/2}-(k+4)^{3/2})}{50L_{\rho_k}(k+5)^{3/2}}
\leq\frac{3(k+4)^2}{50L_{\rho_k}(k+5)} \\
&{\leq\frac{(k+4)^2}{50L_{\rho_k}}-\frac{k(k+5)}{50L_{\rho_k}} \leq\frac{(k+4)^2}{50L_{\rho_k}}-\frac{k(k+5)}{50L_{\rho_{k+1}}}}\\
&{=\beta_k\gamma_k-(\beta_{k+1}-1)\gamma_{k+1}.}
\end{align*}
Hence, the assumption of Lemma \ref{l-sum} holds for such $\rho_k$, $\beta_k$, and $\gamma_k$.  In addition, one can verify that
\beq\label{aux-ineq2}
\sum_{i=1}^{k-1}(i+4)^{-1/2}\leq2(k+4)^{1/2}, \quad \sum_{i=1}^{k-1}(i+4)^{-2} \leq 1, \quad \sum_{i=1}^{k-1}(i+4)^{-1} \leq \log(k+4).
\eeq 
By \eqref{aux-ineq1}, \eqref{aux-ineq2},  Lemma \ref{l-sum}, and the definitions of $\tC_1$, $\tC_2$ in  \eqref{c12}, one has that for all $k\geq2$, 
\begin{align}
    &F_{\rho_k}(x_k)-F_{\rho_k}^*\leq\frac{1}{(\beta_k-1)\gamma_k}\Big(\frac{1}{2}D_\cX^2+\frac{\Lambda^2}{2}\sum_{i=1}^{k-1}\frac{\gamma_i}{\rho_i}+4\Lambda^2\sum_{i=1}^{k-1}\frac{\beta_i\gamma_i(\rho_{i+1}-\rho_i)}{\rho_{i+1}^2}+2GD_\cX\sum_{i=1}^{k-1}\gamma_i\Big) \nn\\
    &=\frac{50L_{\rho_k}}{(k-1) (k+4)}\Big(\frac{1}{2}D_\cX^2+\frac{\Lambda^2}{2}\sum_{i=1}^{k-1}\frac{i+4}{10(L_{\nabla f}+(i+4)^{3/2}L_{\nabla c^2})(i+4)^{3/2}} \nn \\
&\qquad\qquad\qquad\qquad+4\Lambda^2\sum_{i=1}^{k-1}\frac{(i+4)^2((i+5)^{3/2}-(i+4)^{3/2})}{50(L_{\nabla f}+(i+4)^{3/2}L_{\nabla c^2})(i+5)^{3}}+2GD_\cX\sum_{i=1}^{k-1}\frac{i+4}{10(L_{\nabla f}+(i+4)^{3/2}L_{\nabla c^2})}\Big)\nn\\
    &\overset{\eqref{aux-ineq1}}{\leq}\frac{50L_{\rho_k}}{(k-1) (k+4)}\Big(\frac{1}{2}D_\cX^2+\frac{\Lambda^2}{20L_{\nabla c^2}}\sum_{i=1}^{k-1}\frac{1}{(i+4)^{2}}+\frac{3\Lambda^2}{25L_{\nabla c^2}}\sum_{i=1}^{k-1}\frac{1}{(i+4)^{2}}+\frac{GD_\cX}{5L_{\nabla c^2}}\sum_{i=1}^{k-1}\frac{1}{(i+4)^{1/2}}\Big)\nn\\
   &\overset{\eqref{aux-ineq2}}{\leq}\frac{50L_{\rho_k}}{(k-1) (k+4)}\Big(\frac{1}{2}D_\cX^2+\frac{\Lambda^2}{20L_{\nabla c^2}}+\frac{3\Lambda^2}{25L_{\nabla c^2}}+\frac{2GD_\cX}{5L_{\nabla c^2}}(k+4)^{\frac{1}{2}}\Big)\nn\\
    & \overset{\eqref{c12}}{\leq}\frac{\tC_1L_{\rho_k}}{(k-1) (k+4)}+\frac{20GD_\cX\big(L_{\nabla f}+(k+4)^{3/2}L_{\nabla c^2}\big)(k+4)^{1/2}}{L_{\nabla c^2}(k-1)(k+4)}
   \  \leq \ \frac{\tC_1L_{\rho_k}}{(k-1) (k+4)}+\tC_2,\label{l4-1}
\end{align}
where the last inequality follows from $k\geq2$ and \eqref{c12}. Applying similar arguments as those used to derive \eqref{l4-1},  we can show that for all $k\geq2$, 
\begin{align}
    &\bbE[F_{{\rho_k}}(x_k)-F_{{\rho_k}}^*]\leq \frac{1}{(\beta_k-1)\gamma_k}\Big(\frac{1}{2}D_\cX^2+\sigma^2\sum_{i=1}^{k-1}\gamma_i^2+4\Lambda^2\sum_{i=1}^{k-1}\frac{\beta_i\gamma_i(\rho_{i+1}-\rho_i)}{\rho_{i+1}^2}+\frac{\Lambda^2}{2}\sum_{i=1}^{k-1}\frac{\gamma_i}{\rho_i}\Big)\nn\\
    &\leq \frac{\tC_1L_{\rho_k}}{(k-1) (k+4)}+\frac{\sigma^2L_{\rho_k}}{2(k-1) (k+4)}\sum_{i=1}^{k-1}\frac{(i+4)^2}{(L_{\nabla f}+(i+4)^{3/2}L_{\nabla c^2})^2}\nn\\
    &\leq\frac{\tC_1L_{\rho_k}}{(k-1) (k+4)}+\frac{\sigma^2L_{\rho_k}}{2L_{\nabla c^2}^2(k-1) (k+4)}\sum_{i=1}^{k-1}\frac{1}{i+4}
    \overset{\eqref{aux-ineq2}}{\leq}\frac{\tC_1L_{\rho_k}}{(k-1) (k+4)}+\frac{\sigma^2L_{\rho_k}\log(k+4)}{2L_{\nabla c^2}^2(k-1) (k+4)}.\label{l4-2}
\end{align}
Using  \eqref{feasibility}, \eqref{l4-1}, and the expressions of $\rho_k$ and $L_{\rho_k}$ in \eqref{def-para1} and \eqref{L-rho}, we have
\begin{align*}
    &\|[c(x_k)]_+\|
     \overset{\eqref{feasibility}}{\leq} \frac{2\Lambda}{\rho_k}+\sqrt{\frac{2(F_{{\rho_k}}(x_k)-F_{{\rho_k}}^*)}{{\rho_k}}} 
    \overset{\eqref{l4-1}}{\leq} \frac{2\Lambda}{{\rho_k}}+\sqrt{\frac{2}{{\rho_k}}\left(\frac{\tC_1L_{\rho_k}}{(k-1) (k+4)}+\tC_2\right)} \nn \\
    &\overset{\eqref{def-para1}\eqref{L-rho}}{=}\frac{2\Lambda}{(k+4)^{3/2}}+\sqrt{\frac{2\tC_1(L_{\nabla f}+(k+4)^{3/2}L_{\nabla c^2})}{(k-1)(k+4)^{5/2}}+\frac{2\tC_2}{(k+4)^{3/2}}} \\
    &\leq\frac{2\Lambda}{(k+4)^{3/2}}+\sqrt{\frac{2\tC_1L_{\nabla f}}{(k-1)(k+4)^{5/2}}}+\sqrt{\frac{2\tC_1L_{\nabla c^2}}{(k-1)(k+4)}}+\sqrt{\frac{2\tC_2}{(k+4)^{3/2}}},
\end{align*}
which,  along with $k-1 \geq k/2$ for all $k\geq 2$,  implies that \eqref{thm1-ineq1} holds. In addition, by \eqref{l4-2}, and the expressions of $\rho_k$ and $L_{\rho_k}$ in \eqref{def-para1} and \eqref{L-rho}, one has
\begin{align}
    &\bbE[F_{{\rho_k}}(x_k)-F_{{\rho_k}}^*]\overset{\eqref{l4-2}}{\leq} \frac{\tC_1L_{\rho_k}}{(k-1) (k+4)}+\frac{\sigma^2L_{\rho_k}\log(k+4)}{2L_{\nabla c^2}^2(k-1) (k+4)}\nn\\
    &\overset{\eqref{def-para1}\eqref{L-rho}}{=}\frac{\tC_1L_{\nabla f} }{(k-1) (k+4)}+\frac{\tC_1L_{\nabla c^2} (k+4)^{1/2}}{k-1}+\frac{\sigma^2L_{\nabla f}\log(k+4)}{2L_{\nabla c^2}^2(k-1) (k+4)}+\frac{\sigma^2L_{\nabla c^2}(k+4)^{1/2}\log(k+4)}{2L_{\nabla c^2}^2(k-1)}\label{l4-3} \\
& \leq 2\tC_1L_{\nabla f}k^{-2}+2\sqrt{3}\tC_1L_{\nabla c^2}k ^{-1/2}+3\sigma^2L_{\nabla f}L_{\nabla c^2}^{-2}k^{-2}\log k+3\sqrt{3}\sigma^2L_{\nabla c^2}^{-1}k^{-1/2}\log k, \label{l4-4}
\end{align}
where the last inequality is due to the fact that $k-1 \geq k/2$, $k+4\leq 3k$, and $\log(k+4)\leq 3\log k$ for all $k\geq2$. The relation \eqref{thm1-ineq2} then  follows from \eqref{l4-4} and $\bbE[F(x_k)-F^*]\leq\bbE[F_{\rho_k}(x_k)-F_{\rho_k}^*]$ due to \eqref{opt-gap}.

Lastly, it follows from \eqref{def-para1}, \eqref{feasibility}, and \eqref{l4-3} that
\begin{align*}
    &\bbE[\|[c(x_k)]_+\|]  \overset{\eqref{feasibility}}
    {\leq} \frac{2\Lambda}{\rho_k}+\bbE\left[\sqrt{\frac{{2(F_{\rho_k}(x_k)-F_{\rho_k}^*)}}{\rho_k}}\,\right] 
    \leq\frac{2\Lambda}{\rho_k}+\sqrt{\frac{2\bbE\left[F_{\rho_k}(x_k)-F_{\rho_k}^*\right]}{\rho_k}} \\
   &  \overset{\eqref{def-para1}\eqref{l4-4}}\leq \frac{2\Lambda}{(k+4)^{3/2}}+\sqrt{\frac{2\tC_1L_{\nabla f} }{(k-1) (k+4)^{5/2}}}+\sqrt{\frac{2\tC_1L_{\nabla c^2} }{(k-1)(k+4)}}+\sqrt{\frac{\sigma^2L_{\nabla f}\log(k+4)}{L_{\nabla c^2}^2(k-1) (k+4)^{5/2}}} \\
&\qquad\,\,\,+\sqrt{\frac{\sigma^2L_{\nabla c^2}\log(k+4)}{L_{\nabla c^2}^2(k-1)(k+4)}} \\
& \leq 2\Lambda k^{-3/2}+2(\tC_1L_{\nabla f})^{1/2}k^{-7/4}+2(\tC_1L_{\nabla c^2})^{1/2}k^{-1}+\sigma L_{\nabla c^2}^{-1} (6L_{\nabla f})^{1/2} k^{-7/4}(\log k)^{1/2} \nn \\
    &\quad\, +\sqrt{6}\sigma (L_{\nabla c^2})^{-1/2} k^{-1}(\log k)^{1/2},
\end{align*}
where the last inequality is due to the fact that $k-1 \geq k/2$ and $\log(k+4)\leq 3\log k$ for all $k\geq2$.
This together with \eqref{opt-gap} implies that  \eqref{thm1-ineq3} holds.
\end{proof}

\subsection{Proof for the main results in Section \ref{sec:finite-sum}}\label{sec:proof-finitesum}

In this subsection, we first establish several technical lemmas and then use them together with Lemma \ref{l-penalty} to prove Theorems  \ref{thm:finite-sum_fixed} and \ref{thm:finite-sum}.

Recall that $L_{\nabla c^2}$ and $L_{\nabla \barf}$ are defined in \eqref{L_nablac} and \eqref{Lam}, respectively. For notational convenience, we  define
\begin{equation}\label{notation}
\begin{aligned}
    &f_{\rho}(x) := \barf(x) + \frac{\rho}{2}\|[c(x)]_+\|^2 ,
    \qquad L_{\rho}:=L_{\nabla \barf}+\rho L_{\nabla c^2}, \\
    &\ell_{f_{\rho}}(u, v):=f_{\rho}(v)+\langle \nabla f_{\rho}(v), u-v\rangle
\end{aligned}
\end{equation}
for any $\rho>0$. Clearly, it follows from these  {expressions} and Assumption \ref{a3} that $\barf$ is $L_{\nabla \barf}$-smooth, and $f_{\rho}$ is $L_{\rho}$-smooth.  

The following lemma shows that $g_t$ is an unbiased stochastic estimator of $\nabla f_{\rho_k}(y_t)$ and provides an upper bound on its conditional variance. 

\begin{lemma}
Suppose that Assumptions \ref{a1} and \ref{a3} hold. Let $\{g_t\}$ and $\{y_t\}$ be generated in the $k$th outer iteration of Algorithm \ref{alg2} with $\{q_i\}$ given in \eqref{def-para3}.
Then for all $1\leq t\leq T_k$, we have
\begin{align}
    \bbE_{{i_t}}[\delta_t|\Xi_{k,t-1}]&=0, \label{mean}\\ 
    \bbE_{{i_t}}[\|\delta_t\|^2|\Xi_{k,t-1}]&\leq 2L_{\nabla \barf}(\barf(\tilde x_{k}) - \barf(y_t)-\langle\nabla \barf(y_t),\tilde x_{k} - y_t \rangle),\label{var}
\end{align}
where $\barf$ is defined in \eqref{fun-f}, $L_{\nabla \barf}$ is given in \eqref{Lam}, $\delta_t = g_t-\nabla f_{\rho_k}(y_t)$ with $f_{\rho}$ defined in \eqref{notation},  {$i_t$ is the random sample drawn at the $t$-th inner iteration of the $k$th outer iteration of Algorithm~\ref{alg2},} and $\Xi_{k,t-1}$ denotes the collection of all samples generated by Algorithm~\ref{alg2} prior to the $t$-th inner iteration of the $k$th outer iteration.
\end{lemma}

\begin{proof}
In view of \eqref{notation} and the expressions of $g_t$ and $\delta_t$, one can observe that 
\[
\delta_t=g_t-\nabla f_{\rho_k}(y_t)=(\nabla f_{i_t}(y_t)-\nabla f_{i_t}(\tilde x_{k}))/(q_{i_t}s)+\nabla \barf(\tilde x_{k})- \nabla \barf(y_t).
\]
The rest proof of this lemma follows from this relation and similar arguments used in the proof of \cite[Lemma 3]{lan2019unified}.
\end{proof}

The next lemma presents a well-known result for an inexact proximal gradient step applied to the problem $\min_x\{ f_{\rho_k}(x)+\psi(x)\}$, whose proof is similar to that of \cite[Lemma 5]{lan2019unified}.

\begin{lemma}
Suppose that Assumptions \ref{a1} and \ref{a3} hold.  Let $\{g_t\}$, $\{y_t\}$, $\{z_t\}$ be generated in the $k$th outer iteration of Algorithm \ref{alg2}, and $\delta_t = g_t-\nabla f_{{\rho_k}}(y_t)$, where $f_{{\rho_k}}$ is defined in \eqref{notation}. 
Then for any $x\in\cX$ and $1\leq t\leq T_k$, we have
\begin{equation}\label{l6-1}
    \gamma_k\left(\ell_{f_{\rho_k}}(z_t, y_t)-\ell_{f_{\rho_k}}(x,y_t )+\psi(z_t)-\psi(x)\right)\leq \frac{1}{2}\|z_{t-1}-x\|^2-\frac{1}{2}\|z_{t}-x\|^2-\frac{1}{2}\|z_{t-1}-z_t\|^2-\gamma_k\langle \delta_t, z_t-x\rangle,
\end{equation}
where $\ell_{f_{\rho_k}}$ is defined in \eqref{notation}.
\end{lemma}

{The following lemma establishes upper bounds on
\[
\frac{\gamma_k}{\alpha_k}(F_{{\rho_k}}(x_t)-F_{{\rho_k}}(x))+\frac{1}{2}\|z_t-x\|^2
\]
and its expectation for  all $k\geq 1$, $1\leq t\leq T_k$, and $x\in\cX$. These bounds will be used to derive upper bounds on $F_{\rho_k}(\tx_k) - F_{\rho_k}^*$ and its expectation.}

\begin{lemma}\label{l-iter2}
Suppose that Assumptions \ref{a1} and \ref{a3} hold. Let $\{x_t\}$, $\{y_t\}$, $\{z_t\}$,  $\{g_t\}$, and $\tilde x_{k}$ be generated in the $k$th outer iteration of Algorithm \ref{alg2} with $\{q_i\}$ given in \eqref{def-para3}, and $\alpha_k$, $\gamma_k$, $p_k$ and $\rho_k$ satisfying
\begin{align}
    &1-L_{\rho_k}\alpha_k\gamma_k>0, \label{l7-1}\\
    &p_k-\frac{L_{\nabla \barf}\alpha_k \gamma_k}{1-L_{\rho_k} \alpha_k\gamma_k}\geq 0 \label{l7-2}
\end{align}
for all $k\geq1$, and let $\delta_t = g_t-\nabla f_{\rho_k}(y_t)$, where $L_{\rho_k}$ and $f_{\rho_k}$ are given in \eqref{notation}. Then for all $k\geq 1$, $1\leq t\leq T_k$, and $x\in\cX$, we have
\begin{align}
    \frac{\gamma_k}{\alpha_k}(F_{{\rho_k}}(x_t)-F_{{\rho_k}}(x))+\frac{1}{2}\|z_t-x\|^2 &\leq \frac{\gamma_k}{\alpha_k}(1-\alpha_k-p_k)(F_{{\rho_k}}(x_{t-1})-F_{{\rho_k}}(x))+\frac{\gamma_kp_k}{\alpha_k} (F_{{\rho_k}}(\tilde x_{k})-F_{{\rho_k}}(x))\nn\\
    &\quad+\frac{1}{2}\|z_{t-1}-x\|^2 -\gamma_k\langle \delta_t, z_t-x\rangle, \label{l7-exact}\\
    \frac{\gamma_k}{\alpha_k}\bbE[F_{{\rho_k}}(x_t)-F_{{\rho_k}}(x)]+\frac{1}{2}\bbE[\|z_t-x\|^2] &\leq \frac{\gamma_k}{\alpha_k}(1-\alpha_k-p_k)\bbE[F_{{\rho_k}}(x_{t-1})-F_{{\rho_k}}(x)]+\frac{\gamma_kp_k}{\alpha_k} \bbE[F_{{\rho_k}}(\tilde x_{k})-F_{{\rho_k}}(x)]\nn\\
    &\quad+\frac{1}{2}\bbE[\|z_{t-1}-x\|^2] , \label{l7-exp}
\end{align}
where  $F_{\rho_k}$ is defined in \eqref{penalty-prob}.
\end{lemma}

\begin{proof}
Let $\{x_t\}$, $\{y_t\}$ and $\{z_t\}$ be generated in the $k$th outer iteration of Algorithm \ref{alg2}. Then it follows from the expressions of $x_t$ and $y_t$ in Algorithm \ref{alg2} that $ x_t-y_t=\alpha_k(z_t-z_{t-1})$.
By this {relation}, \eqref{l6-1}, the $L_{\rho_k}$-smoothness of $f_{{\rho_k}}$, and the definitions of $x_t$ and $\ell_{f_{\rho_k}}$, one has
\begin{align}
    &f_{\rho_k}(x_t)\leq \ell_{f_{\rho_k}}(x_t,y_t)+\frac{L_{\rho_k}}{2}\|x_t-y_t\|^2\nn\\
    &=(1-\alpha_k-p_k)\ell_{f_{{\rho_k}}}(x_{t-1},y_t)+\alpha_k \ell_{f_{{\rho_k}}}(z_t,y_t)+p_k\ell_{f_{{\rho_k}}}(\tilde x_{k},y_t)+\frac{L_{\rho_k}\alpha_k^2}{2}\|z_t-z_{t-1}\|^2\nn\\
    &\overset{\eqref{l6-1}}{\leq}(1-\alpha_k-p_k)\ell_{f_{{\rho_k}}}(x_{t-1},y_t) + \alpha_k\Big(\ell_{f_{{\rho_k}}}(x,y_t)+\psi(x)-\psi(z_t)+\frac{1}{2\gamma_k}\big(\|z_{t-1}-x\|^2-\|z_t-x\|^2\nn \\ 
&\qquad-\|z_{t-1}-z_t\|^2\big)-\langle \delta_t, z_t-x\rangle\Big)
+p_k\ell_{f_{{\rho_k}}}(\tilde x_{k},y_t)+\frac{L_{\rho_k}\alpha_k^2}{2}\|z_t-z_{t-1}\|^2\nn\\
    &= (1-\alpha_k-p_k)\ell_{f_{{\rho_k}}}(x_{t-1},y_t)
    +  \alpha_k\Big(\ell_{f_{{\rho_k}}}(x,y_t)+\psi(x)-\psi(z_t)+\frac{1}{2\gamma_k}\|z_{t-1}-x\|^2-\frac{1}{2\gamma_k}\|z_t-x\|^2\Big)\nn\\
    &\quad +p_k\ell_{f_{{\rho_k}}}(\tilde x_{k},y_t)-\frac{\alpha_k}{2\gamma_k}(1-L_{\rho_k}\alpha_k\gamma_k)\|z_{t-1}-z_t\|^2-\alpha_k\langle \delta_t, z_t-x\rangle \label{frho-bnd}  .
\end{align}
Using the convexity of $f_{{\rho_k}}$, we have $\ell_{f_{\rho_k}}(x,y_t)\leq f_{\rho_k}(x)$ for any $x\in\cX$. By this {relation}, \eqref{frho-bnd} and the definition of $F_{\rho_k}$ in \eqref{def-Frho2}, one has 
\begin{align*}
    f_{\rho_k}(x_t)&\leq (1-\alpha_k-p_k)f_{{\rho_k}}(x_{t-1})
    +\alpha_k\Big(F_{\rho_k}(x)-\psi(z_t)+\frac{1}{2\gamma_k}\|z_{t-1}-x\|^2-\frac{1}{2\gamma_k}\|z_t-x\|^2\Big)\nn\\
    &\quad + p_k\ell_{f_{\rho_k}}(\tilde x_{k},y_t)-\frac{\alpha_k}{2\gamma_k}(1-L_{\rho_k}\alpha_k\gamma_k)\|z_{t-1}-z_t\|^2-\alpha_k\langle \delta_t, z_t-x\rangle. 
\end{align*}
Also, by the expression of $x_t$ and the convexity of $\psi$, one has
\[
    \psi(x_t)\leq (1-\alpha_k-p_k)\psi(x_{t-1})+\alpha_k\psi(z_t)+p_k\psi(\tilde x_{k}).
\]
Summing up these two inequalities, and using the definition of $F_{\rho_k}$ in \eqref{def-Frho2}, we have
\begin{align}
    F_{\rho_k}(x_t)& \leq(1-\alpha_k-p_k)F_{\rho_k}(x_{t-1})+\alpha_k\Big(F_{\rho_k}(x)+\frac{1}{2\gamma_k}\|z_{t-1}-x\|^2-\frac{1}{2\gamma_k}\|z_t-x\|^2\Big)+p_k\ell_{f_{\rho_k}}(\tilde x_{k},y_t)\nn \\
    &\quad\,-\frac{\alpha_k}{2\gamma_k}(1-L_{\rho_k}\alpha_k\gamma_k)\|z_{t-1}-z_t\|^2-\alpha_k\langle \delta_t, z_t-x\rangle
+ p_k\psi(\tilde x_{k}). \label{l7-3}
\end{align}
Further, by \eqref{l7-1},  \eqref{l7-3},  $\ell_{f_{\rho_k}}(\tilde x_{k},y_t) \leq f_{\rho_k}(\tilde x_{k})$, and the definition of $F_{\rho_k}$ in \eqref{def-Frho2}, one has
\begin{align*}
    F_{\rho_k}(x_t)& \leq(1-\alpha_k-p_k)F_{\rho_k}(x_{t-1})+\alpha_k\Big(F_{\rho_k}(x)+\frac{1}{2\gamma_k}\|z_{t-1}-x\|^2-\frac{1}{2\gamma_k}\|z_t-x\|^2\Big)+p_kF_{\rho_k}(\tilde x_{k})\\
    &\quad\,-\alpha_k\langle\delta_t, z_t-x\rangle.
\end{align*}
The relation \eqref{l7-exact} then follows from this inequality by rearranging terms.

We next prove that \eqref{l7-exp} holds. Indeed, observe that
\begin{align}
&p_k\ell_{f_{\rho_k}}(\tilde x_{k},y_t)-\frac{\alpha_k}{2\gamma_k}(1-L_{\rho_k}\alpha_k\gamma_k)\|z_{t-1}-z_t\|^2-\alpha_k\langle \delta_t, z_t-x\rangle \nn \\
 &=p_k\ell_{f_{\rho_k}}(\tilde x_{k},y_t)-\frac{\alpha_k}{2\gamma_k}(1-L_{\rho_k}\alpha_k\gamma_k)\|z_{t-1}-z_t\|^2-\alpha_k\langle \delta_t, z_t-z_{t-1}\rangle-\alpha_k\langle \delta_t, z_{t-1}-x\rangle \nn\\
    &\leq  p_k\ell_{f_{\rho_k}}(\tilde x_{k},y_t)+\frac{\alpha_k\gamma_k\|\delta_t\|^2}{2(1-L_{\rho_k}\alpha_k\gamma_k)}
    -\alpha_k\langle \delta_t, z_{t-1}-x\rangle, \label{l7-4}
\end{align}
where the last relation follows from \eqref{l7-1} and the Young's inequality $a\|u\|^2+ b^2\|v\|^2/a \geq 2b\langle u,v \rangle$ for any $a>0$ and $b$. In addition, let $\ell_{\barf}(u,v)=\barf(v)+\langle \nabla \barf(v), u-v\rangle$. Notice from this {relation} and \eqref{notation} that
\[
\ell_{f_{\rho_k}}(u,v) = \ell_{\barf}(u,v) + \frac{{\rho_k}}{2}\big(
\|[c(v)]_+\|^2+\langle \nabla(\|[c(v)]_+\|^2), u-v \rangle\big) \qquad \forall u, v .
\]
Using this {relation}, \eqref{mean}, \eqref{var}, \eqref{l7-2}, and the convexity of $\barf$ and $\|[c(\cdot)]_+\|^2$, we obtain that
\begin{align}
&p_k\ell_{f_{\rho_k}}(\tilde x_{k},y_t)+\frac{\alpha_k\gamma_k\bbE_{{i_t}}[\|\delta_t\|^2|\Xi_{k,t-1}]}{2(1-L_{\rho_k}\alpha_k\gamma_k)}-\alpha_k\bbE_{{i_t}}[\langle \delta_t, z_{t-1}-x\rangle|\Xi_{k,t-1}] \nn \\
&\overset{\eqref{mean}\eqref{var}}{\leq} p_k\ell_{f_{\rho_k}}(\tilde x_{k},y_t)+\frac{L_{\nabla \barf}\alpha_k\gamma_k}{1-L_{\rho_k}\alpha_k\gamma_k}\left(\barf(\tilde x_{k}) - \ell_{\barf}(\tilde x_{k},y_t)\right)\nn \\
&=p_k\ell_{\barf}(\tilde x_{k},y_t)+\frac{L_{\nabla \barf}\alpha_k\gamma_k}{1-L_{\rho_k}\alpha_k\gamma_k}\left(\barf(\tilde x_{k}) - \ell_{\barf}(\tilde x_{k},y_t)\right)+\frac{p_k{\rho_k}}{2}\left(\|[c(y_t)]_+\|^2+\langle \nabla (\|[c(y_t)]_+\|^2), \tilde x_{k} - y_t \rangle\right) \nn \\
&=\Big(p_k-\frac{L_{\nabla \barf}\alpha_k\gamma_k}{1-L_{\rho_k}\alpha_k\gamma_k}\Big)\ell_\barf(\tilde x_{k},y_t)+\frac{L_{\nabla \barf}\alpha_k\gamma_k}{1-L_{\rho_k}\alpha_k\gamma_k}\barf(\tilde x_{k})+\frac{p_k{\rho_k}}{2}\left(\|[c(y_t)]_+\|^2+\langle \nabla (\|[c(y_t)]_+\|^2), \tilde x_{k} - y_t \rangle\right) \nn \\
&\leq p_k\barf(\tilde x_{k}) + \frac{p_k{\rho_k}}{2}\|[c(\tilde x_{k})]_+\|^2
= p_kf_{\rho_k}(\tilde x_{k}), \label{l7-4b}
\end{align}
where the last inequality follows from \eqref{l7-2} and the convexity of $\barf$ and $\|[c(\cdot)]_+\|^2$. {Taking expectation with respect to $\Xi_{k,t-1}$ on both sides of \eqref{l7-4b} yields
\beq \label{l7-5}
\bbE\Big[p_k\ell_{f_{\rho_k}}(\tilde x_{k},y_t)+\frac{\alpha_k\gamma_k\|\delta_t\|^2}{2(1-L_{\rho_k}\alpha_k\gamma_k)} -\alpha_k\langle \delta_t, z_{t-1}-x\rangle\Big] \leq p_k\bbE[f_{\rho_k}(\tilde x_{k})].
\eeq
}
It then follows from \eqref{l7-4} and \eqref{l7-5} that
\begin{align*}
&\bbE\Big[p_k\ell_{f_{\rho_k}}(\tilde x_{k},y_t)-\frac{\alpha_k}{2\gamma_k}(1-L_{\rho_k}\alpha_k\gamma_k)\|z_{t-1}-z_t\|^2-\alpha_k\langle \delta_t, z_t-x\rangle\Big] \nn \\
    &\overset{\eqref{l7-4}}{\leq}  \bbE\Big[p_k\ell_{f_{\rho_k}}(\tilde x_{k},y_t)+\frac{\alpha_k\gamma_k\|\delta_t\|^2}{2(1-L_{\rho_k}\alpha_k\gamma_k)} -\alpha_k\langle \delta_t, z_{t-1}-x\rangle\Big] \overset{\eqref{l7-5}}{\leq} p_k\bbE[f_{\rho_k}(\tilde x_{k})].
\end{align*}
Using this inequality, the definition of $F_{\rho_k}$ in \eqref{def-Frho2}, and taking expectation on both sides of \eqref{l7-3} lead to
\[
 \bbE[F_{\rho_k}(x_t)] \leq(1-\alpha_k-p_k)\bbE[F_{\rho_k}(x_{t-1})]+\alpha_k\Big(F_{\rho_k}(x)+\frac{1}{2\gamma_k}\bbE[\|z_{t-1}-x\|^2]-\frac{1}{2\gamma_k}\bbE[\|z_t-x\|^2]\Big)+p_k\bbE[F_{\rho_k}(\tilde x_{k})].
\]
The relation \eqref{l7-exp} then follows from this inequality by rearranging the terms.
\end{proof}

In view of Lemma~\ref{l-iter2}, we next derive upper bounds on
$F_\rho(\tx_K) - F_\rho^*$ and its expectation under the setting where
Algorithm~\ref{alg2} uses a constant penalty parameter $\rho_k \equiv \rho$
for some $\rho > 0$. This result, together with Lemma~\ref{l-penalty},
will be used to prove Theorem~\ref{thm:finite-sum_fixed}.

\begin{lemma}\label{l-exactbound3}
Suppose that Assumptions \ref{a1} and \ref{a3} hold, and that $\{\tilde x_{k}\}$ is generated by Algorithm \ref{alg2} with $\{q_i\}$, $\{\theta_t\}$, $\{\gamma_k\}$, $\{p_k\}$, $\{\alpha_k\}$ and $\{T_k\}$ given in \eqref{def-para4}, and $\rho_k \equiv \rho$  for some $\rho>0$.  Let $k_0=\lfloor \log_2 s\rfloor+1$. Then for all $K> k_0$, we have
\begin{align}
    &F_{\rho}(\tilde x_K)-F_{\rho}^*\leq
      \frac{16(\tC_3+\rho \tC_4)}{s(K-k_0+3)^2}+16L_{\nabla \barf}D_\cX^2,\label{l9-bound}  \\  
    &\bbE[F_{\rho}(\tilde x_K)-F_{\rho}^*]\leq
      \frac{16(\tC_3+\rho \tC_4)}{s(K-k_0+3)^2}, \label{l9-expbound}
\end{align}
 where $D_\cX$ is given in Assumption \ref{a1}, $L_{\nabla \barf}$ is defined in \eqref{Lam}, and $\tC_3$, $\tC_4$, $F_\rho$, and $F_{\rho}^*$ are defined in \eqref{c1011} and \eqref{penalty-prob}, respectively. 
\end{lemma}

\begin{proof}
Let $\{g_t\}$ and $\{y_t\}$ be generated in the $k$th outer iteration of Algorithm \ref{alg2}, and $\delta_t = g_t-\nabla f_{{\rho_k}}(y_t)$. Notice from \eqref{def-para3} that $q_i=L_i/(\sum_{j=1}^sL_j)$. By these, \eqref{Lam}, the expression of $g_t$,  $L_i$-smoothness of $f_i$, $L_{\nabla \barf}$-smoothness of $f$, and Assumption \ref{a1}, one has
\begin{align*}
    \|\delta_t\| &=\Big\|\frac{\nabla f_{i_t}(y_t)-\nabla f_{i_t}(\tilde x_{k})}{q_{i_t}s}+\nabla f(\tilde x_{k})- \nabla f(y_t)\Big\| \leq\frac{L_{i_t}\|y_t-\tilde x_{k}\|}{q_{i_t}s}+L_{\nabla \barf}\|\tilde x_{k}-y_t\| \\
& \leq \left(\frac{L_{i_t}}{q_{i_t}s}+L_{\nabla \barf}\right)D_\cX\overset{\eqref{Lam}}{=}2L_{\nabla \barf}D_\cX.
\end{align*}
This inequality and Assumption \ref{a1} imply that
\begin{equation}\label{deltabound}
    \langle\delta_t, z_t-x\rangle\leq \|\delta_t\|\|z_t-x\|\leq2L_{\nabla \barf}D_\cX^2 \qquad \forall x\in \cX.
\end{equation}

Let $L_{\rho}$ be defined in  \eqref{notation}.  Notice from  \eqref{def-para4} and \eqref{notation} that $1/(3L_{\rho_k}\alpha_k)=\gamma_k \leq 1/(3L_{\nabla \barf}\alpha_k)$, which implies that 
$L_{\rho_k}\alpha_k\gamma_k=1/3$ and $L_{\nabla \barf}\alpha_k\gamma_k \leq 1/3$. By these, one can verify that  the assumption of Lemma \ref{l-iter2} holds for $\{q_i\}$, $\{\alpha_k\}$, $\{\gamma_k\}$ and $\{p_k\}$ given in \eqref{def-para4}. Let $x^*_\rho\in\argmin\limits_{x\in\cX} F_\rho(x)$. It then follows from \eqref{l7-exact} with $x=x^*_\rho$ and \eqref{deltabound} that 
\begin{align}
    \frac{\gamma_k}{\alpha_k}(F_{\rho}(x_t)-F_\rho^*)&\leq 
 \frac{\gamma_k}{\alpha_k}(1-\alpha_k-p_k)(F_{\rho}(x_{t-1})-F_\rho^*)+\frac{\gamma_kp_k}{\alpha_k} (F_{\rho}(\tilde x_{k})-F_\rho^*)\nn\\
 &\quad +\frac{1}{2}\big(\|z_{t-1}-x^*_\rho\|^2-\|z_t-x^*_\rho\|^2\big) +2\gamma_kL_{\nabla \barf}D_\cX^2.\label{l9-1}
\end{align}
Summing up \eqref{l9-1} from $t=1$ to $T_k$, and using the fact that $x_0=\tilde x_{k}$, $z_0=\tilde z_{k}$, $z_{T_k}=\tilde z_{k+1}$, we have
\begin{align}
\sum_{t=1}^{T_k}\frac{\gamma_k}{\alpha_k}(F_\rho(x_t)-F_\rho^*)
&\leq \frac{\gamma_k}{\alpha_k}(1-\alpha_k-p_k)(F_\rho(x_{0})-F_\rho^*)+ \sum_{t=2}^{T_k}\frac{\gamma_k}{\alpha_k}(1-\alpha_k-p_k)(F_\rho(x_{t-1})-F_\rho^*)\nn\\
&\quad+\frac{T_k\gamma_kp_k}{\alpha_k}(F_\rho(\tilde x_{k})-F_\rho^*)+\frac{1}{2}\big(\|z_{0}-x^*_\rho\|^2-\|z_{T_k}-x^*_\rho\|^2\big)+2\gamma_k T_k L_{\nabla \barf}D_\cX^2\nn\\
&= \frac{\gamma_k}{\alpha_k}(1-\alpha_k-p_k)(F_\rho(\tilde x_{k})-F_\rho^*)+ \sum_{t=1}^{T_k-1}\frac{\gamma_k}{\alpha_k}(1-\alpha_k-p_k)(F_\rho(x_{t})-F_\rho^*)\nn\\
&\quad+\frac{T_k\gamma_kp_k}{\alpha_k}(F_\rho(\tilde x_{k})-F_\rho^*)
+\frac{1}{2}\big(\|\tilde z_{k}-x^*_\rho\|^2-\|\tilde z_{k+1}-x^*_\rho\|^2\big)+2\gamma_k T_k L_{\nabla \barf}D_\cX^2.\label{l9-sum}
\end{align}

Let $\cL_k$ and $\mathcal{R}_k$ be defined as
\beq \label{Lk}
    \cL_k :=\frac{\gamma_k}{\alpha_k}+(T_k-1)\frac{\gamma_k(\alpha_k+p_k)}{\alpha_k}, \qquad
    \mathcal{R}_k :=\frac{\gamma_k}{\alpha_k}(1-\alpha_k)+(T_k-1)\frac{\gamma_kp_k}{\alpha_k}
\eeq
 with the associated $\alpha_k$, $\gamma_k$, $p_k$ and $T_k$ given in \eqref{def-para4}. Notice from \eqref{def-para4}, the definition of $\cL_k$, and 
Algorithm \ref{alg2} that $\cL_k=\sum_{t=1}^{T_k}\theta_t$ and $\tilde x_{k+1}=\sum_{t=1}^{T_k}(\theta_tx_t)/\sum_{t=1}^{T_k}\theta_t$. Using  these, \eqref{l9-sum}, and the convexity of $F_\rho$, we have
    \begin{align}
        & \cL_k(F_\rho(\tilde x_{k+1})-F_\rho^*)=\Big(\sum_{t=1}^{T_k}\theta_t\Big)(F_\rho(\tilde x_{k+1})-F_\rho^*)\leq\sum_{t=1}^{T_k}\theta_t(F_\rho(x_t)-F_\rho^*)\nn\\
        &\overset{\eqref{def-para4}}=\sum_{t=1}^{T_k}\frac{\gamma_k}{\alpha_k}(F_\rho(x_t)-F_\rho^*)-\sum_{t=1}^{T_k-1}\frac{\gamma_k}{\alpha_k}(1-\alpha_k-p_k)(F_\rho(x_t)-F_\rho^*)\nn \\
        &\overset{\eqref{l9-sum}}{\leq} \Big(\frac{\gamma_k}{\alpha_k}(1-\alpha_k)+(T_k-1)\frac{\gamma_kp_k}{\alpha_k}\Big)(F_\rho(\tilde x_{k})-F_\rho^*)
    +\frac{1}{2}\big(\|\tilde z_{k}-x^*_\rho\|^2-\|\tilde z_{k+1}-x^*_\rho\|^2\big)+2\gamma_k T_k L_{\nabla \barf}D_\cX^2 \nn \\
& =\mathcal{R}_k(F_\rho(\tilde x_{k})-F_\rho^*)
    +\frac{1}{2}\big(\|\tilde z_{k}-x^*_\rho\|^2-\|\tilde z_{k+1}-x^*_\rho\|^2\big)+2\gamma_kT_k L_{\nabla \barf}D_\cX^2. \label{l9-2}
    \end{align}

Next claim that $\cL_k-\mathcal{R}_{k+1}\geq 0$ for all $k\geq1$. Indeed, notice from \eqref{def-para4} that if $1\leq k< k_0$,  $\alpha_{k+1}=\alpha_k=1/2$, $p_k=1/2$, $\gamma_{k+1}=\gamma_k$, and $T_{k+1}=2T_k$. Using these {facts} and \eqref{Lk}, we have that for all $1\leq k< k_0$,
\begin{align*}
    &\cL_k-\mathcal{R}_{k+1}
    \overset{\eqref{Lk}}{=}\frac{\gamma_k}{\alpha_k}+(T_k-1)\frac{\gamma_k(\alpha_k+p_k)}{\alpha_k}-\frac{\gamma_{k+1}}{\alpha_{k+1}}(1-\alpha_{k+1})-(T_{k+1}-1)\frac{\gamma_{k+1}p_{k+1}}{\alpha_{k+1}}\\
    &= \frac{\gamma_k}{\alpha_k}\big(1+(T_k-1)(\alpha_k+p_k)-(1-\alpha_k)-(2T_k-1)p_k\big)=0,
\end{align*}
and hence the above claim holds for all $1\leq k< k_0$. On the other hand, if $k\geq k_0$, notice from \eqref{def-para4},  \eqref{notation},  $k_0 \geq 1$, and $\rho_k \equiv \rho$ that $T_k=T_{k+1}=T_{k_0}\geq1$, $\alpha_k=2/(k-k_0+4)$, $p_{k}=1/2$, and $\gamma_k=1/(3L_\rho\alpha_k)=(k-k_0+4)/(6L_\rho)$. Using these  {relations} and \eqref{Lk}, we obtain that for all $k\geq k_0$,
\begin{align*}
    &\cL_k-\mathcal{R}_{k+1}\overset{\eqref{Lk}}{=}\frac{\gamma_k}{\alpha_k}+(T_{k}-1)\frac{\gamma_k(\alpha_k+p_k)}{\alpha_k}-\frac{\gamma_{k+1}}{\alpha_{k+1}}(1-\alpha_{k+1})-(T_{k+1}-1)\frac{\gamma_{k+1}p_{k+1}}{\alpha_{k+1}}\\
    &=\frac{\gamma_k}{\alpha_k}+(T_{k_0}-1)\frac{\gamma_k(\alpha_k+p_k)}{\alpha_k}-\frac{\gamma_{k+1}}{\alpha_{k+1}}(1-\alpha_{k+1})-(T_{k_0}-1)\frac{\gamma_{k+1}p_{k+1}}{\alpha_{k+1}}\\
    &=\frac{\gamma_k}{\alpha_k}-\frac{\gamma_{k+1}}{\alpha_{k+1}}(1-\alpha_{k+1})+(T_{k_0}-1)\Big(\frac{\gamma_k(\alpha_k+p_k)}{\alpha_k}-\frac{\gamma_{k+1}p_{k+1}}{\alpha_{k+1}}\Big)\\
    &=\frac{(k-k_0+4)^2}{12L_\rho}-\frac{(k-k_0+5)^2}{12L_\rho}\Big(1-\frac{2}{k-k_0+5}\Big)\\
    &\quad+(T_{k_0}-1)\Big[\frac{(k-k_0+4)^2}{12L_\rho}\Big(\frac{2}{k-k_0+4}+\frac{1}{2}\Big)-\frac{(k-k_0+5)^2}{24L_\rho}\Big]\\
    &=\frac{(k-k_0+4)^2-(k-k_0+5)^2+2(k-k_0+5)}{12L_\rho} \\
    &\quad+\frac{(T_{k_0}-1)\big[4(k-k_0+4)+(k-k_0+4)^2-(k-k_0+5)^2\big]}{24L_\rho}\\
    &=\frac{1}{12L_\rho}+\frac{(T_{k_0}-1)(2(k-k_0)+7)}{24L_\rho}\geq 0,
\end{align*}
where the second equality is due to $T_k=T_{k+1}=T_{k_0}$, and the fourth equality follows from $\alpha_k=2/(k-k_0+4)$, $p_{k}=1/2$ and $\gamma_k=(k-k_0+4)/(6L_\rho)$. Hence, the above claim also holds for all $k \geq k_0$.

Summing up \eqref{l9-2} from $k=1$ to $K-1$, and using the fact that $F_\rho(x)-F_\rho^*\geq 0$ for any $x\in\cX$, $\cL_k-\mathcal{R}_{k+1}\geq 0$ for all $k\geq 1$, $\tilde x_1=\tilde z_1=x^0$, and $\mathcal{R}_1=2/(3L_\rho)$, we have
\begin{align}
\cL_{K-1}(F_\rho(\tilde x_K)-F_\rho^*)&
\leq \sum_{k=1}^{K-2}(\mathcal{R}_{k+1}-\cL_k)(F_\rho(\tilde x_{k})-F_\rho^*)+\mathcal{R}_1(F_\rho(\tilde x_1)-F_\rho^*)+\frac{1}{2}\|\tilde z_1-x^*_\rho\|^2 \nn \\ 
&\quad+2L_{\nabla \barf}D_\cX^2\sum_{k=1}^{K-1}\gamma_kT_k \nn\\
&\leq\frac{2}{3L_\rho}(F_\rho(x^0)-F_\rho^*)+\frac{1}{2}\|x^0-x^*_\rho\|^2+2L_{\nabla \barf}D_\cX^2\sum_{k=1}^{K-1}\gamma_kT_k.    \label{l9-4}
\end{align}
In addition, recall that $k_0=\lfloor \log_2 s\rfloor+1$ and $T_{k_0}=2^{k_0-1}$, which imply that  $s/2\leq T_{k_0} \leq s$. By this {relation}, $\rho_k\equiv \rho$,  \eqref{def-para4}, \eqref{notation}, and \eqref{Lk}, one has that for all $K>k_0$,
\begin{align*}
    \cL_{K-1}&\overset{\eqref{Lk}}{=}\frac{\gamma_{K-1}}{\alpha_{K-1}}\big(1+(T_{K-1}-1)(\alpha_{K-1}+p_{K-1})\big)\\
    &\overset{\eqref{def-para4}\eqref{notation}}{=}\frac{1}{3L_\rho\alpha_{K-1}^2}\left(1+(T_{K-1}-1)\Big(\alpha_{K-1}+\frac{1}{2}\Big)\right)
    =\frac{T_{K-1}-1}{3L_\rho\alpha_{K-1}}+\frac{T_{K-1}+1}{6L_\rho\alpha_{K-1}^2}\\
    &\overset{\eqref{def-para4}}{=}\frac{(T_{k_0}-1)(K-k_0+3)}{6L_\rho}+\frac{(T_{k_0}+1)(K-k_0+3)^2}{24L_\rho}\\
    &\geq\frac{(s-2)(K-k_0+3)}{12L_\rho}+\frac{(s+2)(K-k_0+3)^2}{48L_\rho} \\
    & =\frac{(s-2)(K-k_0+3)}{12L_\rho}+\frac{2(K-k_0+3)^2}{48L_\rho} + \frac{s(K-k_0+3)^2}{48L_\rho} \\
    &\geq \frac{(s-2)(K-k_0+3)}{12L_\rho}+\frac{8(K-k_0+3)}{48L_\rho} + \frac{s(K-k_0+3)^2}{48L_\rho} \\
    & = \frac{s(K-k_0+3)}{12L_\rho}+ \frac{s(K-k_0+3)^2}{48L_\rho} 
     \geq \frac{s(K-k_0+3)^2}{48L_\rho},\\
    \sum_{k=1}^{K-1}\gamma_kT_k&=\sum_{k=1}^{k_0-1}\gamma_kT_k+\sum_{k=k_0}^{K-1}\gamma_kT_k 
    \overset{\eqref{def-para4}\eqref{notation}}{=} \frac{2}{3L_\rho}\sum_{k=1}^{k_0-1}2^{k-1}+\frac{T_{k_0}}{3L_\rho}\sum_{k=k_0}^{K-1}\frac{k-k_0+4}{2}\\
    &\leq \frac{2^{k_0}}{3L_\rho}+\frac{T_{k_0}}{3L_\rho}\frac{(K-k_0+7)(K-k_0)}{4}
    = \frac{T_{k_0}\big((K-k_0)^2+7(K-k_0)+8\big)}{12L_\rho}\\
    &\leq \frac{T_{k_0}\big(2(K-k_0)^2+12(K-k_0)+18\big)}{12L_\rho}\leq\frac{s(K-k_0+3)^2}{6L_\rho}.
\end{align*}
 It then follows from the above two relations,  \eqref{penalty-prob},  \eqref{notation}, \eqref{l9-4},  and  Assumption \ref{a1}(i) that
\begin{align*}
F_{\rho}(\tilde x_K)-F_\rho^*&\leq \frac{16}{s(K-k_0+3)^2}\Big(2(F_\rho(x^0)-F_\rho^*)+\frac{3L_\rho}{2}\|x^0-x^*_\rho\|^2\Big)+16L_{\nabla \barf}D_\cX^2\nn\\
&\overset{ \eqref{penalty-prob}}{\leq} \frac{16}{s(K-k_0+3)^2}\Big(2(F(x^0)-F(x^*_{\rho})) +\rho\|[c(x^0)]_+\|^2 +\frac{3L_\rho}{2}\|x^0-x^*_\rho\|^2\Big)+16L_{\nabla \barf}D_\cX^2, \nn \\
&\leq \frac{16}{s(K-k_0+3)^2}\Big(2D_F+\rho\|[c(x^0)]_+\|^2 + \frac{3(L_{\nabla \barf}+\rho L_{\nabla c^2})  D_\cX^2}{2}\Big)+16L_{\nabla \barf}D_\cX^2,
\end{align*}
which together with \eqref{c1011} implies that \eqref{l9-bound} holds. In addition, the proof of \eqref{l9-expbound} follows from  \eqref{l7-exp} and similar arguments as above, and is thus omitted.
\end{proof}

We are now ready to prove Theorem \ref{thm:finite-sum_fixed}.

\begin{proof}[\textbf{Proof of Theorem \ref{thm:finite-sum_fixed}}]
We first prove statement (i) of Theorem \ref{thm:finite-sum_fixed}, where $\rho=s^{2/3}K^{4/3}$. Using \eqref{feasibility} and \eqref{l9-bound},  we have
\begin{align}
&\|[c(\tilde x_K)]_+\|\overset{\eqref{feasibility}}{\leq}  \frac{2\Lambda}{\rho}+\sqrt{\frac{2(F_{\rho}(\tilde x_K)-F_{\rho}^*)}{\rho}}
\overset{\eqref{l9-bound}}{\leq}\frac{2\Lambda}{\rho}+\sqrt{\frac{32(\tC_3+\rho \tC_4)}{\rho s(K-k_0+3)^2}+\frac{32L_{\nabla \barf}D_\cX^2}{\rho}}\nn\\
&\leq \frac{2\Lambda}{\rho}+\sqrt{\frac{32\tC_3}{\rho s(K-k_0+3)^2}}+\sqrt{\frac{32\tC_4}{s(K-k_0+3)^2}}+\sqrt{\frac{32L_{\nabla \barf}D_\cX^2}{\rho}}.\label{Thm3-detercons}
\end{align}
Also, observe from \eqref{opt-gap} that $\bbE [F(\tilde x_K)-F^*] \leq \bbE[F_\rho(\tilde x_K)-F_\rho^*]$. Using  this {inequality},  \eqref{l9-expbound}, \eqref{Thm3-detercons}, $\rho=s^{2/3}K^{4/3}$, and $K\geq2(k_0-3)$, we see that  \eqref{thm4-ineq1} and \eqref{thm4-ineq2} hold. In addition, by \eqref{feasibility} and \eqref{l9-expbound}, one has
\begin{align}
    \bbE[\|[c(\tilde x_K)]_+\|] & \overset{\eqref{feasibility}}
    {\leq} \frac{2\Lambda}{\rho}+\bbE\left[\sqrt{\frac{2(F_{\rho}(\tilde x_K)-F_{\rho}^*)}{\rho}}\,\right] \leq\frac{2\Lambda}{\rho}+\sqrt{\frac{2\bbE[F_{\rho}(\tilde x_K)-F_{\rho}^*]}{\rho}} \nn\\
&\overset{\eqref{l9-expbound}}{\leq} \frac{2\Lambda}{\rho} +\sqrt{\frac{32(\tC_3+\rho \tC_4)}{\rho s(K-k_0+3)^2}} \leq \frac{2\Lambda}{\rho} +\sqrt{\frac{32\tC_3}{\rho s(K-k_0+3)^2}}+\sqrt{\frac{32\tC_4}{s(K-k_0+3)^2}}.\label{Thm3-expcons}
\end{align}
Also, notice from  \eqref{opt-gap} that $\bbE [F(\tilde x_K)-F^*] \geq -\Lambda  \bbE[\|[c(\tilde x_K)]_+\|]$. By these, $\rho=s^{2/3}K^{4/3}$,  and $K\geq2(k_0-3)$, one can conclude that  \eqref{thm4-ineq1-exp} and \eqref{thm4-ineq3} hold.

The proof of statement (ii) of Theorem \ref{thm:finite-sum_fixed} follows from $\rho=\sqrt{s}K$ and similar arguments as above, and is thus omitted.
\end{proof}

In the remainder of this subsection, we present the proof of
Theorem~\ref{thm:finite-sum}. To this end, we first establish several technical lemmas below. 
The following lemma provides a relationship between $ \cL_{k}$ and $\mathcal{R}_{k+1}$, which together with Lemma~\ref{l-iter2} will be used to derive upper bounds on $F_{\rho_k}(\tx_k) - F_{\rho_k}^*$ and its
expectation in the setting where Algorithm~\ref{alg2} employs dynamic
penalty parameters $\{\rho_k\}$. 

\begin{lemma}\label{LkRk}
Let $\cL_k$ and $\mathcal{R}_k$ be defined in \eqref{Lk} with $\{\alpha_k\}$, $\{\gamma_k\}$, $\{p_k\}$  given in \eqref{def-para3}, and $k_0$, $\{\rho_k\}$, $\{T_k\}$ given in  \eqref{rho1} or  \eqref{rho2}. Then we have 
\begin{equation}\label{lkrk}
    \cL_{k}\Big(1-\frac{2(\rho_{k+1}-\rho_k)}{\rho_{k+1}}\Big)\geq \mathcal{R}_{k+1} \qquad \forall k \geq 1.
\end{equation}
\end{lemma}

\begin{proof}
We divide the proof into three separate cases: $1\leq k<k_0$, $k=k_0$, and $k>k_0$.

Case 1) $1\leq k<k_0$. Fix any $1\leq k<k_0$. Claim that $T_k \geq 1$ and $T_{k+1}\leq2T_k$ hold for $\{T_k\}$  given in  \eqref{rho1} or \eqref{rho2}. Indeed, if $\{T_k\}$ is given in  \eqref{rho1}, one has $T_{k}=\lceil2^{3(k-1)/4}\rceil$, which implies that $T_k \geq 1$ and 
\[
T_{k+1}=\lceil2^{3(k+1-1)/4}\rceil=\lceil2^{3/4}\times2^{3(k-1)/4}\rceil\leq\lceil2^{3/4}\rceil\times\lceil2^{3(k-1)/4}\rceil\leq2T_k.
\]
In addition, if $\{T_k\}$ is given in  \eqref{rho2}, one has $T_k=2^{k-1}$, yielding $T_k \geq 1$ and $T_{k+1}=2T_k$. Hence, the claim holds as desired. Observe from  \eqref{def-para3} that $\alpha_{k+1}=\alpha_k$ and $\rho_{k+1} \geq \rho_k$, which along with \eqref{notation} imply that $\gamma_{k+1} \leq \gamma_k$. Also, notice from  \eqref{def-para3} that $\alpha_{k+1}=\alpha_k=6/7$, $p_{k+1}=p_k=1/7$, $\rho_{k+1}=\sqrt{2}\rho_k$. Using these  {relations}, \eqref{Lk}, $T_k \geq 1$, $T_{k+1}\leq2T_k$, and $\gamma_{k+1} \leq \gamma_k$, we obtain that
\begin{align*}
    \cL_k-\mathcal{R}_{k+1}&\overset{\eqref{Lk}}{=}\frac{\gamma_k}{\alpha_k}+(T_k-1)\frac{\gamma_k(\alpha_k+p_k)}{\alpha_k}-\frac{\gamma_{k+1}}{\alpha_{k+1}}(1-\alpha_{k+1})-(T_{k+1}-1)\frac{\gamma_{k+1}p_{k+1}}{\alpha_{k+1}}\\
    &\geq\frac{\gamma_k}{\alpha_k}(1+(T_k-1)(\alpha_k+p_k))-\frac{\gamma_{k+1}}{\alpha_k}((1-\alpha_k)+(2T_k-1)p_k) \\
    &\geq\frac{\gamma_k}{\alpha_k}(1+(T_k-1)(\alpha_k+p_k)-(1-\alpha_k)-(2T_k-1)p_k) \\
&=\frac{5\gamma_kT_k}{7\alpha_k}=\frac{5\cL_k}{7}  \geq\frac{2(\rho_{k+1}-\rho_k)}{\rho_{k+1}}\cL_k,
\end{align*}
where the first inequality is  due to $\alpha_{k+1}=\alpha_k$ and $T_{k+1}\leq2T_k$, the second inequality follows from $\gamma_{k+1} \leq \gamma_k$, $0<\alpha_k <1$ and $T_k \geq 1$, the second and third equalities are due  to $\alpha_k=6/7$, $p_k=1/7$, and \eqref{Lk}, and  the last inequality follows from $\rho_{k+1}=\sqrt{2}\rho_k$. Hence, \eqref{lkrk} holds for all $1\leq k< k_0$. 

Case 2) $k=k_0$. Claim that $0<(\rho_{k_0+1}-\rho_{k_0})/\rho_{k_0+1}\leq1/3$ holds for $k_0$, $\rho_{k_0}$, and $\rho_{k_0+1}$ given in \eqref{rho1} or  \eqref{rho2}. Indeed, if they are given in \eqref{rho1}, one can observe that $k_0=\lfloor(4/3)\log_2 s\rfloor+1$, $\rho_{k_0}=2^{k_0/2}$, and $\rho_{k_0+1}=3s^{2/3}/2$. Also, notice that $s^{2/3} \leq 2^{(\lfloor(4/3)\log_2 s\rfloor+1)/2}\leq \sqrt{2} s^{2/3}$ and hence $s^{2/3} \leq \rho_{k_0} \leq \sqrt{2} s^{2/3}$.  This together with $\rho_{k_0+1}=3s^{2/3}/2$ implies that
\[
0<(\rho_{k_0+1}-\rho_{k_0})/\rho_{k_0+1}\leq(3s^{2/3}/2-s^{2/3})/(3s^{2/3}/2)=1/3.
\]
 In addition, if $k_0$, $\rho_{k_0}$, and $\rho_{k_0+1}$ are given in \eqref{rho2}, we observe that  $k_0=\lfloor \log_2 s\rfloor+1$, $\rho_{k_0}=2^{k_0/2}$, and $\rho_{k_0+1}=3\sqrt{s}/2$. Notice that 
$\sqrt{s} \leq 2^{(\lfloor\log_2 s\rfloor+1)/2}\leq \sqrt{2s}$ and hence  $\sqrt{s} \leq \rho_{k_0} \leq  \sqrt{2s}$. By this {relation} and $\rho_{k_0+1}=3\sqrt{s}/2$, one has
\[
0<(\rho_{k_0+1}-\rho_{k_0})/\rho_{k_0+1}\leq(3\sqrt{s}/2-\sqrt{s})/(3\sqrt{s}/2)=1/3.
\]
Hence, the above claim holds as desired. Notice from this claim, \eqref{def-para3}, and \eqref{notation} that  $L_{\rho_{k_0+1}}\geq L_{\rho_{k_0}}$, $\alpha_{k_0}=6/7$, and $\alpha_{k_0+1}=6/8$. By these and  \eqref{def-para3}, one has
\begin{align}
\gamma_{k_0+1}/\alpha_{k_0+1}&=1/(8L_{\rho_{k_0+1}}\alpha_{k_0+1}^2)\leq1/(8L_{\rho_{k_0}} 
\alpha_{k_0+1}^2) \nn \\ 
&{=64/(392L_{\rho_{k_0}}\alpha_{k_0}^2)=64\gamma_{k_0}/(49\alpha_{k_0}).} \label{alpha-gamma}
\end{align}
Notice from \eqref{def-para3}, \eqref{rho1}, \eqref{rho2}, and \eqref{Lk} that $\alpha_{k_0+1} \leq 1$,  $\alpha_{k_0}=6/7$, $p_{k_0+1}=p_{k_0}=1/7$, $T_{k_0+1}=T_{k_0} \geq 1$, and $\cL_{k_0}= \gamma_{k_0}T_{k_0}/\alpha_{k_0}$. Using these  {relations}, \eqref{alpha-gamma}, and $(\rho_{k_0+1}-\rho_{k_0})/\rho_{k_0+1}\leq1/3$, we obtain that
\begin{align*}
    \cL_{k_0}-\mathcal{R}_{k_0+1}
    &\overset{\eqref{Lk}}{=}\frac{\gamma_{k_0}}{\alpha_{k_0}}+(T_{k_0}-1)\frac{\gamma_{k_0}(\alpha_{k_0}+p_{k_0})}{\alpha_{k_0}}-\frac{\gamma_{k_0+1}}{\alpha_{k_0+1}}(1-\alpha_{k_0+1})-(T_{k_0+1}-1)\frac{\gamma_{k_0+1}p_{k_0+1}}{\alpha_{k_0+1}}\\
    &=\frac{\gamma_{k_0}}{\alpha_{k_0}}\big(1+(T_{k_0}-1)(\alpha_{k_0}+p_{k_0})\big)-\frac{\gamma_{k_0+1}}{\alpha_{k_0+1}}\big((1-\alpha_{k_0+1})+(T_{k_0+1}-1)p_{k_0+1}\big)\\
    &\overset{\eqref{alpha-gamma}}{\geq}\frac{\gamma_{k_0}}{\alpha_{k_0}}\Big(1+(T_{k_0}-1)(\alpha_{k_0}+p_{k_0})-\frac{64}{49}\big((1-\alpha_{k_0+1})+(T_{k_0+1}-1)p_{k_0+1}\big)\Big)\\
   &=\frac{\gamma_{k_0}}{\alpha_{k_0}}\Big(\Big(1-\frac{64}{343}\Big)T_{k_0}-\frac{16}{49}+\frac{64}{343}\Big)\\
    &\geq\frac{33\gamma_{k_0}T_{k_0}}{49\alpha_{k_0}}
    \geq \frac{2\gamma_{k_0}T_{k_0}}{3\alpha_{k_0}}
    \geq\frac{2(\rho_{k_0+1}-\rho_{k_0})}{\rho_{k_0+1}}\cL_{k_0},
\end{align*}
where the first inequality follows from \eqref{alpha-gamma}, $\alpha_{k_0+1} \leq 1$, and $T_{k_0+1} \geq 1$, the third equality is due to $\alpha_{k_0}=6/7$, $p_{k_0+1}=p_{k_0}=1/7$, and $T_{k_0+1}=T_{k_0}$, 
the second inequality follows from $T_{k_0}\geq1$ and $-16/49+64/343<0$, and the last inequality is due to 
$(\rho_{k_0+1}-\rho_{k_0})/\rho_{k_0+1}\leq1/3$ and $\cL_{k_0}= \gamma_{k_0}T_{k_0}/\alpha_{k_0}$. Hence, \eqref{lkrk} also holds for $k= k_0$. 

Case 3) $k>k_0$. Fix any $k>k_0$. Claim that 
\begin{align}
& \rho_{k+1} \geq \rho_k, \quad L_{\rho_{k+1}}   \leq 6L_{\rho_k}/5, \label{lrhoineq} \\
& \rho_{k+1}^{-1}(\rho_{k+1}-\rho_k) \leq 4(k-k_0+8)^{-1}/3 \label{rho-ineq}
\end{align}
hold for $\{\rho_k\}$ given in  \eqref{rho1} or \eqref{rho2}. Indeed, suppose that $\{\rho_k\}$ is given in  \eqref{rho1}.  {Then, for $k>k_0$},  one has $\rho_k=3s^{2/3}(k-k_0+7)^{4/3}/32$. Using this {expression}, $k>k_0$, and \eqref{notation}, we obtain that $\rho_{k+1} \geq \rho_k$, and 
\begin{align*}
L_{\rho_{k+1}}  &=L_{\nabla \barf}+\rho_{k+1} L_{\nabla c^2} \leq \rho_{k+1}{\rho_k}^{-1} L_{\nabla \barf}+\rho_{k+1} L_{\nabla c^2}  =\rho_{k+1}{\rho_k}^{-1} L_{\rho_k}  \\ 
& \leq \Big(\frac{k-k_0+8}{k-k_0+7}\Big)^{4/3}L_{\rho_k}\leq\Big(\frac{8}{7}\Big)^{4/3}L_{\rho_k}\leq\frac{6L_{\rho_k}}{5}, \\
\rho_{k+1}-\rho_k & =3s^{2/3}((k-k_0+8)^{4/3}-(k-k_0+7)^{4/3} )/32\leq s^{2/3}(k-k_0+8)^{1/3}/8 \\ 
&= 4\rho_{k+1}(k-k_0+8)^{-1}/3,
\end{align*}
where the last inequality is due to the convexity of the function $\tau^{4/3}$. Hence, \eqref{lrhoineq} and \eqref{rho-ineq} hold for $\{\rho_k\}$ given in  \eqref{rho1}. In addition, if $\{\rho_k\}$ is given in  \eqref{rho2}, one has $\rho_k=3\sqrt{s}(k-k_0+7)/16$ due to $k>k_0$. By this {expression} and similar arguments as above, one can show that \eqref{lrhoineq} and \eqref{rho-ineq} also hold for $\{\rho_k\}$ given in  \eqref{rho2}. Hence, the above claim holds as desired.

Further, by  \eqref{notation} and $\rho_{k+1} \geq \rho_k$, one can see that  $L_{\rho_k}\leq L_{\rho_{k+1}}$. In addition, it follows from \eqref{def-para3},   \eqref{rho1}, \eqref{rho2}, \eqref{notation}, and $k>k_0$ that $\alpha_k=6/(k-k_0+7)$, $\alpha_{k+1}=6/(k-k_0+8)$, $p_{k}=p_{k+1}=1/7$, $\gamma_k=1/(8L_{\rho_k}\alpha_k)$, $\gamma_{k+1}=1/(8L_{\rho_{k+1}}\alpha_{k+1})$, $T_k=T_{k+1}=T_{k_0}$,  $\rho_k=3s^{2/3}(k-k_0+7)^{4/3}/32$ or $\rho_k=3\sqrt{s}(k-k_0+7)/16$. By these relations, \eqref{Lk},  \eqref{lrhoineq}, and \eqref{rho-ineq},  one has that 
\begin{align*}
    &\cL_k-\mathcal{R}_{k+1}\overset{\eqref{Lk}}{=}\frac{\gamma_k}{\alpha_k}+(T_k-1)\gamma_k+(T_k-1)\frac{\gamma_kp_k}{\alpha_k}-\frac{\gamma_{k+1}}{\alpha_{k+1}}+\gamma_{k+1}-(T_{k+1}-1)\frac{\gamma_{k+1}p_{k+1}}{\alpha_{k+1}}\\
    &=\frac{1}{8L_{\rho_{k}}\alpha_k^2}+(T_k-1)\gamma_k+(T_k-1)\frac{p_k}{8L_{\rho_{k}}\alpha_k^2}-\frac{1}{8L_{\rho_{k+1}}\alpha_{k+1}^2}+\gamma_{k+1}-(T_k-1)\frac{p_k}{8L_{\rho_{k+1}}\alpha_{k+1}^2}\\
    &=(T_k-1)\Big(\gamma_k-\frac{p_k}{8L_{\rho_{k+1}}\alpha_{k+1}^2}+\frac{p_k}{8L_{\rho_{k}}\alpha_{k}^2}\Big)+\gamma_{k+1}-\frac{1}{8L_{\rho_{k+1}}\alpha_{k+1}^2}+\frac{1}{8L_{\rho_{k}}\alpha_{k}^2}\\
    &\geq (T_k-1)\Big(\gamma_k-\frac{p_k}{8L_{\rho_k}}\Big(\frac{1}{\alpha_{k+1}^2}-\frac{1}{\alpha_{k}^2}\Big)\Big)+\gamma_{k+1}-\frac{1}{8L_{\rho_{k+1}}}\Big(\frac{1}{\alpha_{k+1}^2}-\frac{1}{\alpha_{k}^2}\Big) \ (\text{due to} \ L_{\rho_k}\leq L_{\rho_{k+1}})\\
    &=(T_k-1)\Big(\gamma_k-\frac{p_k}{8L_{\rho_k}}\frac{(k-k_0+8)^2-(k-k_0+7)^2}{36}\Big)+\gamma_{k+1}-\frac{1}{8L_{\rho_{k+1}}}\frac{(k-k_0+8)^2-(k-k_0+7)^2}{36}  \\
    &=(T_k-1)\Big(\gamma_k-\frac{p_k}{8L_{\rho_k}}\frac{2(k-k_0)+15}{36}\Big)+\gamma_{k+1}-\frac{1}{8L_{\rho_{k+1}}}\frac{2(k-k_0)+15}{36}\\
    &\geq(T_k-1)\Big(\gamma_k-\frac{p_k}{8L_{\rho_k}}\frac{3(k-k_0+7)}{36}\Big)+\gamma_{k+1}-\frac{1}{8L_{\rho_{k+1}}}\frac{2(k-k_0+8)}{36} \\
&=(T_k-1)\Big(\gamma_k-\frac{p_k}{16L_{\rho_k}\alpha_k}\Big)+\gamma_{k+1}-\frac{1}{24L_{\rho_{k+1}}\alpha_{k+1}} \ (\text{due to} \ \alpha_k=6/(k-k_0+7)) \\
    &=(T_k-1)\Big(\gamma_k-\frac{p_k\gamma_k}{2}\Big)+\gamma_{k+1}-\frac{\gamma_{k+1}}{3}  \ (\text{due to} \ \gamma_k=1/(8L_{\rho_k}\alpha_k)) \\
    &=\frac{13}{14}(T_k-1)\gamma_k+\frac{2\gamma_{k+1}}{3}
    =\frac{13}{14}(T_k-1)\gamma_k+\frac{1}{12L_{\rho_{k+1}}\alpha_{k+1}} \ (\text{due to} \ p_k=1/7, \ \gamma_{k+1}=1/(8L_{\rho_{k+1}}\alpha_{k+1})) \\
    &\geq\frac{13}{14}(T_k-1)\gamma_k+\frac{1}{12L_{\rho_{k+1}}\alpha_{k}}
    \overset{\eqref{lrhoineq}}{\geq}\frac{13}{14}(T_k-1)\gamma_k+\frac{5}{72L_{\rho_{k}}\alpha_{k}}   \ (\text{due to} \ \alpha_{k+1} \leq \alpha_k)
    \\
    &=\frac{13}{14}(T_k-1)\gamma_k+\frac{5\gamma_{k}}{9}
    \geq \frac{8}{3(k-k_0+8)}\frac{k-k_0+7}{6}(T_k-1)\gamma_k+\frac{8}{3(k-k_0+8)}\frac{k-k_0+7}{6}\gamma_k\\
    &=\frac{8}{3(k-k_0+8)}\Big((T_k-1)\frac{\gamma_k}{\alpha_k}+\frac{\gamma_k}{\alpha_k}\Big) \ (\text{due to} \  \alpha_k=6/(k-k_0+7)) \\
   & \geq\frac{8}{3(k-k_0+8)}\Big((T_k-1)\frac{\gamma_k(\alpha_k+p_k)}{\alpha_k}+\frac{\gamma_k}{\alpha_k}\Big) \ (\text{due to} \  \alpha_k+p_k \leq 1) \\
    &\overset{\eqref{rho-ineq}}{\geq}\frac{2(\rho_{k+1}-\rho_k)}{\rho_{k+1}}\Big((T_k-1)\frac{\gamma_k(\alpha_k+p_k)}{\alpha_k}+\frac{\gamma_k}{\alpha_k}\Big)
    \overset{\eqref{Lk}}{=}\frac{2(\rho_{k+1}-\rho_k)}{\rho_{k+1}}\cL_k,
\end{align*}
where the second equality follows from $\gamma_k=1/(8L_{\rho_k}\alpha_k)$, $\gamma_{k+1}=1/(8L_{\rho_{k+1}}\alpha_{k+1})$, and $T_k=T_{k+1}$, and the fourth equality is due to $\alpha_k=6/(k-k_0+7)$ and $\alpha_{k+1}=6/(k-k_0+8)$.  Hence, \eqref{lkrk} also holds for $k>k_0$. 
\end{proof}

In view of Lemmas~\ref{l-iter2} and \ref{LkRk}, we next derive upper bounds on $F_{{\rho_k}}(\tilde x_k)-F_{{\rho_k}}^*$ and its expectation in the setting where Algorithm~\ref{alg2} employs dynamic
penalty parameters $\{\rho_k\}$. These bounds, together with Lemma~\ref{l-penalty}, will be used to prove Theorem~\ref{thm:finite-sum}.

\begin{lemma}\label{l-exactbound2}
Suppose that Assumptions \ref{a1} and \ref{a3} hold, and that $\{\tilde x_{k}\}$ is generated by Algorithm \ref{alg2} with $\{q_i\}$, $\{\theta_t\}$, $\{\gamma_k\}$, $\{p_k\}$, $\{\alpha_k\}$ given in \eqref{def-para3},  and $\{T_k\}$, $\{\rho_k\}$ given in \eqref{rho1} or   \eqref{rho2}. Let $F_{{\rho_k}}^*$, $F_{{\rho_k}}$, $\cL_k$ and $\mathcal{R}_k$ be defined in \eqref{def-Frho2} and \eqref{Lk} with such  $\{\rho_k\}$, $\{\gamma_k\}$, $\{p_k\}$, $\{\alpha_k\}$, and $\{T_k\}$. Then for all $k \geq 1$, we have
\begin{align}
F_{\rho_{k}}(\tilde x_{k})-F_{\rho_{k}}^*&\leq \mathcal{R}_{k}^{-1}\Big(\frac{1}{40L_{\rho_1}}\big(D_F+\frac{\rho_1}{2}\|[c(x^0)]_+\|^2\big)+\frac{D_\cX^2}{2}+\frac{\Lambda^2}{2}\sum_{i=1}^{k-1}\frac{\gamma_iT_i}{\rho_i} \nn\\
&\qquad\quad\,\,\,+4\Lambda^2\sum_{i=1}^{k-1}\frac{\cL_i(\rho_{i+1}-\rho_{i})}{\rho_{i+1}^2}+2L_{\nabla \barf}D_\cX^2\sum_{i=1}^{k-1}\gamma_i T_i\Big),\label{l8-bound}\\
\bbE[F_{\rho_{k}}(\tilde x_{k})-F_{\rho_{k}}^*]&\leq \mathcal{R}_{k}^{-1}\Big(\frac{1}{40L_{\rho_1}}\big(D_F+\frac{\rho_1}{2}\|[c(x^0)]_+\|^2\big)+\frac{D_\cX^2}{2} +\frac{\Lambda^2}{2}\sum_{i=1}^{k-1}\frac{\gamma_iT_i}{\rho_i}\nn\\
&\qquad\quad\,\,\,+4\Lambda^2\sum_{i=1}^{k-1}\frac{\cL_i(\rho_{i+1}-\rho_{i})}{\rho_{i+1}^2}\Big), \label{expbound}
\end{align}
where $L_{\nabla \barf}$,  $D_{\mathcal{X}}$, $D_F$, and $\Lambda$ are given  in \eqref{Lam} and Assumptions \ref{a1}, respectively. 
\end{lemma}

\begin{proof}
Notice from  \eqref{def-para3} and \eqref{notation} that $1/(8L_{\rho_k}\alpha_k)=\gamma_k \leq 1/(8L_{\nabla \barf}\alpha_k)$, which implies that 
$L_{\rho_k}\alpha_k\gamma_k=1/8$ and $L_{\nabla \barf}\alpha_k\gamma_k \leq 1/8$. By these, one can easily verify that  the assumption of Lemma \ref{l-iter2} holds for $\{q_i\}$, $\{\alpha_k\}$, $\{\gamma_k\}$ and $\{p_k\}$ given in \eqref{def-para3}. Let $x^*$ be an arbitrary optimal solution of problem \eqref{prob2}. It then follows from \eqref{deltabound}, $F_{{\rho_k}}(x^*)=F^*$, and \eqref{l7-exact} with $x=x^*$ that 
\begin{align}
    \frac{\gamma_k}{\alpha_k}(F_{{\rho_k}}(x_t)-F^*)&\leq 
 \frac{\gamma_k}{\alpha_k}(1-\alpha_k-p_k)(F_{{\rho_k}}(x_{t-1})-F^*)+\frac{\gamma_kp_k}{\alpha_k} (F_{{\rho_k}}(\tilde x_{k})-F^*)\nn\\
 &\quad +\frac{1}{2}\big(\|z_{t-1}-x^*\|^2-\|z_t-x^*\|^2\big) +2\gamma_kL_{\nabla \barf}D_\cX^2.\label{l8-1}
\end{align}
Summing up \eqref{l8-1} from $t=1$ to $T_k$, and using the fact that $x_0=\tilde x_{k}$, $z_0=\tilde z_{k}$  and $z_{T_k}=\tilde z_{k+1}$, we have
\begin{align}
\sum_{t=1}^{T_k}\frac{\gamma_k}{\alpha_k}(F_{\rho_k}(x_t)-F^*)
&\leq \frac{\gamma_k}{\alpha_k}(1-\alpha_k-p_k)(F_{\rho_k}(x_{0})-F^*)+ \sum_{t=2}^{T_k}\frac{\gamma_k}{\alpha_k}(1-\alpha_k-p_k)(F_{\rho_k}(x_{t-1})-F^*)\nn\\
&\quad+\frac{T_k\gamma_kp_k}{\alpha_k}(F_{\rho_k}(\tilde x_{k})-F^*)+\frac{1}{2}\big(\|z_{0}-x^*\|^2-\|z_{T_k}-x^*\|^2\big)+2\gamma_k T_k L_{\nabla \barf}D_\cX^2\nn\\
&= \frac{\gamma_k}{\alpha_k}(1-\alpha_k-p_k)(F_{\rho_k}(\tilde x_{k})-F^*)+ \sum_{t=1}^{T_k-1}\frac{\gamma_k}{\alpha_k}(1-\alpha_k-p_k)(F_{\rho_k}(x_{t})-F^*)\nn\\
&\quad+\frac{T_k\gamma_kp_k}{\alpha_k}(F_{\rho_k}(\tilde x_{k})-F^*)
+\frac{1}{2}\big(\|\tilde z_{k}-x^*\|^2-\|\tilde z_{k+1}-x^*\|^2\big)+2\gamma_k T_k L_{\nabla \barf}D_\cX^2.\label{l8-sum}
\end{align}
 By the  definitions of $\theta_t$ and $\cL_k$ in  \eqref{def-para3} and \eqref{Lk}, one can observe that $\cL_k=\sum_{t=1}^{T_k}\theta_t$. Using this {relation}, \eqref{l8-sum}, the definition of $\theta_t$ in \eqref{def-para3}, the convexity of $F_{\rho_k}(\cdot)$, and $\tilde x_{k+1}=\sum_{t=1}^{T_k}(\theta_tx_t)/(\sum_{t=1}^{T_k}\theta_t)$, we have
    \begin{align}
        & \cL_{k}(F_{\rho_k}(\tilde x_{k+1})-F^*)=\Big(\sum_{t=1}^{T_k}\theta_t\Big)(F_{\rho_k}(\tilde x_{k+1})-F^*) \leq\sum_{t=1}^{T_k}\theta_t(F_{\rho_k}(x_t)-F^*) \nn \\
        & \overset{\eqref{def-para3}}=\sum_{t=1}^{T_k}\frac{\gamma_k}{\alpha_k}(F_{\rho_k}(x_t)-F^*)-\sum_{t=1}^{T_k-1}\frac{\gamma_k}{\alpha_k}(1-\alpha_k-p_k)(F_{\rho_k}(x_t)-F^*)\nn \\
        &\overset{\eqref{l8-sum}}{\leq} \Big(\frac{\gamma_k}{\alpha_k}(1-\alpha_k)+(T_k-1)\frac{\gamma_kp_k}{\alpha_k}\Big)(F_{\rho_k}(\tilde x_{k})-F^*)
    +\frac{1}{2}\big(\|\tilde z_{k}-x^*\|^2-\|\tilde z_{k+1}-x^*\|^2\big)+2\gamma_k T_k L_{\nabla \barf}D_\cX^2, \nn \\
&\overset{\eqref{Lk}}{=}\mathcal{R}_{k}(F_{\rho_k}(\tilde x_{k})-F^*)
    +\frac{1}{2}\big(\|\tilde z_{k}-x^*\|^2-\|\tilde z_{k+1}-x^*\|^2\big)+2\gamma_k T_k L_{\nabla \barf}D_\cX^2. 
\label{l8-2}
    \end{align}

Let $x_k^*$ be an arbitrary optimal solution of \eqref{def-Frho2}. It then follows that $F_{\rho_k}^*=F_{\rho_k}(x_k^*)$. Notice from \eqref{def-Frho2} and  \eqref{Lk} that  $F_{\rho_{k+1}}^*\geq F_{\rho_k}^*$ and $\cL_k-\mathcal{R}_k=\gamma_kT_k\geq 0$. Using these  {relations}, \eqref{prob2},  \eqref{def-Frho2}, \eqref{opt-gap}, \eqref{feasibility}, \eqref{Lk}, and \eqref{l8-2}, we obtain that for all $k\geq1$,
\begin{align}
&\frac{1}{2}\|\tilde z_{k}-x^*\|^2-\frac{1}{2}\|\tilde z_{k+1}-x^*\|^2+2\gamma_kT_k L_{\nabla \barf}D_\cX^2 \overset{\eqref{l8-2}}{\geq}\cL_k(F_{\rho_k}(\tilde x_{k+1})-F^*)- \mathcal{R}_k(F_{\rho_k}(\tilde x_{k})-F^*) \nn\\
&=\cL_k\big(F_{\rho_k}(\tilde x_{k+1})-F_{\rho_{k+1}}^*\big)-\mathcal{R}_k\big(F_{\rho_k}(\tilde x_{k})-F_{\rho_{k}}^*\big)+\cL_k\big(F_{\rho_{k+1}}^*-F^*\big)- \mathcal{R}_k\big(F_{\rho_{k}}^*-F^*\big)\nn\\
&\geq\cL_k\big(F_{\rho_k}(\tilde x_{k+1})-F_{\rho_{k+1}}^*\big)-\mathcal{R}_k\big(F_{\rho_k}(\tilde x_{k})-F_{\rho_{k}}^*\big)+(\cL_k-\mathcal{R}_k)\big(F_{\rho_k}^*-F^*\big)\nn\\
&\overset{\eqref{Lk}}{=}\cL_k\big(F_{\rho_k}(\tilde x_{k+1})-F_{\rho_{k+1}}^*\big)-\mathcal{R}_k\big(F_{\rho_k}(\tilde x_{k})-F_{\rho_{k}}^*\big)+\gamma_kT_k\big(F_{\rho_k}^*-F^*\big)\nn\\
&\overset{\eqref{prob2}\eqref{def-Frho2}}{=}\cL_k\big(F_{\rho_k}(\tilde x_{k+1})-F_{\rho_{k+1}}^*\big)-\mathcal{R}_k\big(F_{\rho_k}(\tilde x_{k})-F_{\rho_{k}}^*\big)+\gamma_kT_k\big(F(x_k^*)-F^*\big)+\frac{\gamma_kT_k\rho_k}{2}\|[c(x_k^*)]_+\|^2\nn\\
&\overset{\eqref{opt-gap}}{\geq}\cL_k\big(F_{\rho_k}(\tilde x_{k+1})-F_{\rho_{k+1}}^*\big)-\mathcal{R}_k\big(F_{\rho_k}(\tilde x_{k})-F_{\rho_{k}}^*\big)-\gamma_kT_k \Lambda \|[c(x_k^*)]_+\|+\frac{\gamma_kT_k\rho_k}{2}\|[c(x_k^*)]_+\|^2\nn\\
&\geq\cL_k\big(F_{\rho_k}(\tilde x_{k+1})-F_{\rho_{k+1}}^*\big)-\mathcal{R}_k\big(F_{\rho_k}(\tilde x_{k})-F_{\rho_{k}}^*\big)- \frac{\gamma_kT_k\Lambda^2}{2\rho_k} \ (\text{due to Young's inequality}) \nn\\
&\overset{\eqref{def-Frho2}}{=}\cL_k\big(F_{\rho_{k+1}}(\tilde x_{k+1})-F_{\rho_{k+1}}^*\big)-\mathcal{R}_k\big(F_{\rho_{k}}(\tilde x_{k})-F_{\rho_{k}}^*\big)- \frac{\gamma_kT_k\Lambda^2}{2\rho_k}-\frac{\cL_k(\rho_{k+1}-\rho_{k})}{2}\|[c(\tilde x_{k+1})]_+\|^2\nn\\
&\overset{\eqref{feasibility}}{\geq}\cL_k\big(F_{\rho_{k+1}}(\tilde x_{k+1})-F_{\rho_{k+1}}^*\big)-\mathcal{R}_k\big(F_{\rho_{k}}(\tilde x_{k})-F_{\rho_{k}}^*\big)- \frac{\gamma_kT_k\Lambda^2}{2\rho_k}\nn\\
&\quad\,\,-\frac{\cL_k(\rho_{k+1}-\rho_{k})}{2}\left(\frac{2\Lambda}{\rho_{k+1}}+\sqrt{\frac{2(F_{\rho_{k+1}}(\tilde x_{k+1})-F_{\rho_{k+1}}^*)}{\rho_{k+1}}}\,\right)^2\nn\\
&\geq \cL_k\Big(1-\frac{2(\rho_{k+1}-\rho_{k})}{\rho_{k+1}}\Big)\big(F_{\rho_{k+1}}(\tilde x_{k+1})-F_{\rho_{k+1}}^*\big)-\mathcal{R}_k\big(F_{\rho_{k}}(\tilde x_{k})-F_{\rho_{k}}^*\big)- \frac{\gamma_kT_k\Lambda^2}{2\rho_k}\nn\\
&\quad-\frac{4\cL_k(\rho_{k+1}-\rho_{k})\Lambda^2}{\rho_{k+1}^2}, \label{lem8-ineq1}
\end{align}
where the last inequality is due to $(a+b)^2 \leq 2(a^2+b^2)$ for all $a,b\in\rr$.
It then follows from \eqref{lkrk}, \eqref{lem8-ineq1}, and $F_{\rho_{i+1}}(\tilde x_{i+1})\geq F_{\rho_{i+1}}^*$ that for all $i\geq 1$, 
\begin{align}
&\mathcal{R}_{i+1}(F_{\rho_{i+1}}(\tilde x_{i+1})-F_{\rho_{i+1}}^*)
\overset{\eqref{lkrk}}{\leq}\cL_i\Big(1-\frac{2(\rho_{i+1}-\rho_{i})}{\rho_{i+1}}\Big)(F_{\rho_{i+1}}(\tilde x_{i+1})-F_{\rho_{i+1}}^*)\nn\\
&\overset{\eqref{lem8-ineq1}}{\leq} \mathcal{R}_i(F_{\rho_i}(\tilde x_{i})-F_{\rho_{i}}^*)
    +\frac{1}{2}\big(\|\tilde z_{i}-x^*\|^2-\|\tilde z_{i+1}-x^*\|^2\big)+2\gamma_iT_i L_{\nabla \barf}D_\cX^2+\frac{\gamma_iT_i\Lambda^2}{2\rho_i}+\frac{4\cL_i(\rho_{i+1}-\rho_{i})\Lambda^2}{\rho_{i+1}^2}.\label{l8-3}
\end{align}  
By \eqref{def-para3},  \eqref{rho1}, \eqref{rho2}, \eqref{notation}, and \eqref{Lk}, one can verify that $\mathcal{R}_1=7/(288L_{\rho_1})<1/(40L_{\rho_1})$. Summing up \eqref{l8-3} from $i=1$ to $k-1$, and using Assumption \ref{a1}(i), $\tilde x_1=\tilde z_1=x^0$, and $\mathcal{R}_1<1/(40L_{\rho_1})$, we have
\begin{align*}
\mathcal{R}_{k}(F_{\rho_{k}}(\tilde x_{k})-F_{\rho_{k}}^*)
&\leq \frac{1}{40L_{\rho_1}}(F_{\rho_1}(x^0)-F_{\rho_{1}}^*)+\frac{1}{2}\|x^0-x^*\|^2+2L_{\nabla \barf}D_\cX^2\sum_{i=1}^{k-1}\gamma_i T_i \nn\\
&\quad+\frac{\Lambda^2}{2}\sum_{i=1}^{k-1}\frac{\gamma_iT_i}{\rho_i}+4\Lambda^2\sum_{i=1}^{k-1}\frac{\cL_i(\rho_{i+1}-\rho_{i})}{\rho_{i+1}^2}.
\end{align*}
The conclusion \eqref{l8-bound} then follows from this {inequality} and the fact that  
\[
F_{\rho_1}(x^0)-F_{\rho_{1}}^*\overset{\eqref{def-Frho2}}{\leq} F(x^0)-F(x^*_1)+\frac{\rho_1}{2}\|[c(x^0)]_+\|^2 \leq D_F+\frac{\rho_1}{2}\|[c(x^0)]_+\|^2,
\]
where the last inequality is due to Assumption \ref{a1}(i). In addition, the proof of \eqref{expbound} follows from \eqref{l7-exp} and similar arguments as above, and is thus omitted.
\end{proof}

We are now ready to prove Theorem \ref{thm:finite-sum}.

\begin{proof}[\textbf{Proof of Theorem \ref{thm:finite-sum}}]
Using  \eqref{opt-gap} and \eqref{feasibility} with $\rho=\rho_k$ and $x=\tilde x_k$, we have
\begin{align}
\|[c(\tilde x_k)]_+
\| &\overset{\eqref{feasibility}}{\leq} 
\frac{2\Lambda}{\rho_{k}}+\sqrt{\frac{2(F_{\rho_{k}}(\tilde x_k)-F_{\rho_{k}}^*)}{\rho_{k}}},\label{t2-cons}\\
\bbE[\|[c(\tilde x_k)]_+\|] &  \overset{\eqref{t2-cons}}{\leq} 
\frac{2\Lambda}{\rho_{k}}+\bbE\left[\sqrt{\frac{2(F_{\rho_{k}}(\tilde x_k)-F_{\rho_{k}}^*)}{\rho_{k}}}\,\right] \leq \frac{2\Lambda}{\rho_{k}}+\sqrt{\frac{2\bbE\left[F_{\rho_{k}}(\tilde x_k)-F_{\rho_{k}}^*\right]}{\rho_{k}}}.\label{t2-obj}
\end{align}
Also, notice from \eqref{def-para3} that $\alpha_k+p_k \leq 1$. Moreover, as shown in the proof of Lemma \ref{LkRk}, the sequence $\{\rho_k\}$ defined in \eqref{rho1} or \eqref{rho2} is nondecreasing.

We first prove statement (i) of Theorem \ref{thm:finite-sum}, where $k_0$, $\{T_k\}$,  and $\{\rho_k\}$ are chosen as in \eqref{rho1}. Notice from \eqref{def-para3} and \eqref{rho1} that $k_0=\lfloor (4/3)\log_2 s\rfloor+1$, 
$\alpha_k=6/7$, $p_k=1/7$,  $T_k=\lceil2^{3(k-1)/4}\rceil$, $\rho_k=2^{k/2}$, and $\gamma_k=1/(8(L_{\nabla \barf}+2^{k/2}L_{\nabla c^2})$  for all $1\leq k\leq k_0$.  Using these relations and \eqref{Lk}, we obtain that 
\begin{align}
  & \cL_k\overset{\eqref{Lk}}{=}\frac{\gamma_k}{\alpha_k}+(T_k-1)\frac{\gamma_k(\alpha_k+p_k)}{\alpha_k}=\frac{T_k\gamma_k}{\alpha_k}=\frac{49\times\lceil2^{3(k-1)/4}\rceil}{288(L_{\nabla \barf}+2^{k/2}L_{\nabla c^2})}\leq\frac{49\times2^{k/4}}{288L_{\nabla c^2}} \qquad \forall 1\leq k\leq k_0, \label{libound} \\
& s/2\leq T_{k_0}=\lceil2^{3(k_0-1)/4}\rceil=\lceil2^{3\lfloor (4/3)\log_2 s\rfloor/4}\rceil\leq s, \quad 
2^{k_0/4}=2^{(\lfloor (4/3)\log_2 s\rfloor+1)/4}\leq2^{1/4}s^{1/3}. \label{Tk0-bnd}
\end{align}
By \eqref{def-para3}, \eqref{rho1},  \eqref{Lk}, \eqref{Tk0-bnd}, and $\alpha_k+p_k \leq 1$, one has that for all $k> k_0$, 
\begin{align}
&\mathcal{R}_k\overset{\eqref{Lk}}{=}\frac{\gamma_k}{\alpha_k}(1-\alpha_k)+(T_k-1)\frac{\gamma_kp_k}{\alpha_k}\geq\frac{\gamma_kp_kT_k}{\alpha_k}\overset{\eqref{def-para3}\eqref{rho1}}{=}\frac{T_{k_0}(k-k_0+7)^2}{2016(L_{\nabla \barf}+(3s^{2/3}(k-k_0+7)^{4/3}/32)L_{\nabla c^2})} \nn \\
&\quad \ \, \overset{\eqref{Tk0-bnd}}{\geq}\frac{s(k-k_0+7)^2}{4032(L_{\nabla \barf}+(3s^{2/3}(k-k_0+7)^{4/3}/32)L_{\nabla c^2})}, \label{Rk-lowbnd} \\
& \cL_k\overset{\eqref{Lk}}{=}\frac{\gamma_k}{\alpha_k}+(T_k-1)\frac{\gamma_k(\alpha_k+p_k)}{\alpha_k}\leq\frac{\gamma_kT_k}{\alpha_k} \overset{\eqref{def-para3}\eqref{rho1}}{=}\frac{T_{k_0}(k-k_0+7)^2}{288(L_{\nabla \barf}+(3s^{2/3}(k-k_0+7)^{4/3}/32)L_{\nabla c^2})} \nn \\
&\quad \ \, \overset{\eqref{Tk0-bnd}}{\leq}\frac{s(k-k_0+7)^2}{288(L_{\nabla \barf}+(3s^{2/3}(k-k_0+7)^{4/3}/32)L_{\nabla c^2})}  \leq\frac{s^{1/3}(k-k_0+7)^{2/3}}{27L_{\nabla c^2}}. \label{sub-ineq2}
\end{align}
In addition, observe that $\lceil2^{3(k-1)/4}\rceil\leq2^{3k/4}$ for all $k\geq 1$. Using this relation, $s\geq1$, \eqref{def-para3},  \eqref{rho1}, and \eqref{Tk0-bnd}, we obtain that for all $k\geq 1$, 
\begin{align}
&\sum_{i=1}^{k-1}\gamma_i T_i=\sum_{i=1}^{k_0}\gamma_i T_i+\sum_{i=k_0+1}^{k-1}\gamma_i T_i
\overset{\eqref{def-para3}\eqref{rho1}}{=}\sum_{i=1}^{k_0}\frac{7T_i}{48(L_{\nabla \barf}+\rho_i L_{\nabla c^2})}+\sum_{i=k_0+1}^{k-1}\frac{T_{k_0}}{8(L_{\nabla \barf}+\rho_i L_{\nabla c^2})\alpha_i} \nn\\
&\qquad\quad \overset{\eqref{Tk0-bnd}}{\leq}\frac{7}{48L_{\nabla c^2}}\sum_{i=1}^{k_0}\frac{T_i}{\rho_i}+\frac{s}{8L_{\nabla c^2}}\sum_{i=k_0+1}^{k-1}\frac{1}{\rho_i\alpha_i} \nn \\
&\qquad\quad\overset{\eqref{def-para3}\eqref{rho1}}{=} \frac{7}{48L_{\nabla c^2}}\sum_{i=1}^{k_0}\frac{\lceil2^{3(i-1)/4}\rceil}{2^{i/2}}+\frac{2s^{1/3}}{9L_{\nabla c^2}}\sum_{i=k_0+1}^{k-1}(i-k_0+7)^{-1/3}  \nn \\ 
&\qquad\quad\ \    \leq\frac{7}{48L_{\nabla c^2}}\sum_{i=1}^{k_0}\frac{2^{3i/4}}{2^{i/2}}+\frac{2s^{1/3}}{9L_{\nabla c^2}}  \int_{k_0+1}^k (\tau-k_0+6)^{-1/3} d\tau \nn\\
&\qquad\quad\ \ \leq\frac{7\times 2^{k_0/4}}{48(1-2^{-1/4})L_{\nabla c^2}}+\frac{s^{1/3}(k-k_0+7)^{2/3}}{3L_{\nabla c^2}} 
  \overset{\eqref{Tk0-bnd}}{\leq} \frac{11s^{1/3}}{10L_{\nabla c^2}}+\frac{s^{1/3}(k-k_0+7)^{2/3}}{3L_{\nabla c^2}}, 
 \label{sum-gt}\\
&\sum_{i=1}^{k-1}\frac{\gamma_iT_i}{\rho_i}=\sum_{i=1}^{k_0}\frac{\gamma_iT_i}{\rho_i}+\sum_{i=k_0+1}^{k-1}\frac{\gamma_iT_i}{\rho_i}
\overset{\eqref{def-para3}\eqref{rho1}}{=}\sum_{i=1}^{k_0}\frac{7T_i}{48(L_{\nabla \barf}+\rho_i L_{\nabla c^2})\rho_i}+\sum_{i=k_0+1}^{k-1}\frac{T_{k_0}}{8(L_{\nabla \barf}+\rho_i L_{\nabla c^2})\rho_i\alpha_i} \nn\\
&\qquad\quad  \overset{\eqref{Tk0-bnd}}{\leq}\frac{7}{48L_{\nabla c^2}}\sum_{i=1}^{k_0}\frac{T_i}{\rho_i^2}+\frac{s}{8L_{\nabla c^2}}\sum_{i=k_0+1}^{k-1}\frac{1}{\rho_i^2\alpha_i} \nn\\
&\qquad\quad \overset{\eqref{def-para3}\eqref{rho1}}{=}\frac{7}{48L_{\nabla c^2}}\sum_{i=1}^{k_0}\frac{\lceil2^{3(i-1)/4}\rceil}{2^{i}}+ \frac{64s^{-1/3}}{27L_{\nabla c^2}}\sum_{i=k_0+1}^{k-1}(i-k_0+7)^{-5/3}\nn \\
&\qquad\quad  \leq \frac{7}{48L_{\nabla c^2}}\sum_{i=1}^{k_0}\frac{2^{3i/4}}{2^{i}}+\frac{64s^{-1/3}}{27L_{\nabla c^2}} \int^k_{k_0+1} (\tau-k_0+6)^{-5/3}d\tau \nn \\
&\qquad\quad \leq \frac{7}{48(2^{1/4}-1)L_{\nabla c^2}}+\frac{32}{27L_{\nabla c^2}} 
 \leq\frac{2}{L_{\nabla c^2}}.  \label{sum-gtr}
 \end{align}
Observe from \eqref{rho1} that $2^{k_0/2}=2^{(\lfloor(4/3)\log_2s\rfloor+1)/2}\geq s^{2/3}$. 
By this {relation}, \eqref{rho1},  and the convexity of the functions $2^{\tau/2}$ and $\tau^{4/3}$, one has
\begin{align} \label{rho-diff}
    \rho_{i+1}-\rho_i=\left\{\begin{array}{ll}
       2^{(i+1)/2}-2^{i/2}\leq (2^{(i+1)/2}\log2)/2=(\rho_{i+1}\log2)/2  & \text{if} \  \ 1\leq i< k_0, \\
       3s^{2/3}/2-2^{k_0/2}\leq s^{2/3}/2=\rho_{i+1}/3\leq(\rho_{i+1}\log2)/2 & \text{if} \  \  i= k_0,\\
       3s^{2/3}((i-k_0+8)^{4/3}-(i-k_0+7)^{4/3})/32 \leq s^{2/3}(i-k_0+8)^{1/3}/8  & \text{if} \  \ i> k_0.
       \end{array}\right. 
\end{align}
Using \eqref{rho1}, \eqref{libound}, \eqref{sub-ineq2}, \eqref{rho-diff}, $s\geq1$, and $\rho_{i+1} \geq \rho_i$, we obtain that 
\begin{align}
& \sum_{i=1}^{k-1}\frac{\cL_i(\rho_{i+1}-\rho_{i})}{\rho_{i+1}^2} =\sum_{i=1}^{k_0}\frac{\cL_i(\rho_{i+1}-\rho_{i})}{\rho_{i+1}^2}+\sum_{i=k_0+1}^{k-1}\frac{\cL_i(\rho_{i+1}-\rho_{i})}{\rho_{i+1}^2} \nn \\
&\overset{\eqref{rho-diff}}{\leq}\sum_{i=1}^{k_0}\frac{\cL_i\log2}{2\rho_{i+1}}+\sum_{i=k_0+1}^{k-1}\frac{s^{2/3}\cL_i(i-k_0+8)^{1/3}}{8\rho_{i+1}^2}\nn\\
& \leq\sum_{i=1}^{k_0}\frac{\cL_i\log2}{2\rho_{i}}+\sum_{i=k_0+1}^{k-1}\frac{s^{2/3}\cL_i(i-k_0+8)^{1/3}}{8\rho_{i+1}^2} \  \  \ (\text{due to} \ \rho_{i+1} \geq \rho_i) \nn\\
&\leq \frac{49\log2}{576L_{\nabla c^2}}\sum_{i=1}^{k_0}\frac{2^{i/4}}{2^{i/2}}+ \frac{128s^{-1/3}}{243L_{\nabla c^2}}\sum_{i=k_0+1}^{k-1}\frac{(i-k_0+7)^{2/3}(i-k_0+8)^{1/3}}{(i-k_0+8)^{8/3}} \  \  \  (\text{due to} \  \eqref{rho1}, \eqref{libound}, \eqref{sub-ineq2}) \nn \\
& \leq \frac{49\log2}{576(2^{1/4}-1)L_{\nabla c^2}}+ \frac{128}{243L_{\nabla c^2}}\sum_{i=k_0+1}^{k-1}(i-k_0+8)^{-5/3} \nn \\
&{\leq\frac{49\log2}{576(2^{1/4}-1)L_{\nabla c^2}}+\frac{16}{81L_{\nabla c^2}}\leq\frac{1}{L_{\nabla c^2}}.} \nn
\end{align}
It follows from this {inequality}, \eqref{c3-5}, \eqref{notation}, \eqref{l8-bound}, \eqref{Rk-lowbnd}, \eqref{sum-gt}, \eqref{sum-gtr}, $\rho_1=\sqrt{2}$, and $s\geq 1$ that for all $k>k_0$, 
\begin{align}
&F_{\rho_{k}}(\tilde x_{k})-F_{\rho_{k}}^*\overset{\eqref{l8-bound}}{\leq} \mathcal{R}_{k}^{-1}\Big(\frac{D_F+\rho_1\|[c(x^0)]_+\|^2/2}{40L_{\rho_1}}+\frac{D_\cX^2}{2}+\frac{\Lambda^2}{2}\sum_{i=1}^{k-1}\frac{\gamma_iT_i}{\rho_i} +4\Lambda^2\sum_{i=1}^{k-1}\frac{\cL_i(\rho_{i+1}-\rho_{i})}{\rho_{i+1}^2} \nn \\
&\qquad\qquad\qquad\qquad\qquad\, +2L_{\nabla \barf}D_\cX^2\sum_{i=1}^{k-1}\gamma_i T_i\Big)\nn\\
&\leq \frac{4032L_{\nabla \barf}+378L_{\nabla c^2}s^{2/3}(k-k_0+7)^{4/3}}{s(k-k_0+7)^2}\Big(\frac{D_F+\sqrt{2}\|[c(x^0)]_+\|^2/2}{40(L_{\nabla \barf}+\sqrt{2}L_{\nabla c^2})}+\frac{D_\cX^2}{2}+\frac{\Lambda^2}{L_{\nabla c^2}}\nn\\
&\quad +\frac{4\Lambda^2}{L_{\nabla c^2}}+\frac{11L_{\nabla \barf}D_\cX^2s^{1/3}}{5L_{\nabla c^2}}+\frac{2L_{\nabla \barf}D_\cX^2s^{1/3}(k-k_0+7)^{2/3}}{3L_{\nabla c^2}}\Big)\nn\\
&\overset{\eqref{c3-5}}{\leq}\frac{\tC_5s^{-1}+\tC_7s^{-2/3}}{(k-k_0+7)^{2}}+\frac{\tC_6s^{-1/3}+\tC_8}{(k-k_0+7)^{2/3}}+\frac{\tC_9s^{-2/3}}{(k-k_0+7)^{4/3}}+252L_{\nabla \barf}D_\cX^2\nn\\
&{\leq}\frac{(\tC_5+\tC_7)s^{-2/3}}{(k-k_0+7)^{2}}+\frac{\tC_6+\tC_8}{(k-k_0+7)^{2/3}}+\frac{\tC_9s^{-2/3}}{(k-k_0+7)^{4/3}}+252L_{\nabla \barf}D_\cX^2,\label{t2-bound3}
\end{align}
where $\tC_5$, $\tC_6$, $\tC_7$, $\tC_8$, and $\tC_9$ are given in \eqref{c3-5}. Similarly, by \eqref{c3-5} and \eqref{expbound}, one has
\beq\label{t2-bound4}
\bbE[F_{\rho_{k}}(\tilde x_{k})-F_{\rho_{k}}^*]\leq\frac{\tC_5s^{-1}}{(k-k_0+7)^{2}}+\frac{\tC_6s^{-1/3}}{(k-k_0+7)^{2/3}}.
\eeq
Further, using \eqref{expbound}, \eqref{t2-cons}, \eqref{t2-obj},  and \eqref{t2-bound3}, we obtain that
\begin{align*}
\|[c(\tilde x_k)]_+\| & \overset{\eqref{t2-cons}\eqref{t2-bound3}}{\leq} \frac{2\Lambda}{\rho_k}+\sqrt{\frac{2(\tC_5+\tC_7)s^{-2/3}}{(k-k_0+7)^{2}\rho_k}+\frac{2(\tC_6+\tC_8)}{(k-k_0+7)^{2/3}\rho_k}+\frac{2\tC_9s^{-2/3}}{(k-k_0+7)^{4/3}\rho_k}+\frac{504L_{\nabla \barf}D_{\cX}^2}{\rho_k}}\nn\\
&\ \  \leq\frac{2\Lambda}{\rho_k}+\sqrt{\frac{2(\tC_5+\tC_7)s^{-2/3}}{(k-k_0+7)^{2}\rho_k}}+\sqrt{\frac{2(\tC_6+\tC_8)}{(k-k_0+7)^{2/3}\rho_k}}+\sqrt{\frac{2\tC_9s^{-2/3}}{(k-k_0+7)^{4/3}\rho_k}}+\sqrt{\frac{504L_{\nabla \barf}D_{\cX}^2}{\rho_k}}, \\
    \bbE[\|[c(\tilde x_k)]_+\|] & \overset{\eqref{t2-obj}\eqref{t2-bound4}} \leq\frac{2\Lambda}{\rho_k}+\sqrt{\frac{2\tC_5s^{-1}}{(k-k_0+7)^{2}\rho_k}+\frac{2\tC_6s^{-1/3}}{(k-k_0+7)^{2/3}\rho_k}} \\
    &\quad  \leq \frac{2\Lambda}{\rho_k}+\sqrt{\frac{2\tC_5s^{-1}}{(k-k_0+7)^{2}\rho_k}}+\sqrt{\frac{2\tC_6s^{-1/3}}{(k-k_0+7)^{2/3}\rho_k}},
\end{align*}
which, together with $k\geq2(k_0-7)$, $\rho_k=3s^{2/3}(k-k_0+7)^{4/3}/32$ and $\bbE [F(\tilde x_k)-F^*] \geq -\Lambda\,\bbE[\|[c(\tilde x_k)]_+\|]$ due to \eqref{opt-gap}, implies that \eqref{thm3-ineq1}, \eqref{thm3-ineq1-exp}, and \eqref{thm3-ineq3} hold. In addition,  \eqref{thm3-ineq2} follows from $k\geq2(k_0-7)$, \eqref{opt-gap}, and \eqref{t2-bound4}.

We next prove statement (ii) of Theorem \ref{thm:finite-sum}, where $k_0$, $\{T_k\}$,  and $\{\rho_k\}$ are chosen as in \eqref{rho2}.  Notice from \eqref{def-para3} and \eqref{rho2} that $k_0=\lfloor \log_2 s\rfloor+1$, 
$\alpha_k=6/7$, $p_k=1/7$,  $T_k=2^{k-1}$, $\rho_k=2^{k/2}$, and $\gamma_k=1/(8(L_{\nabla \barf}+2^{k/2}L_{\nabla c^2})$  for all $1\leq k\leq k_0$.  Using these relations and \eqref{Lk}, we obtain that 
\begin{align}
  & \cL_k\overset{\eqref{Lk}}{=}\frac{\gamma_k}{\alpha_k}+(T_k-1)\frac{\gamma_k(\alpha_k+p_k)}{\alpha_k}=\frac{T_k\gamma_k}{\alpha_k}=\frac{49\times2^{k-1}}{288(L_{\nabla \barf}+2^{k/2}L_{\nabla c^2})}\leq\frac{49\times2^{k/2}}{576L_{\nabla c^2}} \qquad \forall 1\leq k\leq k_0, \label{libound-1} \\
& s/2\leq T_{k_0}=2^{k_0-1}=2^{\lfloor \log_2 s\rfloor}\leq s, \quad 2^{k_0/2}=2^{(\lfloor \log_2 s\rfloor+1)/2}\leq\sqrt{2s}. \label{Tk0-bnd-1}
\end{align}
By \eqref{def-para3}, \eqref{rho2},  \eqref{Lk}, \eqref{Tk0-bnd-1}, and $\alpha_k+p_k \leq 1$, one has that for all $k> k_0$, 
\begin{align}
&\mathcal{R}_k\overset{\eqref{Lk}}{=}\frac{\gamma_k}{\alpha_k}(1-\alpha_k)+(T_k-1)\frac{\gamma_kp_k}{\alpha_k}\geq\frac{\gamma_kp_kT_k}{\alpha_k}\overset{\eqref{def-para3}\eqref{rho2}}{=}\frac{T_{k_0}(k-k_0+7)^2}{2016(L_{\nabla \barf}+(3\sqrt{s}(k-k_0+7)/16)L_{\nabla c^2})} \nn \\
&\quad \ \, \overset{\eqref{Tk0-bnd-1}}{\geq}\frac{s(k-k_0+7)^2}{4032(L_{\nabla \barf}+(3\sqrt{s}(k-k_0+7)/16)L_{\nabla c^2})}, \label{Rk-lowbnd-1} \\
& \cL_k\overset{\eqref{Lk}}{=}\frac{\gamma_k}{\alpha_k}+(T_k-1)\frac{\gamma_k(\alpha_k+p_k)}{\alpha_k}\leq\frac{\gamma_kT_k}{\alpha_k}\overset{\eqref{def-para3}\eqref{rho2}}{\leq}\frac{s(k-k_0+7)^2}{288(L_{\nabla \barf}+(3\sqrt{s}(k-k_0+7)/16)L_{\nabla c^2})} \nn\\
&\quad \ \,   \overset{\eqref{Tk0-bnd-1}}{\leq}\frac{\sqrt{s}(k-k_0+7)}{54L_{\nabla c^2}}. \label{sub-ineq3}
\end{align}
Using \eqref{def-para3},  \eqref{rho2}, and \eqref{Tk0-bnd-1}, we obtain that for all $k\geq 1$, 
\begin{align}
&\sum_{i=1}^{k-1}\gamma_i T_i=\sum_{i=1}^{k_0}\gamma_i T_i+\sum_{i=k_0+1}^{k-1}\gamma_i T_i
\overset{\eqref{def-para3}\eqref{rho2}}{=}\sum_{i=1}^{k_0}\frac{7T_i}{48(L_{\nabla \barf}+\rho_i L_{\nabla c^2})}+\sum_{i=k_0+1}^{k-1}\frac{T_{k_0}}{8(L_{\nabla \barf}+\rho_i L_{\nabla c^2})\alpha_i} \nn\\
&\qquad\quad\ \ \overset{\eqref{Tk0-bnd-1}}{\leq}\frac{7}{48L_{\nabla c^2}}\sum_{i=1}^{k_0}\frac{T_i}{\rho_i}+\frac{s}{8L_{\nabla c^2}}\sum_{i=k_0+1}^{k-1}\frac{1}{\rho_i\alpha_i}
\overset{\eqref{def-para3}\eqref{rho2}}{=} \frac{7}{48L_{\nabla c^2}}\sum_{i=1}^{k_0}\frac{2^{i-1}}{2^{i/2}}+\frac{\sqrt{s}}{9L_{\nabla c^2}}\sum_{i=k_0+1}^{k-1}1  \nn \\ 
&\qquad\quad\ \    \leq\frac{7\times 2^{k_0/2}}{48(2-\sqrt{2})L_{\nabla c^2}}+\frac{\sqrt{s}(k-k_0-1)}{9L_{\nabla c^2}}   
  \overset{\eqref{Tk0-bnd-1}}{\leq}\frac{7\sqrt{s}}{48(\sqrt{2}-1)L_{\nabla c^2}}+\frac{\sqrt{s}(k-k_0+7)}{9L_{\nabla c^2}}, \label{sum-gt-1} \\
&\sum_{i=1}^{k-1}\frac{\gamma_iT_i}{\rho_i}=\sum_{i=1}^{k_0}\frac{\gamma_iT_i}{\rho_i}+\sum_{i=k_0+1}^{k-1}\frac{\gamma_iT_i}{\rho_i}
\overset{\eqref{def-para3}\eqref{rho2}}{=}\sum_{i=1}^{k_0}\frac{7T_i}{48(L_{\nabla \barf}+\rho_i L_{\nabla c^2})\rho_i}+\sum_{i=k_0+1}^{k-1}\frac{T_{k_0}}{8(L_{\nabla \barf}+\rho_i L_{\nabla c^2})\rho_i\alpha_i} \nn\\
&\qquad\qquad \overset{\eqref{Tk0-bnd-1}}{\leq}\frac{7}{48L_{\nabla c^2}}\sum_{i=1}^{k_0}\frac{T_i}{\rho_i^2}+\frac{s}{8L_{\nabla c^2}}\sum_{i=k_0+1}^{k-1}\frac{1}{\rho_i^2\alpha_i} 
\overset{\eqref{def-para3}\eqref{rho2}}{=}\frac{7}{48L_{\nabla c^2}}\sum_{i=1}^{k_0}\frac{2^{i-1}}{2^{i}}+ \frac{16}{27L_{\nabla c^2}}\sum_{i=k_0+1}^{k-1}(i-k_0+7)^{-1}\nn \\
&\qquad\qquad  \leq \frac{7k_0}{96L_{\nabla c^2}}+\frac{16}{27L_{\nabla c^2}} \int^k_{1} \tau^{-1}d\tau
\leq\frac{k_0}{12L_{\nabla c^2}}+\frac{3\log k}{5L_{\nabla c^2}}. \label{sum-gtr-1} 
\end{align}
Observe from \eqref{rho2} that $2^{k_0/2}=2^{(\lfloor\log_2s\rfloor+1)/2}\geq \sqrt{s}$. 
By this {relation}, \eqref{rho2},  and the convexity of the function $2^{\tau/2}$, one has
\begin{align} \label{rho-diff-1}
    \rho_{i+1}-\rho_i=\left\{\begin{array}{ll}
       2^{(i+1)/2}-2^{i/2}\leq (2^{(i+1)/2}\log2)/2\leq(\rho_{i+1}\log2)/2  & \text{if} \  \ 1\leq i< k_0, \\
       3\sqrt{s}/2-2^{k_0/2}\leq\sqrt{s}/2=\rho_{i+1}/3\leq(\rho_{i+1}\log2)/2 & \text{if} \  \  i= k_0,\\
       3\sqrt{s}((i-k_0+8)-(i-k_0+7))/16 = 3\sqrt{s}/16  & \text{if} \  \ i> k_0.
       \end{array}\right. 
\end{align}
Using \eqref{rho2}, \eqref{libound-1}, \eqref{sub-ineq3},  \eqref{rho-diff-1}, and $\rho_{i+1} \geq \rho_i$, we obtain that 
\begin{align}
&\sum_{i=1}^{k-1}\frac{\cL_i(\rho_{i+1}-\rho_{i})}{\rho_{i+1}^2}=\sum_{i=1}^{k_0}\frac{\cL_i(\rho_{i+1}-\rho_{i})}{\rho_{i+1}^2}+\sum_{i=k_0+1}^{k-1}\frac{\cL_i(\rho_{i+1}-\rho_{i})}{\rho_{i+1}^2}
\overset{\eqref{rho-diff-1}}{\leq}\sum_{i=1}^{k_0}\frac{\cL_i\log2}{2\rho_{i+1}}+\sum_{i=k_0+1}^{k-1}\frac{3\sqrt{s}\cL_i}{16\rho_{i+1}^2}\nn\\
&\qquad\qquad\qquad\quad\ \leq\sum_{i=1}^{k_0}\frac{\cL_i\log2}{2\rho_{i}}+\sum_{i=k_0+1}^{k-1}\frac{3\sqrt{s}\cL_i}{16\rho_{i+1}^2} \  \  \ (\text{due to} \ \rho_{i+1} \geq \rho_i) \nn\\
&\qquad\qquad\qquad\quad\
\leq \frac{49\log2}{1152L_{\nabla c^2}}\sum_{i=1}^{k_0}1 +\frac{8}{81L_{\nabla c^2}}\sum_{i=k_0+1}^{k-1}\frac{i-k_0+7}{(i-k_0+8)^{2}}\  \  \  (\text{due to} \  \eqref{rho2}, \eqref{libound-1}, \eqref{sub-ineq3}) \nn \\
&\qquad\qquad\qquad\quad\ \leq \frac{49k_0\log2}{1152L_{\nabla c^2}}+ \frac{8}{81L_{\nabla c^2}}\sum_{i=k_0+1}^{k-1}{(i-k_0+7)^{-1}}
\leq\frac{k_0}{24L_{\nabla c^2}}+\frac{\log k}{10L_{\nabla c^2}}. \nn
\end{align}
By this inequality, \eqref{c6-9}, \eqref{notation}, \eqref{l8-bound}, \eqref{Rk-lowbnd-1}, \eqref{sum-gt-1}, \eqref{sum-gtr-1}, $\rho_1=\sqrt{2}$, and $s\geq1$, one has that for all $k>k_0$, 
\begin{align}
& F_{\rho_{k}}(\tilde x_{k})-F_{\rho_{k}}^*\overset{\eqref{l8-bound}}{\leq} \mathcal{R}_{k}^{-1}\Big(\frac{D_F+\rho_1\|[c(x^0)]_+\|^2/2}{40L_{\rho_1}}+\frac{D_\cX^2}{2}+\frac{\Lambda^2}{2}\sum_{i=1}^{k-1}\frac{\gamma_iT_i}{\rho_i} +4\Lambda^2\sum_{i=1}^{k-1}\frac{\cL_i(\rho_{i+1}-\rho_{i})}{\rho_{i+1}^2} \nn \\ 
&\qquad\qquad\qquad\qquad\qquad\,+2L_{\nabla \barf}D_\cX^2\sum_{i=1}^{k-1}\gamma_i T_i\Big)\nn\\
&\leq \frac{4032L_{\nabla \barf}+756\sqrt{s}(k-k_0+7)L_{\nabla c^2}}{s(k-k_0+7)^2}\Big(\frac{D_F+\sqrt{2}\|[c(x^0)]_+\|^2/2}{40(L_{\nabla \barf}+\sqrt{2}L_{\nabla c^2})}+\frac{D_\cX^2}{2}+\frac{\Lambda^2k_0}{24L_{\nabla c^2}}\nn\\
&\quad
+\frac{3\Lambda^2\log k}{10L_{\nabla c^2}}+\frac{\Lambda^2k_0}{6L_{\nabla c^2}}+\frac{2\Lambda^2\log k}{5L_{\nabla c^2}}+\frac{7L_{\nabla \barf}D_\cX^2\sqrt{s}}{24(\sqrt{2}-1)L_{\nabla c^2}}+\frac{2L_{\nabla \barf}D_\cX^2\sqrt{s}(k-k_0+7)}{9L_{\nabla c^2}}\Big)\nn\\
&\overset{ \eqref{c6-9}}{\leq} \frac{\tC_{10}s^{-1}+\tC_{13}s^{-1/2}}{(k-k_0+7)^2}+\frac{\tC_{11}s^{-1/2}+\tC_{14}+\tC_{15}s^{-1/2}}{k-k_0+7}+\frac{\tC_{12}s^{-1}\log k}{(k-k_0+7)^2}+\frac{523\Lambda^2s^{-1/2}\log k}{k-k_0+7}\nn\\
&\quad+168L_{\nabla \barf}D_\cX^2
\nn\\
&\leq \frac{(\tC_{10}+\tC_{13})s^{-1/2}}{(k-k_0+7)^2}+\frac{\tC_{11}+\tC_{14}+\tC_{15}}{k-k_0+7}+\frac{\tC_{12}s^{-1}\log k}{(k-k_0+7)^2}+\frac{523\Lambda^2s^{-1/2}\log k}{k-k_0+7}\nn\\
&\quad+168L_{\nabla \barf}D_\cX^2,\label{t2-bound4.5}
\end{align}
where $\tC_{10}$, $\tC_{11}$, $\tC_{12}$, $\tC_{13}$, $\tC_{14}$, and $\tC_{15}$ are given in \eqref{c6-9}. Similarly, by \eqref{c6-9} and \eqref{expbound}, one has
\beq\label{t2-bound5}
\bbE[F_{\rho_{k}}(\tilde x_{k})-F_{\rho_{k}}^*]\leq \frac{\tC_{10}s^{-1}}{(k-k_0+7)^2}+\frac{\tC_{11}s^{-1/2}}{k-k_0+7}+\frac{\tC_{12}s^{-1}\log k}{(k-k_0+7)^2}+\frac{523\Lambda^2s^{-1/2}\log k}{k-k_0+7}. 
\eeq
 Using \eqref{t2-cons}, \eqref{t2-obj}, \eqref{t2-bound4.5}, and \eqref{t2-bound5}, we have
\begin{align*}
&\|[c(\tilde x_k)]_+\| \overset{\eqref{t2-cons}\eqref{t2-bound4.5}}{\leq} \frac{2\Lambda}{\rho_k}+\sqrt{\frac{2(\tC_{10}+\tC_{13})s^{-1/2}}{(k-k_0+7)^{2}\rho_k}}+\sqrt{\frac{2(\tC_{11}+\tC_{14}+\tC_{15})}{(k-k_0+7)\rho_k}} \nn \\ 
& \qquad\qquad\qquad\quad\ +\sqrt{\frac{2\tC_{12}s^{-1}\log k}{(k-k_0+7)^2\rho_k}}+\sqrt{\frac{1046\Lambda^2s^{-1/2}\log k}{(k-k_0+7)\rho_k}} +\sqrt{\frac{336L_{\nabla \barf}D_{\cX}^2}{\rho_k}}, \\
&\bbE[\|[c(\tilde x_k)]_+\|]\overset{\eqref{t2-obj}\eqref{t2-bound5}}{\leq} \frac{2\Lambda}{\rho_k}+\sqrt{\frac{2\tC_{10}s^{-1}}{(k-k_0+7)^2\rho_k}}+\sqrt{\frac{2\tC_{11}s^{-1/2}}{(k-k_0+7)\rho_k}}+\sqrt{\frac{2\tC_{12}s^{-1}\log k}{(k-k_0+7)^2\rho_k}}\\
&\qquad\qquad\qquad\qquad\ +\sqrt{\frac{1046\Lambda^2s^{-1/2}\log k}{(k-k_0+7)\rho_k}}, 
\end{align*}
which, together with \eqref{opt-gap}, $\rho_k=3\sqrt{s}(k-k_0+7)/16$, and the relation $k-k_0+7\geq k/2$ for all $k\geq \max\{k_0+1,2(k_0-7)\}$,  imply that \eqref{thm3-ineq4-1}, \eqref{thm3-ineq4}, and \eqref{thm3-ineq6} hold. In addition, \eqref{thm3-ineq5} follows from $k \geq 2(k_0-7)$, \eqref{opt-gap}, and \eqref{t2-bound5}. 
\end{proof}

\section{Concluding remarks}\label{sec:conclude}

In this paper, we proposed first-order methods for computing an $\epsilon$-\emph{surely feasible} stochastic optimal ($\epsilon$-SFSO) solution for both stochastic and finite-sum convex optimization problems with deterministic constraints, and established first-order oracle (FO) complexity results for these methods. 

Although we focused on problem \eqref{prob1} with convex inequality constraints, our algorithms and results can be directly extended to problems that include both affine equality constraints and convex inequality constraints. Furthermore, they can also be extended to problems with convex conic constraints of the form $c(x) \in -\mathcal{K}$, where $\mathcal{K}$ is a closed convex cone whose projection operator can be evaluated exactly, and $c$ is both Lipschitz smooth and $\mathcal{K}$-convex; that is, $\alpha c(x) + (1-\alpha) c(y) - c(\alpha x + (1 - \alpha) y) \in \mathcal{K}$ for all $x, y \in \mathbb{R}^n$ and $\alpha \in [0,1]$.

 An interesting direction for future research is to develop a stochastic first-order method for solving finite-sum convex optimization problems with deterministic constraints that achieves an FO complexity of $\mathcal{O}(s \log s + \sqrt{s} \epsilon^{-1})$ for computing an $\epsilon$-SFSO solution. This would significantly improve upon the current complexity of $\mathcal{O}(s \log s + \sqrt{s} \epsilon^{-3/2})$ achieved by our proposed method, and would match the optimal complexity known for computing an $\epsilon$-\emph{expected feasible} stochastic optimal solution.

\end{document}